\documentclass[11pt, a4paper]{article}
\usepackage{amssymb,amsmath,amsthm,amscd,amsbsy}
\usepackage[greek,english]{babel}

\usepackage{verbatim}
\usepackage{a4wide}

\usepackage[active]{srcltx}
\usepackage{enumerate}
\usepackage{mathtools}
\usepackage{yhmath}

\usepackage[pdfauthor={Anne-Laure Dalibard and Christophe Prange},pdftitle={Well-posedness of the Stokes-Coriolis system in a rough half-space},pdftex]{hyperref}

\usepackage[latin1]{inputenc}
\usepackage[T1]{fontenc}

\usepackage{lipsum}
\usepackage{comment}

\usepackage{color}

\usepackage{mathrsfs}
\usepackage{stmaryrd}
\usepackage{esint}
\usepackage{subfig}
\usepackage{tabularx}
\usepackage{calc}
\usepackage{cancel}
\usepackage{yhmath}
\usepackage[pdftex]{graphicx}
\usepackage[all]{xy}
\xyoption{v2}
\xyoption{2cell}
\UseAllTwocells

\usepackage{calc}
\usepackage{graphicx}

\makeatletter
\newlength{\temp@wc@width}
\newlength{\temp@wc@height}
\newcommand{\widecheck}[1]{%
  \setlength{\temp@wc@width}{\widthof{$#1$}}%
  \setlength{\temp@wc@height}{\heightof{$#1$}}%
  #1\hspace{-\temp@wc@width}%
  \raisebox{\temp@wc@height+2pt}[\heightof{$\widehat{#1}$}]%
     {\rotatebox[origin=c]{180}{\vbox to 0pt{\hbox{$\widehat{\hphantom{#1}}$}}}}%
}
\makeatother

\newtheorem{theorem}{Theorem}
\newtheorem{lemma}{Lemma}[section]
\newtheorem{proposition}[lemma]{Proposition}

\newtheorem{definition}[lemma]{Definition}

\newtheorem{remark}[lemma]{Remark}
\newtheorem{corollary}[lemma]{Corollary}

\numberwithin{equation}{section}

\newcommand{\Om}{\Omega}
\newcommand{\be}{\begin{equation}}
\newcommand{\ee}{\end{equation}}
\newcommand{\bestar}{\begin{equation*}}
\newcommand{\eestar}{\end{equation*}}

\newcommand{\R}{\mathbb R}
\newcommand{\N}{\mathbb N}
\newcommand{\Z}{\mathbb Z}

\newcommand{\C}{\mathbb C}
\newcommand{\bK}{\mathbb K}
\newcommand{\Omb}{\Omega^b}
\newcommand{\Kc}{\mathcal{K}}

\newcommand{\hv}{\hat v}
\newcommand{\pa}{\partial}
\newcommand{\na}{\nabla}
\newcommand{\dv}{\mathrm{div}\;}
\newcommand{\Supp}{\operatorname{Supp}}

\newcommand{\tu}{\tilde{u}}
\newcommand{\tp}{\tilde p}

\newcommand{\om}{\omega}

\newcommand{\la}{\lambda}

\newcommand{\ds}{\displaystyle}

\DeclareMathOperator{\logg}{ln}
\DeclareMathOperator{\supp}{Supp}
\DeclareMathOperator{\cof}{Cof}
\DeclareMathOperator{\DtoN}{DN}
\DeclareMathOperator{\Aa}{A}
\DeclareMathOperator{\Bb}{B}
\DeclareMathOperator{\Cc}{C}

\numberwithin{equation}{section}

\title{Well-posedness of the Stokes-Coriolis system in the  half-space over a rough surface}

\author{Anne-Laure Dalibard and Christophe Prange}

\begin{document}

\maketitle


\begin{abstract}
This paper is devoted to the well-posedness of the stationary $3$d Stokes-Coriolis system set in a half-space with rough bottom and Dirichlet data which does not decrease at space infinity. Our system is a linearized version of the Ekman boundary layer system. We look for a solution of infinite energy in a space of Sobolev regularity. Following an idea of G\'erard-Varet and Masmoudi, the general strategy is to reduce the problem to a bumpy channel bounded in the vertical direction thanks to a transparent boundary condition involving a Dirichlet to Neumann operator. Our analysis emphasizes some strong singularities of the Stokes-Coriolis operator at low tangential frequencies. One of the main features of our work lies in the definition of a Dirichlet to Neumann operator for the Stokes-Coriolis system with data in the Kato space $H^{1/2}_{uloc}$.
\end{abstract}

\section{Introduction}

The goal of the present paper is to prove the existence and uniqueness of solutions to the Stokes-Coriolis system
\be\label{SC}
\left\{\begin{array}{rll}
-\Delta u + e_3 \times u + \na p&=0& \text{in }\Om,\\
\dv u&=0&\text{in }\Om,\\
u|_{\Gamma}&=u_0&
\end{array}\right.
\ee
where 
$$\begin{aligned}
\Om:=\{x\in \R^3, \ x_3>\om(x_h)\},\\
\Gamma=\pa\Om=\{x\in \R^3, \ x_3=\om(x_h)\}
\end{aligned}
$$
and $\om: \R^2\to \R^2$ is a bounded function. 

When $\om$ has some structural properties, such as periodicity, existence and uniqueness of solutions are easy to prove: our aim here is to prove well-posedness when the function $\om$ is arbitrary, say $\om\in W^{1,\infty}(\R^2)$, and when the boundary data $u_0$ is not square integrable. More precisely, we wish to work with $u_0$ in a space of infinite energy of Sobolev regularity, such as Kato spaces. We refer to the end of this introduction for a definition of these uniformly locally Sobolev spaces $L^2_{uloc},\ H^s_{uloc}$.

The interest for such function spaces to study fluid systems goes back to the papers by Lieumari\'e-Rieusset \cite{LemRieu1,Lemarecent}, in which existence is proved for weak solutions of the Navier-Stokes equations in $\R^3$ with initial data in $L^2_{uloc}$. These works fall into the analysis of fluid flows  with infinite energy, which is an field of intense research. Without being exhaustive, let us quote the works of: 
\begin{itemize}
\item Cannon and Knightly \cite{CK70}, Giga, Inui and Matsui \cite{GIM99}, Solonnikov \cite{Sol03nondecay}, Bae and Jin \cite{BaeJin12} (local solutions), Giga, Matsui and Sawada \cite{GMS01} (global solutions) on the nonstationary Navier-Stokes system in the whole space or in the half-space with initial data in $L^\infty$ or in $BUC$ (bounded uniformly continuous);
\item Basson \cite{Basson06}, Maekawa and Terasawa \cite{MaeTera06} on local solutions of the nonstationary Navier-Stokes system in the whole space with initial data in $L^p_{uloc}$ spaces;
\item Giga and Miyakawa \cite{GiMi}, Taylor \cite{Taylor92} (global solutions), Kato \cite{Kato92} on local solutions to the nonstationary Navier-Stokes system, and Gala \cite{GalaQGS06} on global solutions to a quasi-geostrophic equation, with initial data in Morrey spaces;
\item Gallagher and Planchon \cite{GaPl} on the nonstationary Navier-Stokes system in $\mathbb R^2$ with initial data in the homogeneous Besov space $\dot B^{2/r-1}_{r,q}$;
\item Giga and co-authors \cite{Giga++07} on the nonstationary Ekman system in $\mathbb R^3_+$ with initial data in the Besov space $\dot B^0_{\infty, 1,\sigma}\left(\R^2;L^p(\R_+)\right)$, for $2<p<\infty$; see also \cite{Giga++06} (local solutions), \cite{Giga++08} (global solutions) on the Navier-Stokes-Coriolis system in $\mathbb R^3$ and the survey of Yoneda \cite{Yoneda_survey} for initial data spaces containing almost-periodic functions;
\item Konieczny and Yoneda \cite{KoYo11} on the stationary Navier-Stokes system in Fourier-Besov spaces.

\item G\'erard-Varet and Masmoudi \cite{DGVNMnoslip} on the 2d Stokes system in the half-plane above a rough surface, with $H^{1/2}_{uloc}$ boundary data.

\item Alazard, Burq and Zuily \cite{ABZ} on the Cauchy problem for gravity water waves with data in $H^s_{uloc}$; the authors study in particular the Dirichlet to Neumann operator associated with  the laplacian in a domain $\Om=\{(x,y)\in \mathbb R^{d+1}, \eta^*(x)<y<\eta(x)\}$, with $H^{1/2}_{uloc}$ boundary data.

\end{itemize}
Despite this huge literature on initial value problems in fluid mechanics in spaces of infinite energy, we are not aware of such work concerning stationary systems and non homogeneous boundary value problems in $\mathbb R^3_+$.  Let us emphasize that the derivation of energy estimates in  stationary and time dependent settings are rather different: indeed, in a time dependent setting,  boundedness of the solution at time $t$ follows from  boundedness of the initial data and of the associated semi-group. In a stationary setting and in a domain with a boundary, to the best of our knowledge, the only way to derive estimates without assuming any structure on the function $\om$ is based on the arguments of Ladyzhenskaya and Solonnikov \cite{LS} (see also \cite{DGVNMnoslip} for the Stokes system in a bumped half plane).

In the present case, our motivation comes from the asymptotic analysis of highly rotating fluids near a rough boundary. Indeed, consider the system
\begin{equation}\label{sysrotfluidslin}
\left\{\begin{array}{rll}
\ds -{\varepsilon}\Delta u^\varepsilon + \frac{1}{\varepsilon}e_3 \times u^\varepsilon + \na p^\varepsilon&=0&\text{in }\Om^\varepsilon,\\
\dv u^\varepsilon&=0&\text{in }\Om^\varepsilon,\\
u^\varepsilon|_{\Gamma^\varepsilon}&=0,&\\
u^\varepsilon|_{x_3=1}&=(V_h,0),&
\end{array}\right.
\end{equation}
where 
$\Om^\varepsilon:=\{ x\in \R^3,\ \varepsilon \om(x_h/\varepsilon)<x_3<1\}$ and $\Gamma^\varepsilon:=\pa\Om^\varepsilon\setminus\{x_3=1\}$. Then it is expected that $u^\varepsilon$ is the sum of a two-dimensional interior flow $(u^{int}(x_h),0)$ balancing the rotation with the pressure term and a boundary layer flow $u^{BL}(x/\varepsilon;x_h)$, located in the vicinity of the lower boundary. In this case, the equation satisfied by $u^{BL}$ is precisely \eqref{SC}, with $u_0(y_h;x_h)=-(u^{int}(x_h),0)$. Notice that $x_h$ is the macroscopic variable and is a parameter in the equation on $u^{BL}$. The fact that the Dirichlet boundary condition is constant with respect to the fast variable $y_h$ is the original motivation for study of the well-posedness \eqref{SC} in spaces of infinite energy, such as the Kato spaces $H^s_{uloc}$.

The system \eqref{sysrotfluidslin} models large-scale geophysical fluid flows in the linear r\'egime. In order to get a physical insight into the physics of rotating fluids, we refer to the book by Greenspan \cite{Greenspan} (rotating fluids in general, including an extensive study of the linear r\'egime) and to the one by Pedlosky \cite{Pedlovsky} (focus on geophysical fluids). In \cite{Ekman05}, Ekman analyses the effect of the interplay between viscous forces and the Coriolis acceleration on geophysical fluid flows.

For further remarks on the system \eqref{sysrotfluidslin}, we refer to the book \cite[section $7$]{CDGG} by Chemin, Desjardins, Gallagher and Grenier, and to \cite{CDGGEkman}, where a model with anisotropic viscosity is studied and an asymptotic expansion for $u^\varepsilon$ is obtained.

Studying \eqref{SC} with an arbitrary function $\om$ is more realistic from a physical point of view, and also allows us to bring to light some bad behaviours of the system at low horizontal frequencies, which are masked in a periodic setting.

Our main result is the following.
\begin{theorem}
Let $\om\in W^{1,\infty}(\R^2)$, and let $u_{0,h}\in H^2_{uloc}(\R^2)^2$, $u_{0,3}\in H^1_{uloc}(\R^2)$. Assume that there exists $U_h\in H^{1/2}_{uloc}(\R^2)^2$ such that
\be\label{compatibility}
u_{0,3} -\na_h \om \cdot u_{0,h}=\na_h\cdot U_h.
\ee

Then there exists a unique solution $u$ of \eqref{SC} such that
$$
\begin{aligned}
\forall a>0,\quad \sup_{l\in \Z^2} \|u\|_{H^1\left(\left((l+[0,1]^2)\times (-1,a)\right)\cap \Omega\right)}<\infty,\\
 \sup_{l\in \Z^2}\sum_{\alpha\in \N^3, |\alpha|=q}\int_1^\infty \int_{l+[0,1]^2}|\na^\alpha u|^2<\infty
\end{aligned}
$$
for some integer $q$ sufficiently large, which does not depend on $\om$ nor $u_0$ (say $q\geq 4$).

\label{thm:ex/uni}
\end{theorem}
\begin{remark}
\begin{itemize}
\item Assumption \eqref{compatibility} is a compatibility condition, which stems from singularities at low horizontal frequencies in the system. When the bottom is flat, it merely becomes $u_{0,3}=\na_h\cdot U_h$. Notice that this condition only bears on the normal component of the velocity at the boundary: in particular, if $u_0\cdot n|_{\Gamma}=0$, then \eqref{compatibility} is satisfied. We also stress that \eqref{compatibility} is satisfied in the framework of highly rotating fluids near a rough boundary, since in this case $u_{0,3}=0$ and $u_{0,h}$ is constant with respect to the microscopic variable.

\item The singularities at low horizontal frequencies also account for the possible lack of integrability of the gradient far from the rough boundary: we were not able to prove that
$$
 \sup_{l\in \Z^2}\int_1^\infty \int_{l+[0,1]^2}|\na u|^2<\infty
$$
although this estimate is true for the Stokes system. In fact, looking closely at our proof, it seems that non-trivial cancellations should occur for such a result to hold in the Stokes-Coriolis case.

\item Concerning the regularity assumptions on $\om$ and $u_0$, it is classical to assume Lipschitz regularity on the boundary. The regularity required on $u_0$, however, may not be optimal, and stems in the present context from an explicit lifting of the boundary condition. It is possible that the regularity could be lowered if a different type of lifting were used, in the spirit of Proposition 4.3 in \cite{ABZ}. Let us stress as well that if $\om$ is constant, then $H^{1/2}_{uloc}$ regularity is enough (cf. Corollary \ref{cor:ex/uni-SC-uloc}).

\item The same tools can be used to prove a similar result for the Stokes system in three dimensions (we recall that the paper \cite{DGVNMnoslip} is concerned with the Stokes system in two dimensions). In fact, the treatment of the Stokes system is  easier, because the associated kernel is homogeneous and has no singularity at low frequencies. The results proved in Section \ref{sec:prelim} can be obtained thanks to the Green function associated with the Stokes system in three dimensions (see \cite{GaldiI}). On the other hand, the arguments of sections \ref{sec:existence} and \ref{sec:uniqueness} of the present paper can be transposed as such to the Stokes system in 3d. The main novelties of these sections, which rely on careful energy estimates, are concerned with the higher dimensional space rather than with the presence of the rotation term (except for Lemma \ref{lemnoyauordre1}).

\end{itemize}

\end{remark}

The statement of Theorem \ref{thm:ex/uni} is very close to one of the main results of the paper \cite{DGVNMnoslip} by G\'erard-Varet and Masmoudi, namely the well-posedness of the Stokes system in a bumped half-plane with boundary data in $H^{1/2}_{uloc}(\R)$. Of course, it shares the main difficulties of \cite{DGVNMnoslip}: spaces of functions of infinite energy, lack of a Poincar\'e inequality, irrelevancy of scalar tools (Harnack inequality, maximum principle) which do not apply to systems. But two additional problems are encountered when studying \eqref{SC}:
\begin{enumerate}
\item First, \eqref{SC} is set in three dimensions, whereas the study of \cite{DGVNMnoslip} took place in 2d. This complicates the derivation of energy estimates. Indeed, the latter are based on the truncation method by Ladyzhenskaya and Solonnikov \cite{LS}, which consists more or less in multiplying \eqref{SC} by $\chi_k u$, where $\chi_k\in \mathcal C^\infty_0(\R^{d-1})$ is a cut-off function in the horizontal variables such that $\Supp \chi_k\subset B_{k+1}$ and $\chi_k\equiv 1$ on $B_k$, for $k\in \N$. If $d=2$, the size of the support of $\na \chi_k$ is bounded, while it is unbounded when $d=3$. This has a direct impact on the treatment of some commutator terms.

\item Somewhat more importantly, the kernel associated with the Stokes-Coriolis operator has a more complicated expression than the one associated with the Stokes operator (see \cite[Chapter IV]{GaldiI} for the computation of the Green function associated to the Stokes system in the half-space). In the case of the Stokes-Coriolis operator, the kernel is not homogeneous, which prompts us to distinguish between high and low horizontal frequencies throughout the paper. Moreover, it
exhibits strong singularities at low horizontal frequencies, which have repercussions on the whole proof and account for assumption \eqref{compatibility}.

\end{enumerate}

The proof of Theorem \ref{thm:ex/uni} follows the same general scheme as in \cite{DGVNMnoslip} (this scheme has also been successfully applied in \cite{DGVALDnavier} in the case of a Navier slip boundary condition on the rough bottom): we first perform a thorough analysis of the Stokes-Coriolis system in $\R^3_+$, and we define the associated Dirichlet to Neumann operator for boundary data in $H^{1/2}_{uloc}$.  In particular, we derive a representation formula for solutions of the Stokes-Coriolis system in $\R^3_+$, based on a decomposition of the kernel which distinguishes high and low frequencies, and singular/regular terms. We also prove a similar representation formula for the Dirichlet to Neumann operator.
Then, we derive an equivalent system to \eqref{SC}, set in a domain which is bounded in $x_3$ and in which a transparent boundary condition is prescribed on the upper boundary. These two preliminary steps are performed in Section \ref{sec:prelim}. We then work with the equivalent system, for which we derive energy estimates in $H^1_{uloc}$; this allows us to prove existence in Section \ref{sec:existence}. Eventually, we prove uniqueness in Section \ref{sec:uniqueness}. An Appendix gathers several technical lemmas used throughout the paper.

\subsection*{Notations}
We will be working with spaces of uniformly locally integrable functions, called Kato spaces, whose definition we now recall (see \cite{kato}).
Let $\vartheta\in  \mathcal C^\infty_0(\R^d)$ such that $\supp \vartheta \subset [-1,1]^d$, $\vartheta \equiv 1$ on $[-1/4,1/4]^d$, and
\be\label{hyp:chi}
\sum_{k\in \Z^d}\tau_k\vartheta(x)=1\quad \forall x\in \R^d,
\ee
where $\tau_k$ is the translation operator defined by $\tau_k f(x)=f(x-k)$.

Then, for $s\geq 0$, $p\in [1,\infty)$
$$
\begin{aligned}
L^p_{uloc}(\R^d):=\{u\in L^p_{loc}(\R^d),\ \sup_{k\in \Z^d}\|(\tau_k \vartheta) u\|_{L^p(\R^d)}<\infty\},\\
H^s_{uloc}(\R^d):=\{u\in H^s_{loc}(\R^d),\ \sup_{k\in \Z^d}\|(\tau_k \vartheta) u\|_{H^s(\R^d)}<\infty\}.
\end{aligned}
$$
The space $H^s_{uloc}$ is independent of the choice of the function $\vartheta$ (see Lemma 3.1 in \cite{ABZ}).

We will also work in the domain $\Omb:=\{x\in \R^3,\ \om(x_h)<x_3<0\}$, assuming that $\om$ takes values in $(-1,0)$. With a slight abuse of notation, we will write
$$
\begin{aligned}
\|u\|_{L^p_{uloc}(\Omb)}:=\sup_{k\in \Z^2}\|(\tau_k \vartheta)u\|_{L^p(\Omb)},\\
\|u\|_{H^s_{uloc}(\Omb)}:=\sup_{k\in \Z^2}\|(\tau_k \vartheta) u\|_{H^s( \Omb)},
\end{aligned}
$$
where the function $\vartheta$ belongs to $\mathcal C^\infty_0(\R^2)$ and satisfies \eqref{hyp:chi}, $\supp \vartheta \subset [-1,1]^2$, $\vartheta \equiv 1$ on $[-1/4,1/4]^2$, 
and $H^s_{uloc}(\Omb)=\{u\in H^s_{loc}(\Omb),\ \|u\|_{H^s_{uloc}(\Omb)}<\infty\},\ L^p_{uloc}(\Omb)=\{u\in L^p_{loc}(\Omb),\ \|u\|_{L^p_{uloc}(\Omb)}<\infty\}$.

Throughout the proof, we will often use the notation $|\na^q u |$, where $q\in \N$, for the quantity
$$
\sum_{\alpha\in \N^d, |\alpha|=q}|\na^\alpha u|,
$$
where $d=2$ or $3$, depending on the context.
\section{Presentation of a reduced system and main tools}
\label{sec:prelim}
Following an idea of David G\'erard-Varet and Nader Masmoudi \cite{DGVNMnoslip}, the first step is to transform \eqref{SC} so as to work in a domain bounded in the vertical direction (rather than a half-space). This allows us eventually to use Poincar\'e inequalities, which are paramount in the proof.
To that end, we introduce an artificial flat boundary above the rough surface $\Gamma$, and we replace the Stokes-Coriolis system in the half-space above the artificial boundary by a transparent boundary condition, expressed in terms of a Dirichlet to Neumann operator.

In the rest of the article, without loss of generality, we assume that $\sup \om=:\alpha<0$ and $\inf \om \geq -1$, and we place the artificial boundary  at $x_3=0$. We set
$$\begin{aligned}
\Omb&:=\{x\in \R^3,\ \om(x_h)<x_3<0\},\\
\Sigma&:=\{x_3=0\}.
\end{aligned}
$$

The Stokes-Coriolis system differs in several aspects from the Stokes system; in the present paper, the most crucial differences are  the lack of an explicit Green function, and the bad behaviour of the system at low horizontal frequencies.
The main steps of the proof are as follows:
\begin{enumerate}
\item Prove existence and uniqueness of a solution of the Stokes-Coriolis system in a half-space with a boundary data in $H^{1/2}(\R^2)$;
\item Extend this well-posedness result to boundary data in $H^{1/2}_{uloc}(\R^2)$;
\item Define the Dirichlet to Neumann operator for functions in $H^{1/2}(\R^2)$, and extend it to functions in $H^{1/2}_{uloc}(\R^2)$;
\item Define an equivalent problem in $\Omb$, with a transparent boundary condition at $\Sigma$, and prove the equivalence between the problem in $\Omb$ and the one in $\Om$;
\item Prove existence and uniqueness of solutions of the equivalent problem.
\end{enumerate}
Items 1-4 will be proved in the current section, and item 5 in sections \ref{sec:existence} and \ref{sec:uniqueness}.

\subsection{The Stokes-Coriolis system in a half-space}
\label{secesthalfspace}

The first step is to study the properties of the Stokes-Coriolis system in $\R^3_+$, namely
\be
\label{SC-R3+}
\left\{\begin{array}{rll}
-\Delta u + e_3 \times u + \na p&=0& \text{in }\R^3_+,\\
\dv u&=0& \text{in }\R^3_+,\\
u|_{x_3=0}&=v_0.&
\end{array}\right.
\ee
In order to prove the result of Theorem \ref{thm:ex/uni}, we have to prove the existence and uniqueness of a solution $u$ of the Stokes-Coriolis system in $H^1_{loc}(\R^3_+)$ such that for some $q\in \N$ sufficiently large,
$$
\sup_{l\in \Z^2}\int_{l+(0,1)^2}\int_1^\infty |\na^q u|^2<\infty
$$
However, the Green function for the Stokes-Coriolis is far from being explicit, and its Fourier transform, for instance, is much less well-behaved than the one of the Stokes system (which is merely the Poisson kernel). Therefore such a result is not so easy to prove. In particular, because of the singularities of the Fourier transform of the Green function at low frequencies, we are not able to prove that
$$
\sup_{l\in \Z^2}\int_{l+(0,1)^2}\int_1^\infty |\na u|^2<\infty.
$$

$\bullet$ We start by solving the system when $v_0\in H^{1/2}(\R^2)$. We have the following result:

\begin{proposition}
 Let $v_0\in H^{1/2}(\R^2)^3$ such that
\be\label{hyp:v03}
 \int_{\R^2}\frac{1}{|\xi|} |\hat v_{0,3}(\xi)|^2\:d\xi<\infty.
\ee
  Then the system \eqref{SC-R3+} admits a unique solution $u\in H^1_{loc}(\R^3_+)$ such that
$$
\int_{\R^3_+} |\na u|^2<\infty.
$$
\label{prop:ex/uni-SC-H12}
\end{proposition}
\begin{remark}
The condition \eqref{hyp:v03} stems from a singularity at low frequencies of the Stokes-Coriolis system, which we will encounter several times in the proof. Notice that \eqref{hyp:v03} is satisfied in particular when $v_{0,3}=\na_h\cdot V_h$ for some $V_h\in H^{1/2}(\R^2)^2$, which is sufficient for further purposes.
\end{remark}

\begin{proof}
\noindent $\bullet$ \textit{Uniqueness}.
Consider a solution whose gradient is in $L^2(\R^3_+)$ and with zero boundary data on $x_3=0$. Then, 
using the Poincar\'e inequality, we infer that
$$
\int_0^a\int_{\R^2}|u |^2 \leq C_a \int_0^a\int_{\R^2}|\na u |^2 <\infty,
$$
and therefore we can take the Fourier transform of $u$ in the horizontal variables. Denoting by $\xi\in\R^2$ the Fourier variable associated with $x_h$, we get
\be\label{eq:SC-Fourier}
\left\{\begin{array}{rl}
(|\xi|^2-\pa_3^2) \hat u_h + \hat u_h^\bot + i \xi \hat p&=0,\\
(|\xi|^2-\pa_3^2) \hat u_3 + \pa_3 \hat p&=0,\\
i\xi\cdot \hat u_h + \pa_3 \hat u_3&=0,
\end{array}\right.
\ee
and
$$
\hat u|_{x_3=0}=0.
$$
Eliminating the pressure, we obtain
$$
(|\xi|^2-\pa_3^2)^2 \hat u_3 -i \pa_3 \xi^\bot \cdot \hat u_h=0.
$$
Taking the scalar product of the first equation in \eqref{eq:SC-Fourier} with $(\xi^\bot, 0)$, and using the divergence-free condition, we are led to
\be\label{eq:u3-Fourier}
(|\xi|^2-\pa_3^2)^3 \hat u_3 - \pa_3^2 \hat u_3 =0.
\ee
Notice that the solutions of this equation have a slightly different nature when $\xi\neq 0$ or when $\xi=0$ (if $\xi=0$, the associated characteristic polynomial has a multiple root at zero). Therefore, as in \cite{DGVNMnoslip} we introduce a function $\varphi=\varphi(\xi)\in \mathcal C^\infty_0(\R^2)$ such that the support of $\varphi $ does not contain zero. Then $\varphi \hat u_3$ satisfies the same equation as $\hat u_3$, and vanishes in a neighbourhood of $\xi=0$.

For $\xi\neq 0$, the solutions of \eqref{eq:u3-Fourier} are linear combinations of $\exp(-\lambda_k  x_3)$ (with coefficients depending on $\xi$), where $(\lambda_k)_{1\leq k \leq 6}$ are the complex valued solutions of the equation
\be\label{eq:lambda}
(\la^2-|\xi|^2)^3  + \la^2 =0.
\ee
Notice that none of the roots of this equation is purely imaginary, and that if $\lambda$ is a solution of \eqref{eq:lambda}, so are $-\lambda$, $\bar \lambda$ and $-\bar \lambda$. Additionally \eqref{eq:lambda} has exactly one real valued positive solution. Therefore, without loss of generality we assume that $\lambda_1$, $\lambda_2$, $\lambda_3$ have strictly positive real part, while $\lambda_4, \lambda_5,\lambda_6$ have strictly negative real part, and $\lambda_1\in \R$, $\bar \lambda_2=\lambda_3$, with $\Im(\lambda_2)>0$, $\Im(\lambda_3)<0$.

On the other hand, the integrability condition on the gradient becomes
$$
\int_{\R^3_+}(|\xi|^2 |\hat u(\xi,x_3)|^2 + |\pa_3\hat u(\xi,x_3)|^2 )d\xi\: dx_3<\infty.
$$
We infer immediately that $\varphi \hat u_3$ is a linear combination of $\exp(-\lambda_k  x_3)$ for $1\leq k\leq 3$: there exist $A_k:\R^2\to \C^3$ for $k=1,2,3$ such that
$$
\varphi(\xi)\hat u_3(\xi, x_3)=\sum_{k=1}^3 A_k(\xi) \exp(-\lambda_k(\xi) x_3).
$$
Going back to \eqref{eq:SC-Fourier}, we also infer that
\be\label{eq:u_h-Fourier}
\begin{aligned}
\varphi(\xi)\xi \cdot \hat u_h(\xi, x_3)&=-i\sum_{k=1}^3\lambda_k(\xi) A_k(\xi) \exp(-\lambda_k(\xi) x_3),\\
\varphi(\xi)\xi^\bot \cdot \hat u_h(\xi, x_3)&=i\sum_{k=1}^3 \frac{(|\xi|^2-\lambda_k^2)^2}{\lambda_k}A_k(\xi) \exp(-\lambda_k(\xi) x_3).
\end{aligned}
\ee
Notice that by \eqref{eq:lambda}, 
$$
\frac{(|\xi|^2-\lambda_k^2)^2}{\lambda_k}=\frac{\lambda_k}{|\xi|^2-\lambda_k^2}\quad \text{for } k=1,2,3.
$$

Thus the boundary condition $\hat u|_{x_3=0}=0$ becomes
$$
M(\xi)\begin{pmatrix}
A_1(\xi)\\ A_2(\xi)\\ A_3(\xi)
\end{pmatrix}
=0,
$$
where
$$
M:=
\begin{pmatrix}
1&1&1\\
\lambda_1 & \lambda_2 &\lambda_3\\
\ds \frac{(|\xi|^2-\lambda_1^2)^2}{\lambda_1} &  \ds\frac{(|\xi|^2-\lambda_2^2)^2}{\lambda_2} &  \ds\frac{(|\xi|^2-\lambda_3^2)^2}{\lambda_3}
\end{pmatrix}.
$$
We have the following lemma:
\begin{lemma}\label{lem:det-M}
$$\det M= (\lambda_1-\lambda_2)(\lambda_2-\lambda_3)(\lambda_3-\lambda_1)(|\xi |+ \lambda_1 + \lambda_2 + \lambda_3).$$
\end{lemma}
Since the proof of the result is a mere calculation, we have postponed it to Appendix \ref{appendixexp}. It is then clear that $M$ is invertible for all $\xi\neq 0$: indeed it is easily checked that all the roots of \eqref{eq:lambda} are simple, and we recall that $\lambda_1, \lambda_2, \lambda_3$ have positive real part.

We conclude that $A_1=A_2=A_3=0$, and thus $\varphi(\xi) \hat u(\xi,x_3)=0$ for all $\varphi\in \mathcal C^\infty_0(\R^2)$ supported far from $\xi=0$. Since $\hat u \in L^2(\R^2\times(0,a))^3$ for all $a>0$, we infer that $\hat u=0$.

\vskip2mm

\noindent $\bullet$ \textit{Existence}.
Now, given $v_0\in H^{1/2}(\R^2)$, we define $u$ through its Fourier transform in the horizontal variable. It is enough to define the Fourier transform for $\xi\neq 0$, since it is square integrable in $\xi$. Following the calculations above, we define coefficients
$A_1,A_2,A_3$ by the equation
\be\label{def:A}
M(\xi)\begin{pmatrix}
A_1(\xi)\\ A_2(\xi)\\ A_3(\xi)
\end{pmatrix}
= \begin{pmatrix}
\hat v_{0,3}\\ i\xi \cdot \hat v_{0,h}\\ -i\xi^\bot \cdot \hat v_{0,h}
\end{pmatrix}\quad \forall \xi\neq 0.
\ee
As stated in Lemma \ref{lem:det-M}, the matrix $M$ is invertible, so that $A_1, A_2, A_3$ are well defined. We then set
\be\label{def:u-Fourier}
\begin{aligned}
\hat u_3(\xi, x_3)&:= \sum_{k=1}^3 A_k(\xi)\exp(-\lambda_k(\xi) x_3),\\
\hat u_h(\xi, x_3)&:= \frac{i}{|\xi|^2}\sum_{k=1}^3A_k(\xi)\left(-\lambda_k(\xi)\xi +  \frac{(|\xi|^2-\lambda_k^2)^2}{\lambda_k}\xi^\bot \right)\exp(-\lambda_k(\xi) x_3).
\end{aligned}
\ee
We have to check that the corresponding solution is sufficiently integrable, namely
\be\label{bornes:existence-H1}
\begin{aligned}
\int_{\R^3_+}(|\xi|^2 | \hat u_h(\xi, x_3)|^2 +| \pa_3 \hat u_h(\xi, x_3)|^2  )d\xi\:dx_3<\infty,\\
\int_{\R^3_+}(|\xi|^2 | \hat u_3(\xi, x_3)|^2 +| \pa_3 \hat u_3(\xi, x_3)|^2  )d\xi\:dx_3<\infty.
\end{aligned}
\ee
Notice that by construction, $\pa_3 \hat u_3=-i\xi \cdot \hat u_h$ (divergence-free condition), so that we only have to check   three  conditions. 

To that end, we need to investigate the behaviour of $\lambda_k, A_k$ for $\xi$ close to zero and for $\xi\to \infty$. We gather the results in the following lemma, whose proof is once again postponed to Appendix \ref{appendixexp}:
\begin{lemma}\label{lem:dev-lambda-A}\ 
\begin{itemize}
\item As $\xi \to \infty$, we have
$$
\begin{aligned}
\lambda_1&=|\xi|-\frac{1}{2}|\xi|^{-\frac{1}{3}}+O\left(|\xi|^{-\frac{5}{3}}\right),\\
\lambda_2&=|\xi|-\frac{j^2}{2}|\xi|^{-\frac{1}{3}}+O\left(|\xi|^{-\frac{5}{3}}\right),\\
\lambda_3&=|\xi|-\frac{j}{2}|\xi|^{-\frac{1}{3}}+O\left(|\xi|^{-\frac{5}{3}}\right),
\end{aligned}
$$
where $j=\exp(2i\pi/3)$, so that
\be\label{dev:A-infty}
\begin{pmatrix}
A_1(\xi)\\A_2(\xi) \\A_3(\xi)
\end{pmatrix}
= \frac{1}{3}
\begin{pmatrix}
1&1&1\\
1&j&j^2\\
1&j^2&j
\end{pmatrix}
\begin{pmatrix}
\hat v_{0,3}\\
-2|\xi|^{1/3}(i\xi \cdot \hat v_{0,h}- |\xi| \hat v_{0,3}) + O(|\hat v_0|)\\
-|\xi|^{-1/3} i \xi^\bot \cdot  \hat v_{0,h} + O(|\hat v_0|)
\end{pmatrix}.
\ee

\item As $\xi \to 0$, we have
$$
\begin{aligned}
\lambda_1&=|\xi|^3+O\left(|\xi|^7\right),\\
\lambda_2&=e^{i\frac{\pi}{4}}+O\left(|\xi|^2\right),\\
\lambda_3&=e^{-i\frac{\pi}{4}}+O\left(|\xi|^2\right).
\end{aligned}
$$
As a consequence, for $\xi$ close to zero,
\be\label{dev:A-0}
\begin{aligned}
A_1(\xi)&=\hat v_{0,3}(\xi) - \frac{\sqrt{2}}{2}\left(i \xi \cdot \hv_{0,h} + i \xi^\bot \hv_{0,h} + |\xi| \hv_{0,3}\right) + O(|\xi|^2 |\hat v_0(\xi)|),\\
A_2(\xi)&=\frac{1}{2}\left(e^{-i\pi/4} i \xi \cdot\hv_{0,h} + e^{i\pi/4}(i\xi^\bot \hv_{0,h}+ |\xi| \hv_{0,3})\right)+ O(|\xi|^2 |\hat v_0(\xi)|),\\
A_3(\xi)&=\frac{1}{2}\left(e^{i\pi/4} i \xi \cdot\hv_{0,h} + e^{-i\pi/4}(i\xi^\bot \hv_{0,h}+ |\xi| \hv_{0,3})\right)+ O(|\xi|^2 |\hat v_0(\xi)|).
\end{aligned}
\ee

\item For all $a\geq 1$, there exists a constant $C_a>0$ such that
$$
a^{-1}\leq |\xi|\leq a\Longrightarrow 
\left\{
\begin{array}{l}
|\lambda_k(\xi)| + |\Re(\lambda_k(\xi))|^{-1}\leq C_a ,\\
|A(\xi)| \leq C_a |\hat v_0(\xi)|.
\end{array}
\right.
$$

\end{itemize}
\end{lemma}

We then decompose each integral in \eqref{bornes:existence-H1} into three pieces, one on $\{|\xi|>a\}$, one on $\{|\xi|<a^{-1}\}$ and the last one on $\{|\xi|\in (a^{-1}, a)\}$. All the integrals on $\{a^{-1}\leq |\xi|\leq a\}$ are bounded by
$$
C_a \int_{a^{-1}<|\xi|< a} | \hat v_0(\xi)|^2\:d\xi \leq C_a \| v_0\|_{H^{1/2}(\R^2)}^2.
$$
We thus focus on the two other pieces. We only treat  the term
$$
\int_{\R^3_+} |\xi|^2 |\hat u_3(\xi, x_3)|^2\: d\xi \: dx_3,
$$
since the two other terms can be evaluated using similar arguments.

$\rhd$ On the set $\{|\xi|>a\}$, the difficulty comes from the fact that the contributions of the three exponentials compensate one another; hence a rough estimate is not possible. In order to simplify the calculations,  we introduce the following notation: we set
\be\label{def:B}
\begin{aligned}
B_1&=A_1 + A_2 + A_3,\\
B_2&=A_1 + j^2 A_2 + j A_3,\\
B_3&=A_1 + j A_2 + j^2 A_3,
\end{aligned}
\ee
so that
$$
\begin{pmatrix}
A_1\\ A_2\\ A_3
\end{pmatrix}
= \frac{1}{3}
\begin{pmatrix}
1&1&1\\
1&j&j^2\\
1&j^2&j
\end{pmatrix}
\begin{pmatrix}
B_1\\ B_2\\ B_3
\end{pmatrix}.
$$
Hence we have $A_k=(B_1 + \alpha_k B_2 + \alpha_k^2B_3)/3$, where $\alpha_1=1, \alpha_2=j, \alpha_3=j^2$. Notice that $\alpha_k^3=1$ and $\sum_k \alpha_k=0$. 
According to Lemma \ref{lem:dev-lambda-A}, 
$$
\begin{aligned}
B_1&=\hat v_{0,3},\\
B_2&=-2|\xi|^{1/3}(i\xi \cdot \hat v_{0,h}- |\xi| \hat v_{0,3}) + O(|\hat v_0|),\\
B_3&=-|\xi|^{-1/3} i \xi^\bot \cdot  \hat v_{0,h} + O(|\hat v_0|).
\end{aligned}
$$
For all $\xi\in \R^2$, $|\xi|>a$, we have
$$
|\xi|^2\int_0^\infty|\hat u_3(\xi, x_3)|^2 dx_3= |\xi|^2\sum_{1\leq k,l\leq 3} A_k\bar A_l \frac{1}{\lambda_k+\bar \lambda_l}.
$$
Using the asymptotic expansions in Lemma \ref{lem:dev-lambda-A}, we infer that
$$
\frac{1}{\lambda_k+\bar \lambda_l}= \frac{1}{2|\xi|} \left(1 + \frac{\alpha_k^2 + \bar \alpha_l^2}{2} |\xi|^{-4/3} + O(|\xi|^{-8/3})\right).
$$
Therefore,
we obtain for $|\xi|\gg 1$
\begin{eqnarray*}
|\xi|^2\sum_{1\leq k,l\leq 3} A_k\bar A_l \frac{1}{\lambda_k+\bar \lambda_l}
&=&\frac{|\xi|}{2}\sum_{1\leq k,l\leq 3} A_k\bar A_l \left(1 + \frac{\alpha_k^2 + \bar \alpha_l^2}{2} |\xi|^{-4/3} + O(|\xi|^{-8/3})\right)\\
&=&\frac{|\xi|}{2}\left(|B_1|^2 + \frac{1}{2}(B_2 \bar B_1 + \bar B_2 B_1 )|\xi|^{-4/3} + O(|\hat v_0|^2)\right)\\
&=&O(|\xi| \:|\hat v_0|^2).
\end{eqnarray*}
Hence, since $ v_0\in H^{1/2}(\R^2)$, we deduce that
$$
\int_{|\xi|>a}\int_0^\infty |\xi|^2 |\hat u_3|^2\:dx_3\:d\xi<+\infty.
$$
$\rhd$ On the set $|\xi|\leq a$, we can use a crude estimate: we have
\begin{eqnarray*}
\int_{|\xi|\leq a}\int_0^\infty |\xi|^2|\hat u_3(\xi, x_3)|^2dx_3\:d\xi&\leq&C\sum_{k=1}^3\int_{|\xi|\leq a}|\xi|^2\frac{|A_k(\xi)|^2}{2\Re(\lambda_k(\xi))}\:d\xi.
\end{eqnarray*}
Using the estimates of Lemma \ref{lem:dev-lambda-A}, we infer that
\begin{eqnarray*}
&&\int_{|\xi|\leq a}\int_0^\infty |\xi|^2|\hat u_3(\xi, x_3)|^2dx_3\:d\xi\\&\leq&C\int_{|\xi|\leq a }|\xi|^2\left( (|\hat v_{0,3}(\xi)|^2 + |\xi|^2 |\hat v_{0,h}(\xi)|^2 )\frac{1}{|\xi|^3} + |\xi|^2 |\hat v_{0}(\xi)|^2\right)\:d\xi\\
&\leq & C \int_{|\xi|\leq a }\left(\frac{|\hat v_{0,3}(\xi)|^2}{|\xi|}+ |\xi | \: |\hat v_{0,h}(\xi)|^2 \right)\:d\xi<\infty
\end{eqnarray*}
thanks to the assumption \eqref{hyp:v03} on $\hat v_{0,3}$. In a similar way, we have
$$
\begin{aligned}
\int_{|\xi|\leq a}\int_0^\infty |\xi|^2|\hat u_h(\xi, x_3)|^2dx_3\:d\xi &\leq C \int_{|\xi|\leq a }\left(\frac{|\hat v_{0,3}(\xi)|^2}{|\xi|}+ |\xi | \: |\hat v_{0,h}(\xi)|^2 \right)\:d\xi,\\
\int_{|\xi|\leq a}\int_0^\infty |\pa_3\hat u_h(\xi, x_3)|^2dx_3\:d\xi &\leq C \int_{|\xi|\leq a } |\hat v_{0}|^2\:d\xi.
\end{aligned}
$$
Gathering all the terms, we deduce that
$$
\int_{\R^3_+}(|\xi|^2 |\hat u(\xi, x_3)|^2 + |\pa_3 \hat u(\xi, x_3)|^2)d\xi\:dx_3<\infty,
$$
so that $\na u\in L^2(\R^3_+)$.
\end{proof}

\begin{remark}
\label{rem:decay-classical}
Notice that thanks to the exponential decay in Fourier space, for all $p\in \N$ with $p\geq 2$, there exists a constant $C_p>0$ such that
$$
\int_1^\infty\int_{\R^2} |\na^p u|^2\leq C_p\|v_0\|_{H^{1/2}}^2.
$$
\end{remark}

$\bullet$ We now extend the definition of a solution to boundary data in $H^{1/2}_{uloc}(\R^2)$. 
We introduce the sets
\be\label{def:K}
\begin{aligned}
\Kc&:=\left\{u\in H^{1/2}_{uloc}(\R^2), \ \exists U_h\in H^{1/2}_{uloc}(\R^2)^2,\ u=\na_h\cdot U_h\right\},\\
\bK&:=\left\{u\in H^{1/2}_{uloc}(\R^2)^3, \ u_3\in \Kc\right\}.
\end{aligned}
\ee
In order to extend the definition of solutions to data which are only locally square integrable, we will first derive a representation formula for $v_0\in H^{1/2}(\R^2)$. We will prove that the formula still makes sense when $v_0\in \bK$, and this will allow us to define a solution with boundary data in $\bK$.

To that end, let us introduce some notation. According to the proof of Proposition \ref{prop:ex/uni-SC-H12}, there exists $L_1,L_2,L_3:\R^2\to \mathcal M_3(\mathbb C)$ and $q_1,q_2,q_3:\R^2\to \C^3$ such that
\be\label{def:up}\begin{aligned}
\hat u(\xi,x_3)&= \sum_{k=1}^3 L_k(\xi) \hat v_0(\xi)\exp(-\lambda_k(\xi) x_3),\\
\hat p(\xi,x_3)&= \sum_{k=1}^3 q_k(\xi) \cdot\hat v_0(\xi)\exp(-\lambda_k(\xi) x_3).
\end{aligned}
\ee
For further reference, we state the following lemma:
\begin{lemma}
For all $k\in \{1,2,3\}$, for all $\xi \in \R^2$, the following identities hold
$$
(|\xi|^2-\lambda_k^2) L_k + \begin{pmatrix}
-L_{k,21} & -L_{k,22} &-L_{k,23}\\
L_{k,11}&L_{k,12}&L_{k,13}\\
0&0&0
\end{pmatrix}
+ \begin{pmatrix}
i\xi_1q_{k,1} & &i\xi_1 q_{k,2} & i\xi_1q_{k,3}\\
i\xi_2q_{k,1} & &i\xi_2 q_{k,2} & i\xi_2q_{k,3}\\
-\lambda_kq_{k,1} & &-\lambda_k q_{k,2} &-\lambda_kq_{k,3}
\end{pmatrix}=0
$$
and for $j=1,2,3$, $k=1,2,3$,
$$
i\xi_1 L_{k,1j} + i\xi_2 L_{k,2j} -\lambda_k L_{k,3j}=0.
$$

\label{lem:eq:L_k}
\end{lemma}
\begin{proof}
Let $v_0\in H^{1/2}(\R^2)^3$ such that $v_{0,3}=\na_h\cdot V_h$ for some $V_h\in H^{1/2}(\R^2)$. Then, according to Proposition \ref{prop:ex/uni-SC-H12}, the couple $(u,p)$ defined by \eqref{def:up} is a solution of \eqref{SC-R3+}. Therefore it satisfies \eqref{eq:SC-Fourier}. Plugging the definition \eqref{def:up} into \eqref{eq:SC-Fourier}, we infer that for all $x_3>0$,
\be\label{eq:calA}
\int_{\R^2}  \sum_{k=1}^3\exp(-\lambda_k x_3) \mathcal A_k(\xi)\hat v_0(\xi )\:d\xi=0,
\ee
where
$$
\mathcal A_k:=(|\xi|^2-\lambda_k^2) L_k + \begin{pmatrix}
-L_{k,21} & -L_{k,22} &-L_{k,23}\\
L_{k,11}&L_{k,12}&L_{k,13}\\
0&0&0
\end{pmatrix}
+ \begin{pmatrix}
i\xi_1q_{k,1} & &i\xi_1 q_{k,2} & i\xi_1q_{k,3}\\
i\xi_2q_{k,1} & &i\xi_2 q_{k,2} & i\xi_2q_{k,3}\\
-\lambda_kq_{k,1} & &-\lambda_k q_{k,2} &-\lambda_kq_{k,3}
\end{pmatrix}.
$$
Since \eqref{eq:calA} holds for all $v_0$, we obtain
$$
\sum_{k=1}^3\exp(-\lambda_k x_3) \mathcal A_k(\xi)=0 \quad \forall \xi\ \forall x_3,
$$
and since $\lambda_1,\lambda_2,\lambda_3$ are distinct for all $\xi\neq 0$, we deduce eventually that $\mathcal A_k(\xi)=0$ for all $\xi$ and for all $k$.

The second identity follows in a similar fashion from the divergence-free condition.
\end{proof}

Our goal is now to derive a representation formula for $u$, based on the formula satisfied by its Fourier transform, in such a way that the formula still makes sense when $v_0\in \bK$. The crucial part is to understand the action of the operators $\mathrm{Op}(L_k(\xi) \phi(\xi))$ on $L^2_{uloc}$ functions, where $\phi\in \mathcal C^\infty_0(\R^2)$. To that end, we will need to decompose $L_k(\xi)$ for $\xi$ close to zero into several terms.

Lemma \ref{lem:dev-lambda-A} provides asymptotic developments of $L_1,L_2,L_3$ and $\alpha_1,\alpha_2,\alpha_3$ as $|\xi|\ll 1$ or $|\xi|\gg 1$. In particular, we have, for $|\xi|\ll 1$,
\begin{eqnarray}
\label{exprL1}L_1(\xi)&=&\frac{\sqrt 2}{2|\xi|}\begin{pmatrix}
\xi_2(\xi_2-\xi_1) & -\xi_2(\xi_2+\xi_1) & {-i\sqrt{2}\xi_2}\\
\xi_1(\xi_1-\xi_2) & \xi_1(\xi_2+\xi_1) &{i\sqrt{2}\xi_1}\\
i|\xi|(\xi_2-\xi_1)& -i|\xi|(\xi_2+\xi_1)&\sqrt{2}|\xi|
\end{pmatrix}\\&&\nonumber+\begin{pmatrix}
O(|\xi|^2)& O(|\xi|^2) & O(|\xi|)
\end{pmatrix},
\end{eqnarray}
\begin{eqnarray}
\nonumber L_2(\xi)&=&\frac{1}{2}\begin{pmatrix}
1&i& \ds\frac{2i(-\xi_1+\xi_2)}{|\xi|}\\
-i &1&\ds\frac{-2i(\xi_1+\xi_2)}{|\xi|}\\
{i(\xi_1e^{-i\pi/4} -\xi_2 e^{i\pi/4} )} & {i(\xi_2e^{-i\pi/4} +i\xi_1 e^{i\pi/4} )} & {e^{i\pi/4}}
\end{pmatrix}\\\nonumber&&+ \begin{pmatrix}
O(|\xi|^2)& O(|\xi|^2) & O(|\xi|)
\end{pmatrix},
\end{eqnarray}
\begin{eqnarray}
\nonumber L_3(\xi)&=&
\frac{1}{2}\begin{pmatrix}
1&-i& \ds\frac{2i(\xi_1+\xi_2)}{|\xi|}\\
i &1&\ds\frac{-2i(\xi_1-\xi_2)}{|\xi|}\\
{i(\xi_1e^{i\pi/4} -\xi_2 e^{-i\pi/4} )}&{i(\xi_2e^{i\pi/4} +i\xi_1 e^{-i\pi/4} )} &{e^{-i\pi/4}}
\end{pmatrix}\\\nonumber&&+ \begin{pmatrix}
O(|\xi|^2)& O(|\xi|^2) & O(|\xi|)
\end{pmatrix}.
\end{eqnarray}
The remainder terms are to be understood column-wise.
Notice that the third column of $L_k$, i.e. $L_k e_3$, always acts on $\hat v_{0,3}=i\xi \cdot \hat V_h$.
We thus introduce the following notation: for $k=1,2,3$, $M_k:=(L_ke_1\: L_ke_2)\in \mathcal M_{3,2}(\C)$, and $N_k:=i L_k e_3\prescript{t}{}{\xi}\in \mathcal M_{3,2}(\C)$. 
$M_k^1$ (resp. $N_k^1$) denotes the $3\times 2$ matrix whose coefficients are the nonpolynomial and homogeneous terms of order one in $M_k$ (resp. $N_k$) for $\xi$ close to zero. For instance,
$$
M_1^1:=\frac{\sqrt 2}{2|\xi|}\begin{pmatrix}
\xi_2(\xi_2-\xi_1) & -\xi_2(\xi_2+\xi_1) \\
-\xi_1(\xi_2-\xi_1) & \xi_1(\xi_2+\xi_1) \\
0&0
\end{pmatrix}
,\quad
N_1^1:=\frac{i}{|\xi|}\begin{pmatrix}
-\xi_2\xi_1  &\xi_2^2\\
\xi_1^2&\xi_1\xi_2\\
0& 0
\end{pmatrix}.
$$
We also set $M_k^{rem}=M_k- M_k^1$, $N_k^{rem}:=N_k-N_k^1$, so that for $\xi$ close to zero,
$$
\begin{aligned}
&M_1^{rem}=O(|\xi|),\quad\text{and for }k=2,\ 3,\quad M_k^{rem}=O(1),\\
&\forall k\in \{1,2,3\},\quad N_k^{rem}=O(|\xi|).
\end{aligned}
$$
There are polynomial terms of order one in $M_1^{rem}$ and $N_k^{rem}$ (resp. of order $0$ and $1$ in $M_k^{rem}$ for $k=2,\ 3$) which account for the fact that the remainder terms are not $O(|\xi|^2)$. However, these polynomial terms do not introduce any singularity when there are differentiated and thus, using the results of Appendix \ref{appendixrestes}, we get, for any integer $q\geq 1$,
\be
\left|\na_\xi^q M_k^{rem}\right|, \ \left|\na_\xi^q N_k^{rem}\right| =O(|\xi|^{2-q} +1)\quad \text{for }|\xi|\ll 1.
\ee

$\rhd$ Concerning the Fourier multipliers of order one $M_k^1$ and $N_k^1$, we will rely on the following lemma, which is proved in Appendix \ref{appendixDroniouImbert}:

\begin{lemma}
There exists a constant $C_I$ such that for all $ i,j\in \{1,2\}$, for any function $g\in \mathcal S(\R^2)$, for all $\zeta\in \mathcal C^\infty_0(\R^2)$ and for all $K>0$,
\begin{align}\label{reprnoyau1BF}
&\mathrm{Op}\left(\frac{\xi_i\xi_j}{|\xi|}\zeta(\xi)\right)g(x)\nonumber\\
=\ &C_I\int_{\R^2}dy\left[\frac{\delta_{i,j}}{|x-y|^3} - 3\frac{(x_i-y_i) (x_j-y_j)}{|x-y|^5}\right]\times\\
&\hskip4cm\times\left\{ \rho\ast g(x)- \rho\ast g(y)- \na\rho\ast g(x)\cdot (x-y)\mathbf{1_{|x-y|\leq K}}\right\},\nonumber
\end{align}
where $\rho:=\mathcal F^{-1}\zeta\in\mathcal S\left(\mathbb R^2\right)$.

\label{lem:noyau-ordre1}
\end{lemma}

\begin{definition}\label{defIordre1}
If $L$ is a homogeneous, nonpolynomial function of order one in $\R^2$, of the form
$$
L(\xi)=\sum_{1\leq i,j\leq 2} a_{ij}\frac{\xi_i\xi_j}{|\xi|},
$$
then we define, for $\varphi\in W^{2,\infty}(\R^2)$,
$$
\mathcal I[L] \varphi(x)
:=\sum_{1\leq i,j\leq 2} a_{ij}\int_{\R^2}dy\gamma_{ij}(x-y)\left\{ \varphi (x)- \varphi (y)- \na \varphi(x)\cdot (x-y)\mathbf{1_{|x-y|\leq K}}\right\},
$$
where
$$
\gamma_{i,j}(x)=C_I \left(\frac{\delta_{i,j}}{|x|^3} - 3\frac{x_i x_j}{|x|^5}\right).
$$
\end{definition}
\begin{remark}
The value of the number $K$ in the formula \eqref{reprnoyau1BF} and in Definition \ref{defIordre1} is irrelevant, since for all $\varphi\in W^{2,\infty}(\R^2)$, for all $0<K<K'$,
$$
\int_{\R^2}dy\gamma_{ij}(x-y)\na \varphi(x)\cdot (x-y)\mathbf{1_{K<|x-y|\leq K'}}=0
$$
by symmetry arguments.
\end{remark}

We then have the following bound:
\begin{lemma}
Let $\varphi\in W^{2,\infty}(\R^2)$. Then for all $1\leq i,j\leq 2$,
$$
\left\| \mathcal I\left[\frac{\xi_i\xi_j}{|\xi|}\right] \varphi\right\|_{L^\infty(\R^2)}\leq C \|\varphi\|_\infty^{1/2}\|\na^2 \varphi\|_{\infty}^{1/2}.
$$

\label{lemauxorder1}
\end{lemma}
\begin{remark}\label{rem:convol}
We will often apply the above Lemma with $\varphi=\rho\ast g$, where  $\rho\in \mathcal C^2(\R^2)$  is such that $\rho$  and $\na^2 \rho$ have bounded second order moments in $L^2$, and  $g\in L^2_{uloc}(\R^2)$. In this case, we have
$$
\begin{aligned}
\|\varphi\|_{\infty}\leq C \|g\|_{L^2_{uloc}} \|(1+|\cdot|^2)\rho\|_{L^2(\R^2)} ,\\
\|\na^2\varphi\|_{\infty}\leq C \|g\|_{L^2_{uloc}} \|(1+|\cdot|^2)\na^2\rho\|_{L^2(\R^2)} .
\end{aligned}
$$
Indeed, 
\begin{align*}
\left\| \rho\ast {g}\right\|_{L^\infty}\leq&\sup_{x\in\mathbb R^2}\left(\int_{\R^2}\frac{1}{1+|x-y|^4} |g(y)|^2\:dy\right)^{1/2}\left(\int_{\R^2}(1+|x-y|^4)|\rho(x-y)|^2\:dy\right)^{1/2}\\
\leq&C\|g\|_{L^2_{uloc}}\|(1+|\cdot|^2)\rho\|_{L^2(\R^2)}.
\end{align*}
The $L^\infty$ norm of $\na^2\varphi$ is estimated exactly in the same manner, simply replacing $\rho$ by $\na^2 \rho$.

\end{remark}

\begin{proof}[Proof of Lemma \ref{lemauxorder1}]
We  split the integral in \eqref{reprnoyau1BF} into three parts
\begin{eqnarray}
\mathcal I \left[\frac{\xi_i\xi_j}{|\xi|}\right]\varphi(x)
&=&\int_{|x-y|\leq K}dy\gamma_{ij}(x-y)\left\{ \varphi(x)- \varphi(y)- \na\varphi(x)\cdot (x-y)\right\}\nonumber\\
&+&\int_{|x-y|\geq K}dy\gamma_{ij}(x-y)\varphi(x)\label{splitABC}\\
&-&\int_{|x-y|\geq K}dy\gamma_{ij}(x-y)\varphi(y)\nonumber\\
&=&\Aa(x)+\Bb(x)+\Cc(x).\nonumber
\end{eqnarray}

Concerning the first integral in \eqref{splitABC}, Taylor's formula implies
$$
\left|\Aa(x)\right|\leq C\left\|\nabla^2\varphi\right\|_{L^\infty}\int_{|x-y|\leq K}\frac{dy}{|x-y|}\leq CK\left\|\nabla^2\varphi\right\|_{L^\infty}.
$$
For the second and third integral in \eqref{splitABC},
$$
\left|\Bb(x)\right| + \left|\Cc(x)\right|\leq C \|\varphi\|_{\infty}\int_{|x-y|\geq K}\frac{dy}{|x-y|^3}
\leq CK^{-1} \|\varphi\|_{\infty}.
$$
We infer that for all $K>0$,
$$
\left\| \mathcal I\left[\frac{\xi_i\xi_j}{|\xi|}\right] \varphi\right\|_\infty \leq C\left(K\left\|\nabla^2\varphi\right\|_{\infty} + K^{-1} \|\varphi\|_{\infty}\right).
$$
Optimizing in $K$ (i.e. choosing $K=\|\varphi\|_\infty^{1/2}/\|\na^2\varphi\|_\infty^{1/2} $), we obtain the desired inequality.
\end{proof}

$\rhd$ For the remainder terms  $M_k^{rem}, N_k^{rem}$ as well as the high-frequency terms, we will use the following estimates:
\begin{lemma}[Kernel estimates]
Let $\phi\in\mathcal C^\infty_0(\R^2)$ such that $\phi(\xi)=1$ for $|\xi|\leq 1$. Define
$$
\begin{aligned}
\varphi_{HF}(x_h,x_3)&:=\mathcal{F}^{-1}\left(\sum_{k=1}^3(1-\phi)(\xi )L_k(\xi) \exp(-\lambda_k(\xi )x_3)\right),\\
\psi_1(x_h,x_3)&:=\mathcal{F}^{-1}\left(\sum_{k=1}^3\phi(\xi )M_k^{rem}(\xi)  \exp(-\lambda_k(\xi )x_3)\right),\\
\psi_2(x_h,x_3)&:=\mathcal{F}^{-1}\left(\sum_{k=1}^3\phi(\xi )N_k^{rem}(\xi) \exp(-\lambda_k(\xi )x_3)\right).
\end{aligned}
$$
Then the following estimates hold:
\begin{itemize}
\item for all $q\in\mathbb N$, there exists $c_{0,q}>0$, such that for all $\alpha,\beta>c_{0,q}$,  there exists $C_{\alpha,\beta,q}>0$ such that 
\begin{equation*}
|\na^q\varphi_{HF}(x_h,x_3)|\leq \frac{C_{\alpha,\beta,q}}{|x_h|^\alpha + |x_3|^\beta};
\end{equation*}
\item for all $\alpha\in (0,2/3)$, for all $q\in \N$, there $C_{\alpha,q}>0$ such that
\begin{equation*}
|\na^q \psi_1(x_h,x_3)|\leq \frac{C_{\alpha,q}}{|x_h|^{3+q} + |x_3|^{\alpha+\frac{q}{3}}};
\end{equation*}
\item for all $\alpha\in (0,2/3)$, for all $q\in \N$, there exists $C_{\alpha,q}>0$ such that
\begin{equation*}
|\na^q \psi_2(x_h,x_3)|\leq \frac{C_{\alpha,q}}{|x_h|^{3+q} + |x_3|^{\alpha+\frac{q}{3}}}.
\end{equation*}
\end{itemize}
\label{lem:kernels}
\end{lemma}
\begin{proof}

$\bullet$ Let us first derive the estimate on $\varphi_{HF}$ for $q=0$. We seek to prove that there exists $c_0>0$ such that
\be\label{est:varphi}\forall (\alpha, \beta)\in (c_0,\infty)^2,\ \exists C_{\alpha,\beta},\ |\varphi_{HF}(x_h,x_3)|\leq \frac{C_{\alpha,\beta}}{|x_h|^\alpha + |x_3|^\beta}.\ee
To that end, it is enough to show that for $\alpha\in \N^2$ and $\beta>0$ with $|\alpha|,\beta\geq c_0$,

$$
\sup_{x_3>0}\left(|x_3|^\beta\|\widehat{\varphi_{HF}}(\cdot, x_3)\|_{L^1(\R^2)} + \|\na^{\alpha}_\xi \widehat{\varphi_{HF}}(\cdot, x_3)\|_{L^1(\R^2)}\right)<\infty.
$$
We recall that $\lambda_k(\xi)\sim |\xi|$ for $|\xi|\to \infty$. Moreover, using the estimates of Lemma \ref{lem:dev-lambda-A}, we infer that there exists $\gamma\in \R$ such that $L_k(\xi)=O(|\xi|^\gamma)$ for $|\xi|\gg 1$. Hence
\begin{eqnarray*}
|x_3|^\beta|\hat \varphi_{HF}(\xi,x_3)|&\leq & C |1-\phi(\xi)| \:|\xi|^\gamma \sum_{k=1}^3|x_3|^\beta\exp(-\Re(\lambda_k) x_3)\\
&\leq &  C |1-\phi(\xi)| \:|\xi|^{\gamma-\beta} \sum_{k=1}^3|\Re(\lambda_k)x_3|^\beta\exp(-\Re(\lambda_k) x_3)\\
&\leq &C_\beta   \:|\xi|^{\gamma-\beta}\mathbf 1_{|\xi|\geq 1}.
\end{eqnarray*}
Hence for $\beta$ large enough, for all $x_3>0$,
$$
|x_3|^\beta\|\hat \varphi_{HF}(\cdot,x_3)\|_{L^1(\R^2)}\leq C_\beta.
$$
In a similar fashion, for $\alpha\in \N^2$, $|\alpha|\geq 1$, we have, as $|\xi|\to \infty$ (see  Appendix \ref{appendixrestes})
$$
\begin{aligned}
\na^\alpha L_k(\xi)&=O\left(|\xi|^{\gamma-|\alpha|}\right),\\
\na^\alpha\left(\exp(-\lambda_k x_3)\right)&=O\left((|\xi|^{1-|\alpha|}x_3 + |x_3|^{|\alpha|})\exp(-\Re(\lambda_k) x_3)\right)=O\left(|\xi|^{-|\alpha|}\right).
\end{aligned}
$$
Moreover, we recall that $\na(1-\phi)$ is supported in a ring of the type $B_R\setminus B_1$ for some $R>1$. As a consequence, we obtain, for all $\alpha\in \N^2$ with $|\alpha|\geq 1$,
$$
\left|\na^\alpha \widehat{ \varphi_{HF}}(\xi,x_3)\right|\leq C_\alpha |\xi|^{\gamma-|\alpha|}\mathbf 1_{|\xi|\geq 1},
$$
so that
$$\|\na^\alpha \widehat{ \varphi_{HF}}(\cdot,x_3)\|_{L^1(\R^2)}\leq C_\alpha.$$
Thus $\varphi_{HF}$ satisfies \eqref{est:varphi} for $q=0$. For $q\geq 1$, the proof is the same, changing $L_k$ into $|\xi|^{q_1}|\lambda_k|^{q_2}L_k$ with $q_1+q_2=q$.

$\bullet$ The estimates on $\psi_1,\psi_2$ are similar. The main difference lies in the degeneracy of $\lambda_1$ near zero. For instance, in order to derive an $L^\infty$ bound on $|x_3|^{\alpha + q/3}\na^q \psi_1$, we look for an $L^\infty_{x_3}(L^1_\xi(\R^2))$ bound on $|x_3|^{\alpha + q/3} |\xi|^q \hat\psi_1(\xi,x_3)$. We have
\begin{eqnarray*}
&&\left| |x_3|^{\alpha + q/3} |\xi|^q \phi(\xi )\sum_{k=1}^3M_k^{rem}\exp(-\lambda_k x_3)\right|\\
&\leq & C|x_3|^{\alpha + q/3} |\xi|^{q} \sum_{k=1}^3\exp(-\Re(\lambda_k) x_3) |M_{k}^{rem}|\mathbf{1_{|\xi|\leq R}}\\
&\leq & C|\xi|^{q} \sum_{k=1}^3 |\Re \lambda_k|^{-(\alpha+q/3)}|M_{k}^{rem}|\mathbf{1_{|\xi|\leq R}}\\
&\leq & C|\xi|^q (|\xi|^{1-3\alpha-q}+1)\mathbf{1_{|\xi|\leq R}}.
\end{eqnarray*}
The right-hand side is in $L^1$ provided $\alpha<2/3$. We infer that
$$
\left|\ |x_3|^{\alpha + q/3}\na^q \psi_1(x)\right| \leq C_{\alpha,q}\quad \forall x\ \forall \alpha\in (0,2/3).
$$
The other bound on $\psi_1$ is derived in  a similar way, using the fact that
$$
\na_\xi^q M_1^{rem}=O(|\xi|^{2-q}+1)
$$
for $\xi$ in a neighbourhood of zero.
\end{proof}

$\rhd$ We are now ready to state our representation formula:

\begin{proposition}[Representation formula]

Let $v_0\in H^{1/2}(\R^2)^3$ such that $v_{0,3}=\na_h\cdot V_h$ for some $V_h\in H^{1/2}(\R^2)$, and let $u$ be the solution of \eqref{SC-R3+}.
For all $x\in \R^3$, let $\chi\in \mathcal C^\infty_0(\R^2)$ such that $\chi\equiv 1$ on $B(x_h,1)$. Let $\phi\in\mathcal C^\infty_0(\R^2)$ be a cut-off function as in Lemma \ref{lem:kernels}, and let $\varphi_{HF}$, $\psi_1$, $\psi_2$ be the associated kernels. For $k=1,2,3$, set
$$
f_k(\cdot,x_3):=\mathcal F^{-1}\left( \phi(\xi) \exp(-\lambda_k x_3)\right).
$$

Then
\begin{eqnarray*}
u(x)&=&\mathcal{F}^{-1}\left(\sum_{k=1}^3 L_k(\xi) \begin{pmatrix}
\widehat{\chi v_{0,h}}(\xi)\\ \widehat{\na \cdot(\chi V_h)}
\end{pmatrix}
\exp(-\lambda_k x_3)\right)(x)\\
&+& \sum_{k=1}^3\mathcal{ I}[M_k^1]f_k(\cdot,x_3)\ast((1-\chi) v_{0,h})(x)\\
&+& \sum_{k=1}^3\mathcal{ I}[N_k^1]f_k(\cdot,x_3)\ast((1-\chi) V_{h})(x)\\
&+& \varphi_{HF}\ast \begin{pmatrix}
{(1-\chi) v_{0,h}}\\ {\na\cdot((1-\chi) V_h)}
\end{pmatrix}(x)\\
&+& \psi_1\ast ((1-\chi) v_{0,h})(x) + \psi_2\ast ((1-\chi) V_h)(x) 
\end{eqnarray*}
As a consequence,  for all $a>0$, there exists a constant $C_a$ such that
$$
\sup_{k\in \Z^2} \int_{k+[0,1]^2}\int_0^a|u(x_h,x_3)|^2dx_3\:dx_h\leq C_a \left(\|v_0\|^2_{H^{1/2}_{uloc}(\R^2)} + \|V_h\|^2_{H^{1/2}_{uloc}(\R^2)}\right).
$$
Moreover, there exists $q\in \N$ such that 
$$
\sup_{k\in \Z^2} \int_{k+[0,1]^2}\int_1^\infty |\na^q u(x_h,x_3)|^2dx_3\:dx_h\leq C\left(\|v_0\|^2_{H^{1/2}_{uloc}(\R^2)} + \|V_h\|^2_{H^{1/2}_{uloc}(\R^2)}\right).
$$

\label{prop:est-SC-uloc}
\end{proposition}
\begin{remark}
The integer $q$ in the above proposition is explicit and does not depend on $v_0$. One can take $q=4$ for instance.
\end{remark}

\begin{proof}
The proposition follows quite easily from the preceding lemmas. We have, according to Proposition \ref{prop:ex/uni-SC-H12}, 
\begin{eqnarray*}
u(x)&=&\mathcal{F}^{-1}\left(\sum_{k=1}^3 L_k(\xi) \begin{pmatrix}
\widehat{\chi v_{0,h}}(\xi)\\ \widehat{\na \cdot(\chi V_h)}(\xi)
\end{pmatrix}\exp(-\lambda_k x_3)\right)(x)\\
&+&\mathcal{F}^{-1}\left(\sum_{k=1}^3 L_k(\xi) \begin{pmatrix}
\widehat{(1-\chi) v_{0,h}}(\xi)\\ \widehat{\na \cdot((1-\chi) V_h)}(\xi)
\end{pmatrix}\exp(-\lambda_k x_3)\right)(x).
\end{eqnarray*}
In the latter term, the cut-off function $\phi$ is introduced, writing simply $1=1-\phi+\phi$. We have, for the high-frequency term,
\begin{eqnarray*}
&&\mathcal{F}^{-1}\left(\sum_{k=1}^3(1-\phi(\xi)) L_k(\xi) \begin{pmatrix}
\widehat{(1-\chi) v_{0,h}}(\xi)\\ \widehat{\na \cdot((1-\chi) V_h)}(\xi)
\end{pmatrix}\exp(-\lambda_k x_3)\right)\\
&=&\mathcal{F}^{-1}\left(\hat \varphi_{HF}(\xi,x_3)\begin{pmatrix}
\widehat{(1-\chi) v_{0,h}}(\xi)\\ \widehat{\na \cdot((1-\chi) V_h)}(\xi)
\end{pmatrix}\right)=\varphi_{HF}(\cdot,x_3)\ast\begin{pmatrix}
{(1-\chi) v_{0,h}}(\xi)\\ {\na \cdot((1-\chi) V_h)}(\xi)
\end{pmatrix}
\end{eqnarray*}
Notice that ${\na_h \cdot((1-\chi) V_h)}=(1-\chi) v_{0,3} -\na_h \chi\cdot V_h\in H^{1/2}(\R^2)$.    

In the low frequency terms, we distinguish between the horizontal and the vertical components of $v_0$. Let us deal with the vertical component, which is slightly more complicated: since $v_{0,3}=\na_h\cdot V_h$, we have
\begin{eqnarray*}
&&\mathcal{F}^{-1}\left(\sum_{k=1}^3\phi(\xi) L_k(\xi)e_3 \widehat{\na_h \cdot((1-\chi) V_h)}(\xi)\exp(-\lambda_k x_3)\right)\\
&=&\mathcal{F}^{-1}\left(\sum_{k=1}^3\phi(\xi) L_k(\xi)e_3 i\xi\cdot \widehat{(1-\chi)  V_h}(\xi)\exp(-\lambda_k x_3)\right).
\end{eqnarray*}
We recall that $N_k=iL_ke_3 \prescript{t}{}{\xi}$, so that
$$
L_k(\xi)e_3 i\xi\cdot \widehat{(1-\chi)  V_h}(\xi)=N_k(\xi)\widehat{(1-\chi)  V_h}(\xi).
$$
Then, by definition of $\psi_2$ and $f_k$,
\begin{eqnarray*}
&&\mathcal{F}^{-1}\left(\sum_{k=1}^3\phi(\xi) N_k(\xi)\widehat{(1-\chi)  V_h}(\xi)\exp(-\lambda_k x_3)\right)\\
&=&\mathcal{F}^{-1}\left(\sum_{k=1}^3\phi(\xi) N_k^1(\xi)\widehat{(1-\chi)  V_h}(\xi)\exp(-\lambda_k x_3)\right)\\
&&+\mathcal{F}^{-1}\left(\sum_{k=1}^3\phi(\xi) N_k^{rem}(\xi)\widehat{(1-\chi)  V_h}(\xi)\exp(-\lambda_k x_3)\right)\\
&=&\sum_{k=1}^3\mathcal I\left[N_k^1\right]f_k\ast((1-\chi) \cdot V_h)+\mathcal{F}^{-1}\left(\hat \psi_2(\xi,x_3)\widehat{(1-\chi) \cdot V_h}(\xi)\right)\\
&=&\sum_{k=1}^3\mathcal I\left[N_k^1\right]f_k\ast((1-\chi) \cdot V_h)+\psi_2\ast ((1-\chi) \cdot V_h).
\end{eqnarray*}
The representation formula follows.

There remains to bound every term occurring in the representation formula. In order to derive bounds on $(l+[0,1]^2)\times\R_+$ for some $l\in \Z^2$, we use the representation formula with a function $\chi_l\in \mathcal C^\infty_0(\R^2)$ such that $\chi_l\equiv 1$ on $l+[-1,2]^2$, and  we assume that the derivatives of $\chi_l$ are bounded uniformly in $l$ (take for instance $\chi_l=\chi(\cdot + l) $ for some $\chi\in \mathcal C^\infty_0$).
\begin{itemize}
\item According to Proposition \ref{prop:ex/uni-SC-H12}, we have
\begin{eqnarray*}
&&\int_0^a\left\| \mathcal{F}^{-1}\left(\sum_{k=1}^3 L_k(\xi) \begin{pmatrix}
\widehat{\chi_l v_{0,h}}(\xi)\\ \widehat{\na \cdot(\chi_l V_h)}
\end{pmatrix}
\exp(-\lambda_k x_3)\right)\right\|_{L^2(\R^2)}^2 dx_3\\
&\leq &C_a\left(\|\chi_l v_{0,h}\|_{H^{1/2}}^2 +\|\na\chi_l \cdot V_h\|_{H^{1/2}}^2  + \|\chi_l v_{0,3}\|^2_{H^{1/2}(\R^2)}\right).
\end{eqnarray*}
Using the formula
$$
\|f\|_{H^{1/2}(\R^2)}^2=\|f\|_{L^2}^2 + \int_{\R^2\times \R^2}\frac{|f(x)-f(y)|^2}{|x-y|^3}\:dx\:dy\quad \forall f\in H^{1/2}(\R^2),
$$
it can be easily proved that 
\be\label{est:H12-tronque}
\|\chi u\|_{H^{1/2}(\R^2)}\leq C \|\chi\|_{W^{1,\infty}} \|u\|_{H^{1/2}(\R^2)}
\ee
for all $\chi \in W^{1,\infty}(\R^2)$ and for all $u\in H^{1/2}(\R^2)$, where the constant $C$ only depends on the dimension. Therefore
\begin{eqnarray*}
\|\chi_l v_{0,h}\|_{H^{1/2}}&\leq & \sum_{k\in \Z^2}\|\chi_l \tau_k \vartheta v_{0,h}\|_{H^{1/2}}\\
&\leq & \sum_{k\in \Z^2, |k-l|\leq 1 + 3\sqrt{2}}\|\chi_l \tau_k \vartheta v_{0,h}\|_{H^{1/2}}\\
&\leq & C \|\chi_l\|_{W^{1,\infty}} \|v_{0,h}\|_{H^{1/2}_{uloc}},
\end{eqnarray*}
so that
$$\int_0^a\left\| \mathcal{F}^{-1}\left(\sum_{k=1}^3 L_k(\xi) \begin{pmatrix}
\widehat{\chi_l v_{0,h}}(\xi)\\ \widehat{\na \cdot(\chi_l V_h)}
\end{pmatrix}
\exp(-\lambda_k x_3)\right)\right\|_{L^2(\R^2)}^2 dx_3
\leq  C_a\left(\| v_{0}\|_{H^{1/2}_{uloc}}^2 +\| V_h\|_{H^{1/2}_{uloc}}^2\right).$$

Similarly,
\begin{eqnarray*}
&&\int_0^\infty\left\| \na \mathcal{F}^{-1}\left(\sum_{k=1}^3 L_k(\xi) \begin{pmatrix}
\widehat{\chi_l v_{0,h}}(\xi)\\ \widehat{\na \cdot(\chi_l V_h)}
\end{pmatrix}
\exp(-\lambda_k x_3)\right)\right\|_{L^2(\R^2)}^2 dx_3\\
&\leq & C\left(\| v_{0}\|_{H^{1/2}_{uloc}}^2 +\| V_h\|_{H^{1/2}_{uloc}}^2\right).
\end{eqnarray*}
Moreover, thanks to Remark \ref{rem:decay-classical}, for any $q\geq 2$, 
\begin{eqnarray*}
&&\int_1^\infty\left\| \na^q \mathcal{F}^{-1}\left(\sum_{k=1}^3 L_k(\xi) \begin{pmatrix}
\widehat{\chi_l v_{0,h}}(\xi)\\ \widehat{\na \cdot(\chi_l V_h)}
\end{pmatrix}
\exp(-\lambda_k x_3)\right)\right\|_{L^2(\R^2)}^2 dx_3\\
&\leq & C_q\left(\| v_{0}\|_{H^{1/2}_{uloc}}^2 +\| V_h\|_{H^{1/2}_{uloc}}^2\right).
\end{eqnarray*}

\item We now address the bounds of the terms involving the kernels $\varphi_{HF}, \psi_1,\psi_2$. According to Lemma \ref{lem:kernels}, we have for instance, for all $x_3>0$, for all $x_h\in l+[0,1]^2$, for $\sigma\in \N^2$,
\begin{eqnarray*}
&&\left|\int_{\R^2}\na^\sigma\varphi_{HF}(y_h,x_3) \begin{pmatrix}
{(1-\chi_l) v_{0,h}}\\ {\na\cdot((1-\chi_l) V_h)}
\end{pmatrix}(x_h-y_h)\:dy_h\right|\\ & \leq & C_{\alpha,\beta,|\sigma|}\int_{|y_h|\geq 1}| v_0(x_h-y_h) | \frac{1}{|y_h|^\alpha +x_3^\beta}\:dy_h\\&&+C_{\alpha,\beta,|\sigma|}\int_{1\leq |y_h|\leq 2}| V_h(x_h-y_h) | \frac{1}{|y_h|^\alpha +x_3^\beta}\:dy_h\\
&\leq & C\|V_h\|_{L^2_{uloc}}\frac{1}{1+x_3^\beta}+ C\left(\int_{\R^2}\frac{|v_0(x_h-y_h)|^2}{1+|y_h|^\gamma}dy_h\right)^{1/2}\left(\int_{|y_h|\geq 1}\frac{1+|y_h|^\gamma}{\left(|y_h|^\alpha +x_3^\beta\right)^2}\:dy_h\right)^{1/2}\\
&\leq&C\|V_h\|_{L^2_{uloc}}\frac{1}{1+x_3^\beta} +C\|v_0\|_{L^2_{uloc}}\inf\left(1,x_3^{\beta(\frac{2+\gamma}{2\alpha}-1)}\right)
\end{eqnarray*}
for all $\gamma>2$ and for $\alpha,\ \beta>c_0$ and sufficiently large. In particular the $\dot H^q_{uloc}$ bound follows. The local bounds in $L^2_{uloc}$ near $x_3=0$ are immediate since the right-hand side is uniformly bounded in $x_3$. The treatment of the terms with $\psi_1$, $\psi_2$ are analogous. Notice however that because of the slower decay of $\psi_1$, $\psi_2$ in $x_3$, we only have a uniform bound in $\dot H^q((l+[0,1]^2)\times (1,\infty))$ if $q$ is large enough ($q\geq 2$ is sufficient).

\item There remains to bound the terms involving $\mathcal I [M_k^1], \mathcal I[N_k^1]$, using Lemma \ref{lem:noyau-ordre1} and Remark \ref{rem:convol}. We have for instance, for all $x_3>0$,
\begin{eqnarray*}
&&\left\| \mathcal I[N_k^1] f_k\ast ((1-\chi_l) V_h)\right\|_{L^2(l+[0,1]^2)}\\
&\leq & C \|V_h\|_{L^2_{uloc}} \left( \|(1+|\cdot|^2)f_k(\cdot,x_3)\|_{L^2(\R^2)} + \|(1+|\cdot|^2)\na_h^2f_k(\cdot,x_3)\|_{L^2(\R^2)} \right).
\end{eqnarray*}
Using the Plancherel formula, we infer
\begin{eqnarray*}
\|(1+|\cdot|^2)f_k(\cdot,x_3)\|_{L^2(\R^2)} &
\leq&C\|\phi(\xi)\exp(-\lambda_k x_3)\|_{H^2(\R^2)}\\
&\leq & C\|\exp(-\lambda_k x_3)\|_{H^2(B_R)} + C\exp(-\mu x_3),
\end{eqnarray*}
where $R>1$ is such that $\Supp \phi\subset B_R$ and $\mu$ is a positive constant depending only on $\phi$. We have, for $k=1,2,3$, 
$$
\left| \na^2\exp(-\lambda_k x_3)\right|\leq C\left(x_3|\na_\xi^2\lambda_k| + x_3^2|\na_\xi \lambda_k|^2\right)\exp(-\lambda_k x_3).
$$
The asymptotic expansions in Lemma \ref{lem:dev-lambda-A} together with the results of Appendix \ref{appendixrestes} imply that for $\xi$ in any neighbourhood of zero,
\begin{eqnarray*}
\na^2 \lambda_1=O(|\xi|),& \na \lambda_1= O(|\xi|^2),\\
\na^2\lambda_k=O(1),& \na \lambda_k= O(|\xi|)\text{ for }k=2,3.&
\end{eqnarray*}
In particular, if $k=2,3$, since $\lambda_k$ is bounded away from zero in a neighbourhood of zero,
$$
\int_0^\infty dx_3\|\exp(-\lambda_k x_3)\|_{H^2(B_R)}^2<\infty.
$$
On the other hand, the degeneracy of $\lambda_1$ near $\xi=0$ prevents us from obtaining the same result. Notice however that
$$
\int_0^a \|\exp(-\lambda_1 x_3)\|_{H^2(B_R)}^2\leq C_a
$$
for all $a>0$, and
$$
\int_0^\infty \||\xi|^q\na^2\exp(-\lambda_1 x_3)\|_{L^2(B_R)}^2<\infty
$$
for $q\in \N$ large enough ($q\geq 4$). Hence the  bound on $\na^q u$ follows.\qedhere
\end{itemize}
\end{proof}

$\rhd$ The representation formula, together with its associated estimates, now allows us to extend the notion of solution to locally integrable boundary data. Before stating the corres\-ponding result, let us prove a technical lemma about some nice properties of 
operators of the type $\mathcal I\left[\frac{\xi_i\xi_j}{|\xi|}\right]$, which we will use repeatedly:
\begin{lemma}
Let $\varphi\in \mathcal C^\infty_0(\R^2)$.
Then, for all $g\in L^2_{uloc}(\R^2)$, for all  $\rho\in \mathcal C^\infty(\R^2)$  such that $\na^\alpha \rho$ has bounded second order moments in $L^2$ for $0\leq \alpha\leq 2$,
$$
\begin{aligned}
\int_{\R^2}\varphi \mathcal I\left[\frac{\xi_i\xi_j}{|\xi|}\right]\rho\ast g  =\int_{\R^2} g\mathcal I\left[\frac{\xi_i\xi_j}{|\xi|}\right]\check\rho\ast  \varphi,\\
\int_{\R^2} \na\varphi\mathcal I\left[\frac{\xi_i\xi_j}{|\xi|}\right]\rho\ast g  =-\int_{\R^2} \varphi\mathcal I\left[\frac{\xi_i\xi_j}{|\xi|}\right]\na\rho\ast g .
\end{aligned}
$$
\label{lem:propriete-I}
\end{lemma}

\begin{remark}
Notice that the second formula merely states that
$$
\na \left(\mathcal I\left[\frac{\xi_i\xi_j}{|\xi|}\right]\rho\ast g \right)=\mathcal I\left[\frac{\xi_i\xi_j}{|\xi|}\right]\na\rho\ast g 
$$
in the sense of distributions.
\end{remark}

\begin{proof}
$\bullet$ The first formula is a consequence of Fubini's theorem: indeed,
\begin{eqnarray*}
&&\int_{\R^2}\varphi\mathcal{I}\left[\frac{\xi_i\xi_j}{|\xi|}\right] \rho\ast g \\
&=&\int_{\R^6}dx\:dy\:dt\:\gamma_{ij}(x-y)g(t)\varphi(x)\times\\
&&\hskip3.5cm\times\left\{\rho(x-t)- \rho(y-t)- \na \rho(x-t)\cdot(x-y)\mathbf{1_{|x-y|\leq 1}}\right\}\\
&\underset{y'=x+t-y}{=}&\int_{\R^6}dx\:dy'\:dt \gamma_{ij}(y'-t) g (t)\varphi(x)\times\\
&&\hskip3.5cm\times\left\{\rho(x-t)- \rho(x-y')- \na \rho(x-t)\cdot(y'-t)\mathbf{1_{|y'-t|\leq 1}}\right\}.
\end{eqnarray*}
Integrating with respect to $x$, we obtain
\begin{eqnarray*}
&&\int_{\R^2}\varphi\mathcal{I}\left[\frac{\xi_i\xi_j}{|\xi|}\right] \rho\ast g \\
&=&\int_{\R^4}dy'\:dt\: \gamma_{ij}(y'-t)g(t)\left\{ \varphi \ast \check \rho (t) - \varphi \ast \check \rho (y')- \varphi \ast \na \check \rho(t)\cdot(t-y')\mathbf{1_{|y'-t|\leq 1}}\right\}\\
&=&\int_{\R^2}dt g(t)\mathcal{I}\left[\frac{\xi_i\xi_j}{|\xi|}\right] \varphi \ast \check{\rho}.
\end{eqnarray*}

$\bullet$ The second formula is then easily deduced from the first one: using the fact that $\na \check{\rho}(x)=-\na \rho(-x)=-\widecheck{\na\rho}(x),$ we infer
\begin{eqnarray*}
\int_{\R^2}\na\varphi  \mathcal I\left[\frac{\xi_i\xi_j}{|\xi|}\right]\rho\ast g 
&=&\int_{\R^2}g\mathcal I\left[\frac{\xi_i\xi_j}{|\xi|}\right] \check \rho \ast \na \varphi\\
&=&\int_{\R^2}g \mathcal I\left[\frac{\xi_i\xi_j}{|\xi|}\right] \na\check \rho \ast  \varphi\\
&=&-\int_{\R^2}g \mathcal I\left[\frac{\xi_i\xi_j}{|\xi|}\right] \widecheck{\na\rho} \ast  \varphi\\
&=&-\int_{\R^2}\varphi \mathcal I\left[\frac{\xi_i\xi_j}{|\xi|}\right]\na\rho\ast g .
\end{eqnarray*}
\end{proof}

We are now ready to state the main result of this section:
\begin{corollary}
Let $v_0\in \bK$ (recall that $\bK$ is defined in \eqref{def:K}.) Then there exists a unique solution $u$ of \eqref{SC-R3+} such that $u|_{x_3=0}=v_0$ and
\be\label{est:SC-existence}
\begin{aligned}
\forall a>0,&\quad \sup_{k\in \Z^2} \int_{k+[0,1]^2}\int_0^a|u(x_h,x_3)|^2dx_3\:dx_h<\infty,\\
\exists q \in \N^*,&\quad \sup_{k\in \Z^2} \int_{k+[0,1]^2}\int_1^\infty |\na^q u(x_h,x_3)|^2dx_3\:dx_h<\infty.
\end{aligned}
\ee

\label{cor:ex/uni-SC-uloc}
\end{corollary}

\begin{remark}
As in Proposition \ref{prop:est-SC-uloc}, the integer $q$ in the two results above is explicit and does not depend on $v_0$ (one can take $q=4$ for instance).
\end{remark}

\begin{proof}[Proof of Corollary \ref{cor:ex/uni-SC-uloc}]

\noindent \textit{Uniqueness.} Let $u$ be a solution of \eqref{SC-R3+} satisfying \eqref{est:SC-existence} and such that $u|_{x_3=0}=0$. We use the same type of proof as in Proposition \ref{prop:ex/uni-SC-H12} (see also \cite{DGVNMnoslip}). Using a Poincar\'e inequality near the boundary $x_3=0$, we have
$$
\sup_{k\in \Z^2} \int_{k+[0,1]^2}\int_0^\infty |\na^q u(x_h,x_3)|^2dx_3\:dx_h <\infty.
$$
Hence $u\in \mathcal C(\R_+, \mathcal S'(\R^2))$ and we can take the Fourier transform of $u$ with respect to the horizontal variable. The rest of the proof is identical to the one of Proposition \ref{prop:ex/uni-SC-H12}. The equations in \eqref{eq:SC-Fourier} are meant in the sense of tempered distributions in $x_h$, and in the sense of distributions in $x_3$, which is enough to perform all calculations.

\noindent \textit{Existence.} For all $x_h\in \R^2$, let $\chi\in \mathcal C^\infty_0(\R^2)$ such that $\chi\equiv 1$ on $B(x_h,1)$.
Then we set
\begin{eqnarray}
\nonumber u(x)&=&\mathcal{F}^{-1}\left(\sum_{k=1}^3 L_k(\xi) \begin{pmatrix}
\widehat{\chi v_{0,h}}(\xi)\\ \widehat{\na \cdot(\chi V_h)}
\end{pmatrix}
\exp(-\lambda_k x_3)\right)(x)\\
\nonumber&+& \sum_{k=1}^3\mathcal{ I}[M_k^1]f_k(\cdot,x_3)\ast((1-\chi) v_{0,h})(x)\\
\label{def:u:rep-form}&+& \sum_{k=1}^3\mathcal{ I}[N_k^1]f_k(\cdot,x_3)\ast((1-\chi) V_{h})(x)\\
\nonumber &+& \varphi_{HF}\ast \begin{pmatrix}
{(1-\chi) v_{0,h}}\\ {\na\cdot((1-\chi) V_h)}
\end{pmatrix}(x)\\
\nonumber &+& \psi_1\ast ((1-\chi) v_{0,h})(x) + \psi_2\ast ((1-\chi) V_h)(x) .
\end{eqnarray}

We first claim that this formula does not depend on the choice of the function $\chi$: indeed, let $\chi_1,\ \chi_2\in\mathcal C^\infty_0(\R^2)$ such that $\chi_i\equiv 1$ on $B(x_h,1)$. Then, since $\chi_1-\chi_2=0$ on $B(x_h,1)$ and $\chi_1-\chi_2$ is compactly supported, we may write
\begin{eqnarray*}
&&\sum_{k=1}^3\mathcal{ I}[M_k^1]f_k(\cdot,x_3)\ast((\chi_1-\chi_2) v_{0,h})+\psi_1\ast ((\chi_1-\chi_2) v_{0,h})\\
&=&\mathcal{F}^{-1}\left( \sum_{k=1}^3\phi(\xi)M_k\widehat{(\chi_1-\chi_2) v_{0,h}} \exp(-\lambda_k x_3)\right)
\end{eqnarray*}
and
\begin{eqnarray*}
&& \sum_{k=1}^3\mathcal{ I}[N_k^1]f_k(\cdot,x_3)\ast((\chi_1-\chi_2) V_{h})+\psi_2\ast ((\chi_1-\chi_2)V_h) \\
&=&\mathcal{F}^{-1}\left( \sum_{k=1}^3\phi(\xi)
N_k\widehat{(\chi_1-\chi_2) V_{h}} \exp(-\lambda_k x_3)\right)\\
&=&\mathcal{F}^{-1}\left( \sum_{k=1}^3\phi(\xi)
L_ke_3\mathcal F \left(\na \cdot(\chi_1-\chi_2) V_h\right) \exp(-\lambda_k x_3)\right).\end{eqnarray*}
On the other hand,
\begin{eqnarray*}
&&\varphi_{HF}\ast \begin{pmatrix}
{(\chi_1-\chi_2) v_{0,h}}\\ {\na\cdot((\chi_1-\chi_2) V_h)}
\end{pmatrix}\\
&=&\mathcal{F}^{-1}\left( \sum_{k=1}^3(1-\phi(\xi))L_k \begin{pmatrix}
\widehat{(\chi_1-\chi_2) v_{0,h}}\\ \widehat{\na\cdot((\chi_1-\chi_2) V_h)}
\end{pmatrix}\exp(-\lambda_k x_3)\right).
\end{eqnarray*}

Gathering all the terms, we find that the two definitions coincide. Moreover, $u$
satisfies \eqref{est:SC-existence} (we refer to the proof of Proposition \ref{prop:est-SC-uloc} for the derivation of such estimates: notice that the proof of Proposition \ref{prop:est-SC-uloc} only uses local integrability properties of $v_0$). 

There remains to prove that $u$ is a solution of the Stokes system, which is not completely trivial due to the complexity of the representation formula. We start by deriving a duality formula: we claim that for all $\eta\in \mathcal C^\infty_0(\R^2)^3$, for all $x_3>0$,
\begin{eqnarray}
\label{duality}
\int_{\R^2} u(x_h,x_3)\cdot \eta(x_h)\:dx_h
&=&\int_{\R^2}v_{0,h}(x_h) \cdot \mathcal F^{-1}\left(\sum_{k=1}^3 \left(\overline{\prescript{t}{}{L_k}} \hat \eta(\xi)\right)_h\exp(-\bar \lambda_k x_3)\right)\\
\nonumber&-&\int_{\R^2}V_h(x_h) \cdot \mathcal F^{-1}\left(\sum_{k=1}^3 i\xi\left(\overline{\prescript{t}{}{L_k}} \hat \eta(\xi)\right)_3\exp(-\bar \lambda_k x_3)\right).
\end{eqnarray}
To that end, in \eqref{def:u:rep-form}, we may choose a function $\chi\in \mathcal C^\infty_0(\R^2)$ such that $\chi\equiv 1$ on the set
$$
\{x\in \R^2, \ d(x,\Supp \eta)\leq 1\}.
$$
We then transform every term in \eqref{def:u:rep-form}. We have, according to the Parseval formula
\begin{eqnarray*}
&&\int_{\R^2}\mathcal{F}^{-1}\left(\sum_{k=1}^3 L_k(\xi)\begin{pmatrix}
\widehat{\chi v_{0,h}}(\xi)\\ \widehat{\na \cdot(\chi V_h)}(\xi)
\end{pmatrix}\exp(-\lambda_k x_3)\right) \cdot\eta\\
&=&\frac{1}{(2\pi)^2}\int_{\R^2}\sum_{k=1}^3\overline{\hat \eta(\xi )} \cdot L_k(\xi) \begin{pmatrix}
\widehat{\chi v_{0,h}}(\xi)\\ \widehat{\na \cdot(\chi V_h)}(\xi)
\end{pmatrix}\exp(-\lambda_k x_3)\:d\xi\\
&=&\int_{\R^2}\chi v_{0,h} \mathcal F^{-1}\left(\sum_{k=1}^3\left(\overline{\prescript{t}{}{L_k}} \hat \eta(\xi)\right)_h\exp(-\bar \lambda_k x_3)\right)\\
&-&\int_{\R^2}\chi V_h \cdot \mathcal F^{-1}\left(\sum_{k=1}^3 i\xi\left(\overline{\prescript{t}{}{L_k}} \hat \eta(\xi)\right)_3\exp(-\bar \lambda_k x_3)\right).
\end{eqnarray*}
Using standard convolution results, we have
$$\int_{\R^2} \psi_1\ast ((1-\chi) v_{0,h}) \eta=\int_{\R^2}  (1-\chi) v_{0,h} \prescript{t}{}{\check \psi_1}\ast \eta.
$$
The terms with $\psi_2$ and $\varphi_{HF}$ are transformed using identical computations. Concerning the term with $\mathcal I[M_k^1]$, we use Lemma \ref{lem:propriete-I}, from which we infer that
$$\int_{\R^2}\mathcal{I}\left[M_k^1\right] f_k\ast ((1-\chi) v_{0,h}) \eta=\int_{\R^2}(1-\chi) v_{0,h}\mathcal{I}\left[\prescript{t}{}{M_k^1} \right]\check f_k \ast \eta.
$$
Notice also that by definition of $M_k^1$, $\widecheck{M_k^1}= M_k^1$. Therefore
\begin{eqnarray*}
\int_{\R^2} \psi_1\ast ((1-\chi) v_{0,h}) \eta
&+&\sum_{k=1}^3\int_{\R^2}\mathcal{I}\left[M_k^1\right] f_k\ast ((1-\chi) v_{0,h}) \eta\\
&=&\int_{\R^2}(1-\chi) v_{0,h}\cdot \mathcal{F}^{-1}\left(\sum_{k=1}^3 \prescript{t}{}{\begin{pmatrix}
\widecheck L_ke_1 & \widecheck L_ke_2
\end{pmatrix}}\hat\eta \check\phi(\xi) \exp(-\check \lambda_k x_3)\right).
\end{eqnarray*}
and
\begin{eqnarray*}
&&\int_{\R^2} \psi_2\ast ((1-\chi) V_h \eta
+\sum_{k=1}^3\int_{\R^2}\mathcal{I}\left[N_k^1\right] f_k\ast ((1-\chi) V_h) \eta\\
&=&\int_{\R^2}(1-\chi) V_h\cdot \mathcal{F}^{-1}\left(\sum_{k=1}^3 \xi \prescript{t}{}{\begin{pmatrix}
i \widecheck L_k e_3
\end{pmatrix}}\hat\eta\check \phi(\xi) \exp(-\check \lambda_k x_3)\right).
\end{eqnarray*}
Now, we recall that if $v_0\in H^{1/2}(\R^2)\cap \bK$ is real-valued, then so is the solution $u$ of \eqref{SC-R3+}. Therefore, in Fourier space,
\begin{equation*}
\overline{\hat u(\cdot,x_3)}= \check{\hat u}(\cdot,x_3) \quad \forall x_3>0. 
\end{equation*}
We infer in particular that
$$
\sum_{k=1}^3 \check L_k \exp(-\check \lambda_k x_3)= \sum_{k=1}^3 \bar L_k \exp(-\bar \lambda_k x_3).
$$
Gathering all the terms, we obtain \eqref{duality}.

Now, let $\zeta\in\mathcal C^\infty_0(\R^2\times (0,\infty))^3$ such that $\na\cdot \zeta=0$, and $\eta\in\mathcal C^\infty_0(\R^2\times (0,\infty))$. We seek to prove that
\be\label{SC-weak}\int_{\R^3_+}u\left(-\Delta \zeta - e_3\times\zeta\right) =0
\ee
as well as
\be\label{div-free-weak}
\int_{\R^3_+}u\cdot \na \eta=0.
\ee
Using \eqref{duality}, we infer that
\begin{eqnarray*}
&&\int_{\R^3_+}u\left(-\Delta \zeta - e_3\times\zeta\right) \\
&=&\int_0^\infty \int_{\R^2} v_{0,h} \mathcal F^{-1}\left( \sum_{k=1}^3 \overline{\mathcal M_k(\xi)} \hat \zeta(\xi) \exp(-\bar \lambda_k x_3)\right)\\
&+& \int_0^\infty \int_{\R^2}V_h\mathcal F^{-1}\left( \sum_{k=1}^3 \overline{\mathcal N_k(\xi)} \hat \zeta(\xi) \exp(-\bar \lambda_k x_3)\right),
\end{eqnarray*}
where
$$
\begin{aligned}
\mathcal M_k:=(|\xi|^2-\lambda_k^2) \prescript{t}{}{M_k} +  \prescript{t}{}{M_k}\begin{pmatrix}
0&1&0\\-1&0&0\\0&0&0
\end{pmatrix},
\mathcal N_k:=(|\xi|^2-\lambda_k^2) \prescript{t}{}{N_k} +  \prescript{t}{}{N_k}\begin{pmatrix}
0&1&0\\-1&0&0\\0&0&0
\end{pmatrix}.
\end{aligned}
$$
According to Lemma \ref{lem:eq:L_k}, 
$$
\mathcal M_k=\begin{pmatrix}
i\xi_1 q_{k,1} & &i\xi_2 q_{k,1} & -\lambda_k q_{k,1}\\
i\xi_1 q_{k,2} & &i\xi_2 q_{k,2} & -\lambda_k q_{k,2}\end{pmatrix}$$
so that, since $i\xi\cdot \hat \zeta_h + \pa_3\hat \zeta_3=0$,
$$
\overline{\mathcal M_k(\xi)} \hat \zeta(\xi,x_3) =(\pa_3 \hat\zeta_3-\bar \lambda_k \hat \zeta_3)\begin{pmatrix} \bar q_{k,1}\\ \bar q_{k,2}
\end{pmatrix}.
$$
Integrating in $x_3$, we find that
$$
\int_0^\infty \overline{\mathcal M_k(\xi)} \hat \zeta(\xi,x_3)\exp(-\bar \lambda_k x_3)dx_3=0.
$$
Similar arguments lead to
$$
\int_0^\infty \int_{\R^2}V_h\mathcal F^{-1}\left( \sum_{k=1}^3 \overline{\mathcal N_k(\xi)} \hat \zeta(\xi,x_3) \exp(-\bar \lambda_k x_3)\right)=0$$
and to the divergence-free condition \eqref{div-free-weak}.
\end{proof}

\subsection{The Dirichlet to Neumann operator for the Stokes-Coriolis system}
\label{subsecDtoN}

We now define the Dirichlet to Neumann operator for the Stokes-Coriolis system with boundary data in $\bK$. We start by deriving its expression for a boundary data $v_0\in H^{1/2}(\R^2)$ satisfying \eqref{hyp:v03}, for which we 
 consider the unique solution $u$ of \eqref{SC-R3+} in $\dot H^1(\R^3_+)$. We recall that $u$ is defined in Fourier space by \eqref{def:u-Fourier}. The corresponding pressure term is given by
$$
\hat p(\xi, x_3)=\sum_{k=1}^3A_k(\xi) \frac{|\xi|^2-\lambda_k(\xi)^2}{\lambda_k(\xi)} \exp(-\lambda_k(\xi) x_3).
$$

The Dirichlet to Neumann operator is then defined by
$$
\DtoN v_0:=-\pa_3 u|_{x_3=0} + p|_{x_3=0} e_3.
$$
Consequently, in Fourier space, the Dirichlet to Neumann operator is given by
\be\label{DN-Fourier}
\widehat{\DtoN v_0}(\xi)=\sum_{k=1}^3A_k(\xi)\begin{pmatrix}
\frac{i}{|\xi|^2}(-\lambda_k^2\xi + (|\xi|^2-\lambda_k^2)^2\xi^\bot)\\
\frac{|\xi|^2}{\lambda_k}
\end{pmatrix}=:M_{SC}(\xi) \hat v_0(\xi),
\ee
where $M_{SC}\in \mathcal M_{3,3}(\C)$. Using the notations of the previous paragraph, we have
$$
M_{SC}=\sum_{k=1}^3 \lambda_k L_k + e_3\prescript{t}{}{q_k}.
$$
Let us first review a few useful properties of the Dirichlet to Neumann operator:
\begin{proposition}\label{prop:dev-MSC}\ 
\begin{itemize}
\item Behaviour at large frequencies: when $|\xi|\gg 1$,
$$
M_{SC}(\xi)=\begin{pmatrix}
\ds|\xi| + \frac{\xi_1^2}{|\xi|} & \ds\frac{\xi_1 \xi_2}{|\xi|} & i\xi_1\\
\ds \frac{\xi_1 \xi_2}{|\xi|} & |\xi| +\ds  \frac{\xi_2^2}{|\xi|} &i\xi_2\\
 -i\xi_1 & -i\xi_2 & 2 |\xi|
\end{pmatrix} + O(|\xi|^{1/3}).
$$

\item Behaviour at small frequencies: when $|\xi|\ll 1$,
$$
M_{SC}(\xi)=\frac{\sqrt{2}}{2}\begin{pmatrix}
1&-1&\ds\frac{i(\xi_1 + \xi_2)}{|\xi|}\\
1&1&\ds\frac{i(\xi_2-\xi_1)}{|\xi|}\\
\ds\frac{i(\xi_2-\xi_1)}{|\xi|} &\ds \frac{-i(\xi_1+\xi_2)}{|\xi|} & \ds\frac{\sqrt{2}}{|\xi|} -1
\end{pmatrix} + O(|\xi|).
$$

\item The horizontal part of the Dirichlet to Neumann operator, denoted by $\DtoN_h$, maps $H^{1/2}(\R^2)$ into $H^{-1/2}(\R^2)$.

\item Let $\phi\in \mathcal C^\infty_0(\R^2)$ such that $\phi(\xi)=1$ for $|\xi|\leq 1.$ Then
$$
\begin{aligned}
(1-\phi(D))\DtoN_3&:H^{1/2}(\R^2)\to H^{-1/2}(\R^2),\\
D\phi(D)\DtoN_3, |D|\phi(D)\DtoN_3&: L^2(\R^2)\to L^2(\R^2),
\end{aligned}
$$
where, classically, $a(D)$ denotes the operator defined in Fourier space by
$$
\widehat{a(D)u}= a(\xi)\hat u(\xi)
$$
for $a\in \mathcal C(\R^2)$, $u\in L^2(\R^2)$.

\end{itemize}
\end{proposition}

\begin{remark}
For $|\xi|\gg 1$, the Dirichlet to Neumann operator for the Stokes-Coriolis system has the same expression, at main order, as the one of the Stokes system. This can be easily understood since at large frequencies, the rotation term in the system \eqref{eq:SC-Fourier} can be neglected in front of $|\xi|^2 \hat u$, and therefore the system behaves roughly as the Stokes system.

\end{remark}

\begin{proof}
The first two points follow from the expression \eqref{DN-Fourier} together with the  asymptotic expansions in Lemma \ref{lem:dev-lambda-A}. Since they are lengthy but straightforward calculations, we postpone them to the Appendix \ref{appendixexp}.

The horizontal part of the Dirichlet to Neumann operator satisfies
$$
\begin{aligned}
|\widehat{\DtoN_h v_0}(\xi)|=O(|\xi| \: |\hat v_0(\xi)|)\quad \text{for }|\xi|\gg 1,\\
|\widehat{\DtoN_h v_0}(\xi)|=O(|\hat v_0(\xi)|)\quad \text{for }|\xi|\ll 1.
\end{aligned}
$$
Therefore, if $\int_{\R^2}(1+|\xi|^2)^{1/2}  |\hat v_0(\xi)|^2\:d\xi <\infty$, we deduce that
$$
\int_{\R^2}(1+|\xi|^2)^{-1/2}  |\widehat{\DtoN_h v_0}(\xi)|^2\:d\xi <\infty.
$$
Hence $\DtoN_h:H^{1/2}(\R^2)\to H^{-1/2}(\R^2)$.

In a similar way, 
$$
|\widehat{\DtoN_3 v_0}(\xi)|=O(|\xi| \: |\hat v_0(\xi)|)\quad \text{for }|\xi|\gg 1,
$$
so that if $\phi\in \mathcal C^\infty_0(\R^2)$ is such that $\phi(\xi)=1$ for $\xi$ in a neighbourhood of zero, there exists a constant $C$ such that
$$
\left|(1-\phi(\xi)) \widehat{\DtoN_3 v_0}(\xi)\right|\leq C |\xi| \: |\hat v_0(\xi)|\quad\forall \xi\in \R^2.
$$
Therefore $(1-\phi(D))\DtoN_3:H^{1/2}(\R^2)\to H^{-1/2}(\R^2)$.

The vertical part of the Dirichlet to Neumann operator, however, is singular at low frequencies. This is consistent with the singularity observed in $L_1(\xi)$ for $\xi$ close to zero. More precisely, for $\xi$ close to zero, we have
$$
\widehat{\DtoN_3 v_0}(\xi)= \frac{1}{|\xi|}\hv_{0,3} +O(  |\hat v_0(\xi)|).
$$
Consequently, for all $\xi\in\mathbb R^2$
\begin{equation*}
\left|\xi \phi(\xi)\widehat{\DtoN_3 v_0}(\xi)\right|\leq C {|\hat v_0(\xi)|}.\qedhere
\end{equation*}
\end{proof}

Following \cite{DGVNMnoslip}, we now extend the definition of the Dirichlet to Neumann operator to functions which are not square integrable in $\R^2$, but rather locally uniformly integrable. There are several differences with \cite{DGVNMnoslip}: first, the Fourier multiplier associated with $\DtoN$ is not homogeneous, even at the main order. Therefore its kernel (the inverse Fourier transform of the multiplier) is not homogeneous either, and, in general, does not have the same decay as the kernel of Stokes system. Moreover, the singular part of the Dirichlet to Neumann operator for low frequencies prevents us from defining $\DtoN$ on $H^{1/2}_{uloc}$. Hence we will define $\DtoN$ on $\bK$ only (see also Corollary \ref{cor:ex/uni-SC-uloc}).

Let us briefly recall the definition  of the Dirichlet to Neumann operator for the Stokes system (see \cite{DGVNMnoslip}), which we denote by $\DtoN_S$\footnote{In \cite{DGVNMnoslip}, D. G\'erard-Varet and N. Masmoudi consider the Stokes system in $\R^2_+$ and not $\R^3_+$, but this part of their proof does not depend on the dimension.}. The Fourier multiplier of $\DtoN_S$ is
$$
M_S(\xi):=\begin{pmatrix}
\ds|\xi| + \frac{\xi_1^2}{|\xi|} & \ds\frac{\xi_1 \xi_2}{|\xi|} & i\xi_1\\
\ds \frac{\xi_1 \xi_2}{|\xi|} &\ds |\xi| +  \frac{\xi_2^2}{|\xi|} &i\xi_2\\
 -i\xi_1 & -i\xi_2 & 2 |\xi|
\end{pmatrix}.
$$
The inverse Fourier transform  of $M_S$ in $\mathcal S'(\R^2)$ is homogeneous of order -3, and consists of two parts:
\begin{itemize}
\item One part which is the inverse Fourier transform of  coefficients equal to $i\xi_1$ or  $i\xi_2$. This part is singular, and is the derivative of a Dirac mass at point $t=0$. 

\item One kernel part, denoted by $K_S$, which satisfies
$$
|K_S(t)|\leq \frac{C}{|t|^3}.
$$
\end{itemize}
In particular, it is legitimate to say that
$$
\left|\mathcal F^{-1} M_S (t)\right|\leq \frac{C}{|t|^3}\quad \text{in }\mathcal D'(\R^2\setminus\{0\}).
$$

Hence $\DtoN_S$ is defined on $H^{1/2}_{uloc}$ in the following way: for all $\varphi\in \mathcal C^\infty_0(\R^2)$, let $\chi\in \mathcal C^\infty_0(\R^2)$ such that $\chi\equiv 1$ on the set $\{t\in \R^2,d(t,\Supp\varphi)\leq 1\}$. Then
$$
\langle \DtoN_S u, \varphi\rangle_{\mathcal D',\mathcal D}:=\langle\mathcal F^{-1}\left( M_S \widehat{\chi u}\right), \varphi\rangle_{H^{-1/2}, H^{1/2}}+ \int_{\R^2}K_S\ast ((1-\chi)u)\cdot \varphi.
$$
The assumption on $\chi$ ensures that there is no singularity in the last integral, while the decay of $K_S$ ensures its convergence. Notice also that the singular part (which is local in the physical space) is only present in the first term of the decomposition.

We wish to adopt a similar method here, but a few precautions must be taken because of the singularities at low frequencies, in the spirit of the representation formula \eqref{def:u:rep-form}.
Hence, before defining the action of $\DtoN$ on $\bK$, let us decompose the Fourier multiplier associated with $\DtoN$.
We have
$$
M_{SC}(\xi)= M_S(\xi) + \phi(\xi)(M_{SC}-M_S)(\xi) + (1-\phi)(\xi )(M_{SC}-M_S)(\xi).
$$
Concerning the third term, we have the following result, which is a straightforward consequence of Proposition \ref{prop:dev-MSC} and Appendix \ref{appendixrestes}:
\begin{lemma}
As $|\xi|\to \infty$, there holds
$$
\begin{aligned}
\na^\alpha_\xi (M_{SC}-M_S)(\xi)&=O\left(|\xi|^{\frac{1}{3}-|\alpha|}\right)
\end{aligned}
$$
for $\alpha\in \N^2$, $0\leq |\alpha|\leq 3$.

\label{lem:deriv3DN}
\end{lemma}

We deduce from Lemma \ref{lem:deriv3DN} that $\na^\alpha \left[(1-\phi(\xi))(M_{SC}-M_S)(\xi)\right]\in L^1(\R^2)$ for all $\alpha\in \N^2$ with $|\alpha|=3$, so that it follows from lemma \ref{lemappendixB1} that there exists a constant $C>0$ such that
$$
\left| \mathcal F^{-1} \left[(1-\phi(\xi))(M_{SC}-M_S)(\xi)\right](t)\right|\leq \frac{C}{|t|^3}.
$$

There remains to decompose $\phi(\xi)(M_{SC}-M_S)(\xi)$. As in Proposition \ref{prop:est-SC-uloc}, the multipliers which are homogeneous of order one near $\xi=0$ are treated separately. Note that since the last column and the last line of $M_{SC}$ act on horizontal divergences (see Proposition \ref{prop:def-DN}), we are interested in multipliers homogeneous of order zero in $M_{SC,3i}, M_{SC,i3}$ for $i=1,2$, and homogeneous of order $-1$ in $M_{SC,33}$. In the following, we set
$$
\begin{aligned}
\bar M_h:=\frac{\sqrt{2}}{2}\begin{pmatrix}
1&-1\\
1&1
\end{pmatrix} ,\quad \bar M:=\begin{pmatrix}
\bar M_h&0\\0&0
\end{pmatrix},\\
V_1:=\frac{i\sqrt{2}}{2|\xi|}\begin{pmatrix}
\xi_1 +\xi_2\\\xi_1-\xi_2
\end{pmatrix},\quad V_2:=\frac{i\sqrt{2}}{2|\xi|}\begin{pmatrix}
-\xi_1 +\xi_2\\-\xi_1-\xi_2
\end{pmatrix}.
\end{aligned}
$$
We decompose $M_{SC}-M_S$ near $\xi=0$ as 
$$
\phi(\xi)(M_{SC}-M_S)(\xi)= \bar M + \phi(\xi)\begin{pmatrix} M_1&V_1\\ \prescript{t}{}{V_2}& |\xi|^{-1}
\end{pmatrix}-(1-\phi(\xi))\bar M + \phi(\xi)M^{rem},
$$
where $M_1\in \mathcal M_2(\C)$ only contains homogeneous and nonpolynomial terms of order one, and $M^{rem}_{ij}$ contains either polynomial terms or remainder terms which are $o(|\xi|)$ for $\xi$ close to zero if $1\leq i,j\leq 2$. Looking closely at the expansions for $\lambda_k$ in a neighbourhood of zero (see \eqref{dvptlambdakBF}) and at the calculations in paragraph \ref{sssec:LFDN}, we infer that $M^{rem}_{ij}$ contains either polynomial terms or remainder terms of order $O(|\xi|^2)$ if $1\leq i,j\leq 2$. We emphasize that the precise expression of $M^{rem}$ is not needed in the following, although it can be computed by pushing forward the expansions of Appendix \ref{appendixexp}. In a similar fashion, $M^{rem}_{i,3}$ and $M^{rem}_{3,i}$ contain constant terms and remainder terms 
of order $O(|\xi|)$ for $i=1,2$, $M^{rem}_{3,3}$ contains remainder terms of order $O(1)$.
 As a consequence, if we define the low-frequency kernels $K^{rem}_i:\R^2\to \mathcal M_{2}(\C)$ for $1\leq i\leq 4$ by
$$
\begin{aligned}
K^{rem}_1&:=\mathcal F^{-1} \left(\phi\begin{pmatrix}
M^{rem}_{11}&M^{rem}_{12}\\M^{rem}_{21}&M^{rem}_{22}
\end{pmatrix}\right),\\
K^{rem}_2&:=\mathcal F^{-1}\left(\phi\begin{pmatrix}
M^{rem}_{13}\\M^{rem}_{23}
\end{pmatrix}i\begin{pmatrix}
\xi_1 &\xi_2
\end{pmatrix}\right),\\
K^{rem}_3&:=\mathcal F^{-1}\left(-i\phi (\xi)\xi \begin{pmatrix}
M^{rem}_{31}&M^{rem}_{32}
\end{pmatrix}\right),\\
K^{rem}_4&:=\mathcal F^{-1}\left(\phi(\xi )M^{rem}_{33}\begin{pmatrix}
\xi_1^2&\xi_1\xi_2\\\xi_1\xi_2 &\xi_2^2
\end{pmatrix}\right)
\end{aligned}
$$
we have, for $1\leq i \leq 4$ (see Lemmas \ref{lem:restes} and \ref{lem:restes-bis})
$$
|K^{rem}_i(x_h)|\leq \frac{C}{|x_h|^3}\quad \forall x_h\in \R^2.
$$
We also denote by $M^{rem}_{HF}$ the kernel part of
$$
\mathcal F^{-1}\left(-(1-\phi)\bar M  + (1-\phi)(M_{SC}-M_S)\right),
$$
which  satisfies
$$
|M^{rem}_{HF}(x_h)|\leq \frac{C}{|x_h|^3}\quad \forall x_h\in \R^2\setminus\{0\}.
$$
Notice that there is also a singular part in $\mathcal F^{-1}\left(-(1-\phi)\bar M\right)$, which corresponds in fact to $\mathcal F^{-1}(-\bar M)$, and which is therefore a Dirac mass at $x_h=0$.

There remains to define the kernels homogeneous of order one besides $M_1$. We set
$$
\begin{aligned}
M_2&:= V_1 i\begin{pmatrix}
\xi_1 &\xi_2
\end{pmatrix},\\
M_3&:=-i\xi \prescript{t}{}{V_2},\\
M_4&:=\frac{1}{|\xi|}\begin{pmatrix}
\xi_1^2&\xi_1\xi_2\\\xi_1\xi_2&\xi_2^2
\end{pmatrix},
\end{aligned}
$$
so that $M_1,M_2,M_3,M_4$ are $2\times 2$ real valued matrices whose coefficients are linear combinations of $\frac{\xi_i \xi_j}{|\xi|}$.
In the end, we will work with the following decomposition for the matrix $M_{SC}$, where the treatment of each of the terms has been explained above:
$$
M_{SC}= M_S + \bar M + (1-\phi)(M_{SC}- M_S - \bar M)+ \phi \begin{pmatrix}
M_1 & V_1 \\ \prescript{t}{}{V_2}& |\xi|^{-1}
\end{pmatrix}
+ \phi M^{rem} .
$$

We are now ready to extend the definition of the Dirichlet to Neumann operator to functions in $\bK$: in the spirit of Proposition \ref{prop:est-SC-uloc}-Corollary \ref{cor:ex/uni-SC-uloc}, we derive a representation formula for functions in $\bK\cap H^{1/2}(\R^2)^3$, which still makes sense for functions in $\bK$:

\begin{proposition}
Let $\varphi\in \mathcal C^\infty_0\left(\R^2\right)^3$ such that $\varphi_3=\na_h\cdot \Phi_h$ for some $\Phi_h\in \mathcal C^\infty_0\left(\R^2\right)$. Let $\chi\in \mathcal C^\infty_0(\R^2)$ such that $\chi\equiv 1$ on the set
$$
\{x\in \R^2,\ d(x,\Supp \varphi \cup \Supp \Phi_h)\leq 1\}.
$$
Let $\phi\in\mathcal C^\infty_0(\R^2_\xi)$ such that $\phi(\xi)=1$ if $|\xi|\leq 1$, and let $\rho:=\mathcal F^{-1}\phi$.

\begin{itemize}
\item Let $v_0\in H^{1/2}(\R^2)^3 $ such that $v_{0,3}=\na_h\cdot V_h$. Then
\begin{eqnarray*}
\left\langle \DtoN(v_0),\varphi\right\rangle_{\mathcal D',\mathcal D}&=&\left\langle \DtoN_S(v_0),\varphi\right\rangle_{\mathcal D',\mathcal D} +\int_{\R^2}\varphi\cdot \bar M v_0\\
&+&\left\langle \mathcal{F}^{-1}\left((1-\phi)\left(M_{SC}-M_S -\bar M\right)\widehat{\chi v_0}\right),\varphi\right\rangle_{H^{-1/2}, H^{1/2}}\\
&+& \int_{\R^2}\varphi\cdot M^{rem}_{HF}\ast ((1-\chi )v_0)\\
&+&\left\langle\mathcal F^{-1}\left( \phi \left(M^{rem} + \begin{pmatrix}
M_1 & V_1\\\prescript{t}{}{V_2}& |\xi|^{-1}
\end{pmatrix}\right)\begin{pmatrix}
\widehat{\chi v_{0,h}}\\ i\xi\cdot \widehat{\chi V_h}
\end{pmatrix} \right) ,\varphi\right\rangle_{H^{-1/2}, H^{1/2}}\\
&+& \int_{\R^2} \varphi_h\cdot \left\{\mathcal I[M_1] (\rho \ast (1-\chi) v_{0,h}) + K^{rem}_1\ast ((1-\chi) v_{0,h}) \right\} \\
&+& \int_{\R^2} \varphi_h\cdot \left\{\mathcal I[M_2] (\rho \ast (1-\chi) V_h) + K^{rem}_2\ast ((1-\chi) V_h) \right\} \\
&+& \int_{\R^2} \Phi_h\cdot \left\{\mathcal I[M_3] (\rho \ast (1-\chi) v_{0,h}) + K^{rem}_3\ast ((1-\chi) v_{0,h}) \right\} \\
&+& \int_{\R^2} \Phi_h\cdot \left\{\mathcal I[M_4] (\rho \ast (1-\chi) V_h) + K^{rem}_4\ast ((1-\chi) V_h) \right\} .
\end{eqnarray*}

\item The above formula still makes sense when $v_0\in \bK$, which allows us to extend the definition of $\DtoN$ to $\bK$.

\end{itemize}

\label{prop:def-DN}
\end{proposition}

\begin{remark}
Notice that if $v_0\in \bK$ and $\varphi \in \bK$ with $\varphi_3=\na_h\cdot\Phi_h$, and if $\varphi, \Phi_h$ have compact support, then the right-hand side of the formula in Proposition \ref{prop:def-DN} still makes sense. Therefore $\DtoN v_0$ can be extended into a linear form on the set of functions in $\bK$ with compact support. In this case, we will denote it by
$$
\langle \DtoN(v_0),\varphi \rangle,
$$
without specifying the functional spaces.
\end{remark}

The proof of the  Proposition \ref{prop:def-DN} is very close to the one of Proposition \ref{prop:est-SC-uloc} and Corollary \ref{cor:ex/uni-SC-uloc}, and therefore we leave it to the reader.

The goal is now to link the solution of the Stokes-Coriolis system in $\R^3_+$ with $v_0\in \bK$ and $\DtoN(v_0)$. This is done through the following lemma:

\begin{lemma}
Let $v_0\in \bK$, and let $u$ be the unique solution of \eqref{SC-R3+} with $u|_{x_3=0}=v_0$, given by Corollary \ref{cor:ex/uni-SC-uloc}. 

Let $\varphi\in \mathcal C^\infty_0(\bar \R^3_+)^3$ such that $\na \cdot \varphi=0$. Then
$$
\int_{\R^3_+}\na u \cdot \na \varphi + \int_{\R^3_+} e_3\times u\cdot \varphi = \langle \DtoN(v_0),\varphi|_{x_3=0}\rangle.
$$
In particular, if $v_0\in \bK$ with $v_{0,3}=\na_h\cdot V_h$ and if $v_0,V_h$ have compact support, then
$$
\langle \DtoN(v_0),v_0\rangle\geq 0.
$$

\label{lem:lien-SC/DN}

\end{lemma}
\begin{remark}
If $\varphi\in \mathcal C^\infty_0\left(\overline{\R^3_+}\right)^3$ is such that $\na \cdot \varphi=0$, then in particular
\begin{eqnarray*}
\varphi_{3|x_3=0}(x_h)&=&-\int_0^\infty\pa_3 \varphi_3(x_h,z)\:dz\\
&=&\int_0^\infty \na_h\cdot \varphi_h(x_h,z)=\na_h \cdot \Phi_h
\end{eqnarray*}
for  $\Phi_h:=\int_0^\infty\varphi_h(\cdot, z)dz\in\mathcal C^\infty_0(\R^2)$. In particular $\varphi|_{x_3=0}$ is a suitable test function for Proposition \ref{prop:def-DN}.
\end{remark}

\begin{proof}
The proof relies on two duality formulas in the spirit of \eqref{duality}, one for the Stokes-Coriolis system and the other for the Dirichlet to Neumann operator. We claim that if $v_0\in \bK$, then on the one hand
\be\label{duality-system-SC}
\int_{\R^3_+}\na u \cdot \na \varphi + \int_{\R^3_+} e_3\times u\cdot \varphi=\int_{\R^2} v_0\mathcal F^{-1}\left(\prescript{t}{}{\bar M_{SC}(\xi)} \hat \varphi|_{x_3=0}(\xi)\right)
\ee
and on the other hand, for any $\eta\in \mathcal C^\infty_0(\R^2)^3$ such that $\eta_3=\na_h\cdot\theta_h$ for some $\theta_h\in \mathcal C^\infty_0(\R^2)^2$,
\be\label{duality-DN}
\langle \DtoN(v_0),\eta\rangle_{\mathcal D',\mathcal D}=\int_{\R^2} v_0 \mathcal F^{-1}\left( \prescript{t}{}{\bar M_{SC}(\xi)} {\hat \eta(\xi)}\right).\ee
Applying formula \eqref{duality-DN} with $\eta=\varphi|_{x_3=0}$ then yields the desired result. Once again, the proofs of \eqref{duality-system-SC}, \eqref{duality-DN} are close to the one of \eqref{duality}. From \eqref{duality}, one has 
\begin{eqnarray*}
\int_{\R^3_+} e_3\times u \cdot \varphi&=&- \int_{\R^3_+}  u \cdot e_3\times\varphi\\
&=&-\int_{\R^2} v_0 \mathcal F^{-1}\left(\int_0^\infty \sum_{k=1}^3\exp(-\bar \lambda_k x_3)\prescript{t}{}{\bar L_k}  e_3\times{\hat \varphi}\right)\\
&=&\int_{\R^2} v_0\mathcal F^{-1}\left(\int_0^\infty \sum_{k=1}^3\exp(-\bar \lambda_k x_3)\prescript{t}{}{\bar L_k}\begin{pmatrix}
0&1&0\\-1&0&0\\0&0&0
\end{pmatrix}{\hat \varphi}\right).
\end{eqnarray*}
Moreover, we deduce from the representation formula for $u$ and from Lemma \ref{lem:propriete-I} a representation formula for $\na u$: we have
\begin{eqnarray*}
\na u(x)&=&\mathcal{F}^{-1}\left(\sum_{k=1}^3\exp(-\lambda_k x_3) L_k(\xi) \begin{pmatrix}
\widehat{\chi v_{0,h}}\\ \widehat{\na \cdot(\chi V_h)}
\end{pmatrix}\begin{pmatrix}
i\xi_1 & i\xi_2 & -\lambda_k
\end{pmatrix}
\right)(x)\\
&+& \sum_{k=1}^3\mathcal{ I}[M_k^1]\na f_k(\cdot,x_3)\ast((1-\chi) v_{0,h})(x)\\
&+& \sum_{k=1}^3\mathcal{ I}[N_k^1]\na f_k(\cdot,x_3)\ast((1-\chi) V_{h})(x)\\
&+& \na \varphi_{HF}\ast \begin{pmatrix}
{(1-\chi) v_{0,h}}(\xi)\\ {\na\cdot((1-\chi) V_h)}
\end{pmatrix}\\
&+& \na\psi_1\ast ((1-\chi) v_{0,h})(x) + \na\psi_2\ast ((1-\chi) V_h)(x).
\end{eqnarray*}
Then, proceeding exactly as in the proof of Corollary \ref{cor:ex/uni-SC-uloc}, we infer that
\begin{eqnarray*}
\int_{\R^3_+}\na u \cdot \na \varphi 
&=&\int_{\R^2} v_0\mathcal{F}^{-1}\left( \sum_{k=1}^3\int_0^\infty|\xi|^2 \exp(-\bar \lambda_k x_3)\prescript{t}{}{\bar L_k}{\hat \varphi(\xi,x_3) dx_3}\right)\\
&-&\int_{\R^2} v_0\mathcal{F}^{-1}\left( \sum_{k=1}^3\int_0^\infty\bar \lambda_k \exp(-\bar \lambda_k x_3)\prescript{t}{}{\bar L_k}{\pa_3\hat\varphi(\xi,x_3) dx_3}\right).
\end{eqnarray*}
Integrating by parts in $x_3$, we obtain
$$\int_0^\infty \exp(-\bar \lambda_k x_3)\prescript{t}{}{\bar L_k}{\pa_3\hat\varphi(\xi,x_3) dx_3}
=\bar \lambda_k\int_0^\infty \exp(-\bar \lambda_k x_3)\prescript{t}{}{\bar L_k}{\hat \varphi(\xi,x_3) dx_3}
-\prescript{t}{}{\bar L_k}{\hat \varphi|_{x_3=0}(\xi)}.
$$
Gathering the terms, we infer
\begin{eqnarray*}
\int_{\R^3_+}\na u \cdot \na \varphi + \int_{\R^3_+} e_3\times u\cdot \varphi
&=&\int_{\R^2} v_0\mathcal F^{-1}\left(\int_0^\infty\sum_{k=1}^3  \exp(-\bar \lambda_k x_3)\prescript{t}{}{\bar P_k} \hat \varphi\right)\\
&+&\int_{\R^2} v_0\mathcal F^{-1}\left(\sum_{k=1}^3 \bar \lambda_k \prescript{t}{}{\bar L_k} \hat \varphi|_{x_3=0} \right),
\end{eqnarray*}
where
\begin{eqnarray*}
P_k&:=&(|\xi|^2-\lambda_k^2)L_k + \begin{pmatrix}
0&-1&0\\1&0&0\\0&0&0
\end{pmatrix}L_k\\
&=&-\begin{pmatrix}
i\xi_1\\i\xi_2\\-\lambda_k
\end{pmatrix}
\begin{pmatrix}
q_{k,1}&q_{k,2}&q_{k,3}
\end{pmatrix}
\end{eqnarray*}
according to Lemma \ref{lem:eq:L_k}. Therefore, since $\varphi$ is divergence-free, we have
$$
\prescript{t}{}{\bar P_k} \hat \varphi=(-\pa_3 \hat\varphi_3 + \bar \lambda_k \hat \varphi_3)\begin{pmatrix}
\bar q_{k,1}\\ \bar q_{k,2}\\ \bar q_{k,3}
\end{pmatrix},
$$
so that eventually, after integrating by parts once more in $x_3$,
\begin{eqnarray*}
&&\int_{\R^3_+}\na u \cdot \na \varphi + \int_{\R^3_+} e_3\times u\cdot \varphi\\
&=&\int_{\R^2} v_0\mathcal F^{-1}\left(\left[\sum_{k=1}^3 \bar \lambda_k \prescript{t}{}{\bar L_k} +\begin{pmatrix}
\bar q_{k,1}\\ \bar q_{k,2}\\ \bar q_{k,3}
\end{pmatrix}\prescript{t}{}{e_3} \right] \hat \varphi|_{x_3=0} \right)\\
&=&\int_{\R^2} v_0\mathcal F^{-1}\left(\prescript{t}{}{\bar M_{SC}} \hat \varphi|_{x_3=0}\right).
\end{eqnarray*}
The derivation of \eqref{duality-DN} is very similar to the one of \eqref{duality} and therefore we skip its proof.
\end{proof}

We conclude this paragraph with some estimates on the Dirichlet to Neumann operator:
\begin{lemma}
\label{lem:est-DN}
There exists a positive constant $C$ such that the following property holds.\\
Let $\varphi\in \mathcal C^\infty_0(\R^2)^3$ such that $\varphi_3=\na_h\cdot \Phi_h$ for some $\Phi_h\in \mathcal C^\infty_0(\R^2)$, and let $v_0\in \bK$ with $v_{0,3}=\na_h\cdot V_h$. Let $R\geq 1$ and $x_0\in \R^2$ such that 
$$
\supp \varphi \cup \supp \Phi_h \subset B(x_0,R).
$$
Then
$$\left|\langle \DtoN(v_0),\varphi\rangle_{\mathcal D',\mathcal D} \right|
\leq C R\left(\|\varphi\|_{H^{1/2}(\R^2)}+\|\Phi_h\|_{H^{1/2}(\R^2)}\right)\left(\|v_0\|_{H^{1/2}_{uloc}} +\|V_h\|_{H^{1/2}_{uloc}}\right).
$$
Moreover, if $v_0,V_h\in H^{1/2}(\R^2)$, then
$$\left|\langle \DtoN(v_0),\varphi\rangle_{\mathcal D',\mathcal D} \right|
\leq C \left(\|\varphi\|_{H^{1/2}(\R^2)}+\|\Phi_h\|_{H^{1/2}(\R^2)}\right)\bigl(\|v_0\|_{H^{1/2}} +\|V_h\|_{H^{1/2}}\bigr).
$$
\end{lemma}

\begin{proof}
The second inequality is classical and follows from the Fourier definition of the Dirichlet to Neumann operator. We therefore focus on the first inequality, for which we use the representation formula of Proposition \ref{prop:def-DN}.

We consider a truncation function $\chi$ such that $\chi\equiv 1$ on $B(x_0,R+1)$ and $\chi\equiv 0$ on $B(x_0,R+2)^c$ , and such that $\|\na^\alpha \chi\|_\infty\leq C_\alpha$, with $C_\alpha$ independent of $R$, for all $\alpha\in \N$. We must evaluate three different types of term:

\noindent$\triangleright$ Terms of the type
$$
\int_{\R^2}K\ast ((1-\chi) v_0)\cdot \varphi,
$$
where $K$ is a matrix such that $|K(x)|\leq C|x|^{-3}$ for all $x\in \R^2$ (of course, we include in the present discussion all the variants involving $V_h$ and $\Phi_h$). These terms are bounded by
\begin{eqnarray*}
&&C\int_{\R^2\times \R^2}\frac{1}{|t|^3}|1-\chi(x-t)|\: | v_0(x-t)|\: |\varphi(x)|\:dx\:dt\\
&\leq & C \int_{\R^2}dx\: |\varphi(x)|\:\left(\int_{|t|\geq 1}\frac{| v_0(x-t)|^2}{|t|^3}dt\right)^{1/2}\left(\int_{|t|\geq 1}\frac{1}{|t|^3}dt\right)^{1/2}\\
&\leq & C\|v_0\|_{L^2_{uloc}}\|\varphi\|_{L^1}\\
&\leq & C R \|v_0\|_{L^2_{uloc}}\|\varphi\|_{L^2}.
\end{eqnarray*}

\noindent$\triangleright$ Terms of the type
$$
\int_{\R^2}\varphi_h\cdot \mathcal{I}[M]((1-\chi)v_{0,h})\ast \rho,
$$
where $M$ is a $2\times 2$ matrix whose coefficients are linear combinations of $\xi_i\xi_j/|\xi|$. Using Lemma \ref{lemauxorder1} and  Remark \ref{rem:convol}, these terms are bounded by
$$
C\|\varphi\|_{L^1}\|v_0\|_{L^2_{uloc}}\|(1+|\cdot|^2)\rho\|_{L^2}^{1/2}\|(1+|\cdot|^2)\na^2\rho\|_{L^2}^{1/2}.
$$
Using Plancherel's Theorem, we have (up to a factor $2\pi$)
$$
\begin{aligned}
\|(1+|\cdot|^2)\rho\|_{L^2}&=\|(1-\Delta)\phi\|_{L^2(\R^2)}\leq C,\\
\|(1+|\cdot|^2)\na^2\rho\|_{L^2}&=\|(1-\Delta)|\cdot|^2\phi\|_{L^2(\R^2)}\leq C,
\end{aligned}
$$
so that eventually
$$\left|\int_{\R^2}\varphi_h\cdot \mathcal{I}[M]((1-\chi)v_{0,h})\ast \rho\right|\leq C \|\varphi\|_{L^1}\|v_0\|_{L^2_{uloc}}\leq C R \|v_0\|_{L^2_{uloc}}\|\varphi\|_{L^2}.
$$

\noindent$\triangleright$ Terms of the type
$$
\langle \mathcal F^{-1}\left(M(\xi)\widehat{\chi v_0}(\xi)\right), \varphi\rangle_{H^{-1/2}, H^{1/2}}\text{ and } \int_{\R^2}\varphi\cdot \bar M v_{0}
$$
where $M(\xi)$ is some kernel such that $\mathrm{Op}(M):H^{1/2}(\R^2)\to H^{-1/2}(\R^2)$ and $\bar M$ is a constant matrix.

All these terms are bounded by
$$
C\|\chi v_0\|_{H^{1/2}(\R^2)}\|\varphi\|_{H^{1/2}(\R^2)}.
$$
In fact, the trickiest part of the Lemma is to prove that
\be\label{est:u-tronque-H12}
\|\chi v_0\|_{H^{1/2}(\R^2)}\leq C R \|v_0\|_{H^{1/2}_{uloc}}.
\ee
To that end, we recall that
$$
\|\chi v_0\|_{H^{1/2}(\R^2)}^2=\|\chi v_0\|_{L^2(\R^2)}^2 + \int_{\R^2\times \R^2} \frac{|(\chi v_0)(x)-(\chi v_0)(y) |^2}{|x-y|^3}dx\:dy.
$$
We consider a cut-off function $\vartheta$ satisfying \eqref{hyp:chi}, so that
\begin{eqnarray*}
\|\chi v_0\|_{L^2(\R^2)}^2&\leq & \sum_{k\in \Z^2} \| (\tau_k \vartheta ) \chi v_0\|_{L^2}^2\\
&\leq & \|\chi\|_\infty^2 \sum_{\substack{k\in \Z^2,\\|k|\leq CR}}\| (\tau_k \vartheta ) v_0\|_{L^2}^2\\
&\leq & C R^2 \|\chi\|_\infty^2 \sup_k \| (\tau_k \vartheta ) v_0\|_{L^2}^2.
\end{eqnarray*}
Concerning the second term, 
\begin{eqnarray*}
&&|\chi v_0(x)- \chi v_0(y)|^2\\&=&\left(\sum_{k\in \Z^2} \tau_k \vartheta(x) \chi(x) v_0(x) - \tau_k \vartheta(y) \chi (y)v_0(y) \right)^2\\
&=&\sum_{\substack{k,l\in \Z^2,\\|k-l|\leq 3} }\left[\tau_k \vartheta(x) \chi(x) v_0(x) - \tau_k \vartheta(y) \chi (y)v_0(y) \right]\left[\tau_l \vartheta(x) \chi(x) v_0(x) - \tau_l \vartheta(y) \chi (y)v_0(y) \right]\\
&&+\sum_{\substack{k,l\in \Z^2,\\|k-l|> 3}} \left[\tau_k \vartheta(x) \chi(x) v_0(x) - \tau_k \vartheta(y) \chi (y)v_0(y) \right]\left[\tau_l \vartheta(x) \chi(x) v_0(x) - \tau_l \vartheta(y) \chi (y)v_0(y) \right].
\end{eqnarray*}
Notice that according to the assumptions on $\vartheta$, if $|k-l|>3$, then $\tau_k \vartheta(x)\tau_l \vartheta(x)=0$ for all $x\in \R^2$. Moreover, if  $\tau_k(x)\tau_l(y)\neq 0$, then $|x-y|\geq |k-l|-2$.
Notice also that the first sum above contains $O(R^2)$ non zero terms. Therefore, using the Cauchy-Schwartz inequality, we infer that
\begin{eqnarray*}
&& \int_{\R^2\times \R^2} \frac{|(\chi v_0)(x)-(\chi v_0)(y) |^2}{|x-y|^3}dx\:dy\\
&\leq & C R^2 \sup_{k\in \Z^2} \int_{\R^2\times \R^2} \frac{|(\tau_k \vartheta\chi v_0)(x)-(\tau_k \vartheta\chi v_0)(y) |^2}{|x-y|^3}dx\:dy\\
&&+ \sum_{\substack{k,l\in \Z^2,\\|k-l|> 3}} \frac{1}{(|k-l|-2)^3}\int_{\R^2\times \R^2} |\tau_k \vartheta(x) \chi(x) v_0(x)| |\tau_l \vartheta(y) \chi (y)v_0(y)|\:dx\:dy
\end{eqnarray*}
Using \eqref{est:H12-tronque}, the first term is bounded by
$$
C R^2 \|\chi\|_{W^{1,\infty}}^2 \|v_0\|_{H^{1/2}_{uloc}}^2,
$$
while the second is bounded by $C \|v_0\|_{L^2_{uloc}}^2$.

Gathering all the terms, we obtain \eqref{est:u-tronque-H12}. 
This concludes the proof of the Lemma.
\end{proof}

\subsection{Presentation of the new system}

We now come to our main concern in this paper, which is to prove the existence of weak solutions to the linear system of rotating fluids in the bumpy half-space \eqref{SC}.
There are two features which make this problem particularly difficult. Firstly, the fact that the bottom is now bumpy rather than flat prevents us from the use of the Fourier transform in the tangential direction. Secondly, as the domain $\Omega$ is unbounded, it is not possible to rely on Poincar\'e type inequalities. We face this problem using an idea of \cite{DGVNMnoslip}. It consists in defining a problem equivalent to \eqref{SC} yet posed in the bounded channel $\Omega^b$, by the mean of a transparent boundary condition at the interface $\Sigma=\{x_3=0\}$, namely
\be\label{SC-Omb}
\left\{\begin{array}{rll}
-\Delta u + e_3\times u + \na p&=0& \text{in }\Omb,\\
\dv u&=0&\text{in }\Omb,\\
u|_{\Gamma}&=u_0,&\\
-\pa_3 u + p e_3&=\DtoN(u|_{x_3=0})&\text{on }\Sigma.
\end{array}\right.
\ee
In the system above and throughout the rest of the paper, we assume without any loss of generality that $\sup \om <0$, $\inf \om\geq -1$.
Notice that thanks to assumption \eqref{compatibility}, we have
\begin{eqnarray*}
u_{3|x_3=0}(x_h)&=&u_{0,3}(x_h)-\int_{\om(x_h)}^0 \na_h\cdot u_h(x_h,z)\:dz\\
&=&u_{0,3}(x_h)-\na_h\om \cdot u_{0,h}(x_h)\\
&&-\na_h\cdot\int_{\om(x_h)}^0 u_h(x_h,z)\:dz\\
&=&\na_h\cdot \left(U_h(x_h) -\int_{\om(x_h)}^0 u_h(x_h,z)\:dz  \right),
\end{eqnarray*}
so that $u_{3|x_3=0}$ satisfies the assumptions of Proposition \ref{prop:def-DN}.

Let us start by explaining the meaning of \eqref{SC-Omb}:
\begin{definition}
A function $u\in H^1_{uloc}(\Omb)$ is a solution of \eqref{SC-Omb} if it satisfies the bottom boundary condition $u|_{\Gamma}=u_0$ in the trace sense, and if, for all $\varphi\in \mathcal C^\infty_0\left(\overline{\Om_b}\right) $ such that $\na \cdot \varphi=0$ and $\varphi|_{\Gamma}=0$, there holds
$$
\int_{\Omb} (\na u \cdot \na \varphi + e_3 \times u \cdot \varphi)=-\langle \DtoN(u|_{x_3=0}), \varphi|_{x_3=0}\rangle_{\mathcal D',\mathcal D}.
$$
\label{def:sol-SC0}
\end{definition}

\begin{remark}\label{remfcttest}
Notice that if $\varphi\in \mathcal C^\infty_0\left(\overline{\Om_b}\right)$ is such that $\na \cdot \varphi=0$ and $\varphi|_{\Gamma}=0$, then
$$
\varphi_{3|x_3=0}=\na_h\cdot \Phi_h,\text{ where } \Phi_h(x_h):=-\int_{\om(x_h)}^0\varphi_h(x_h,z)dz\in \mathcal C^\infty_0(\R^2).
$$
Therefore $\varphi$ is an admissible test function for Proposition \ref{prop:def-DN}.

\end{remark}

We then have  the following result, which is the Stokes-Coriolis equivalent of \cite[Proposition 9]{DGVNMnoslip}, and which follows easily from Lemma \ref{lem:lien-SC/DN} and Corollary \ref{cor:ex/uni-SC-uloc}:

\begin{proposition}
Let $u_0\in L^2_{uloc}(\R^2)$ satisfying \eqref{compatibility}, and assume that $\om\in W^{1,\infty}(\R^2)$.

\begin{itemize}
\item Let $(u,p)$ be a solution of \eqref{SC} in $\Om$ such that $u\in H^1_{loc}(\Om)$ and
$$
\begin{aligned}
\forall a>0,\quad \sup_{l\in \Z^2}\int_{l+[0,1]^2}\int_{\om(x_h)}^a (| u|^2 + |\na u|^2)<\infty,\\
\sup_{l\in \Z^2}\int_{l+[0,1]^2}\int_1^\infty |\na^q u|^2<\infty,
\end{aligned}
$$
for some $q\in \N$, $q\geq 1$.

Then $u|_{\Omb}$ is a solution of \eqref{SC-Omb}, and for $x_3>0$, $u$ is given by \eqref{def:u:rep-form}, with $v_0:=u|_{x_3=0}\in \bK$.

\item Conversely, let $u^-\in H^1_{uloc}(\Omb)$ be a solution of \eqref{SC-Omb}, and let $v_0:=u^-|_{x_3=0}\in \bK$. Consider the function $u^+\in H^1_{loc}(\R^3_+)$ defined by \eqref{def:u:rep-form}. Then, setting
$$
u(x):=\left\{ \begin{array}{ll}
u^-(x)&\text{ if }\om(x_h)<x_3<0,\\
u^+(x)&\text{ if }x_3>0,
\end{array}\right. 
$$
the function $u\in H^1_{loc}(\Om)$ is such that
$$
\begin{aligned}
\forall a>0,\quad \sup_{l\in \Z^2}\int_{l+[0,1]^2}\int_{\om(x_h)}^a (| u|^2 + |\na u|^2)<\infty,\\
\sup_{l\in \Z^2}\int_{l+[0,1]^2}\int_{1}^\infty |\na^q u|^2<\infty,
\end{aligned}
$$
for some $q\in \N$ sufficiently large, and is a solution of \eqref{SC}.

\end{itemize}

\end{proposition}

As a consequence, we work with the system \eqref{SC-Omb} from now on. In order to have a homogeneous Poincar\'e inequality in $\Omb$, it is convenient to lift the boundary condition on $\Gamma$, so as to work with a homogeneous Dirichlet boundary condition. Therefore, we define $V=\left(V_h,V_3\right)$ by
\begin{equation*}
V_h:=u_{0,h},\quad V_3:=u_{0,3}-\nabla_h\cdot u_{0,h}\left(x_3-\omega(x_h)\right).
\end{equation*}
Notice that $V|_{x_3=0}\in \bK$ thanks to \eqref{compatibility}, and that $V$ is divergence free. By definition, the function
$$
\tu:=u-V\mathbf{1_{x\in \Omb}}
$$
is a solution of
\be\label{SC0}
\left\{\begin{array}{rll}
-\Delta \tu + e_3\times \tu + \na \tp&=f& \text{in }\Omb,\\
\dv \tu&=0& \text{in }\Omb,\\
\tu|_{\Gamma}&=0,&\\
-\pa_3 \tu + \tp e_3&=\DtoN(\tu|_{x_3=0^-})+ F,&\text{on }\Sigma\times\{0\}
\end{array}\right.
\ee
where
$$
\begin{aligned}
f&:=\Delta V - e_3\times V=\Delta_h V - e_3\times V,\\
F&:=\DtoN(V|_{x_3=0}) + \pa_3 V|_{x_3=0}.
\end{aligned}
$$
Notice that thanks to the regularity assumptions on $u_0$ and $\om$, we have, for all $l\in \N$ and for all $\varphi\in \mathcal C^\infty_0(\overline{\Omb})^3$ with $\Supp \varphi \subset ((-l,l)^2\times (-1,0))\cap \overline{\Omb}$, 
\begin{equation}\label{estf}
\left|\langle f,\varphi\rangle_{\mathcal D',\mathcal D}\right|\leq Cl(\left\|u_{0,h}\right\|_{H^2_{uloc}}+ \left\|u_{0,3}\right\|_{H^1_{uloc}})\left\|\varphi\right\|_{H^1\left(\Omega^b\right)}.
\end{equation}
where the constant $C$ depends only on $\|\om\|_{W^{1,\infty}}$. In a similar fashion, if $\varphi\in \mathcal C^\infty_0(\R^2)^3$ is such that $\varphi_3=\na_h\cdot \Phi_h$ for some $\Phi_h\in \mathcal C^\infty_0(\R^2)^2$, and if $\Supp \varphi, \Supp \Phi_h\subset B(x_0,l)$,
then according to Lemma \ref{lem:est-DN},
\be\label{estF}
\left|\langle F,\varphi\rangle_{\mathcal D',\mathcal D}\right| \leq Cl(\left\|u_{0,h}\right\|_{H^2_{uloc}}+ \left\|u_{0,3}\right\|_{H^1_{uloc}} + \|U_h\|_{H^{1/2}_{uloc}})\left(\|\varphi\|_{H^{1/2}(\R^2)}+\|\Phi_h\|_{H^{1/2}(\R^2)}\right).
\ee

\subsection{Strategy of the proof}
\label{ssec:strategy}

From now on, we drop the\ $\widetilde{~}$\ in \eqref{SC0} so as to lighten the notation.

$\bullet$ In order to prove the existence of solutions of \eqref{SC0} in $H^1_{uloc}(\Om)$, we truncate horizontally the domain $\Om$, and we derive uniform estimates on the solutions of the Stoke-Coriolis system in the truncated domains. More precisely, we introduce, for all $n\in \N$, $k\in \N$,
$$
\begin{aligned}
\Om_n&:=\Omb\cap\{x\in\R^3,\ |x_1|\leq n,\ x_2\leq n\},\\
\Om_{k,k+1}&:=\Om_{k+1}\setminus \Om_k,\\
\Sigma_n&:=\{(x_h,0)\in \R^3,\ |x_1|\leq n,\ x_2\leq n\},\\
\Sigma_{k,k+1}&:=\Sigma_{k+1}\setminus \Sigma_k,\\
\Gamma_n&:=\Gamma\cap\{x\in\R^3,\ |x_1|\leq n,\ x_2\leq n\}.
\end{aligned}
$$
We consider the Stokes-Coriolis system in $\Om_n$, with homogeneous boundary conditions on the lateral boundaries
\begin{equation}\label{SCblapp}
\left\{\begin{array}{rll}
-\Delta u_n+e_3\times u_n+\nabla p_n&=f,&x\in\Omega_n\\
\nabla\cdot u_n&=0,&x\in\Omega_n\\
u_n&=0,&x\in\Omb\setminus\Om_n\\
u_n&=0,&x\in \Gamma_n\\
-\partial_3u_n+p_ne_3\arrowvert_{x_3=0}&=\DtoN\left(u_n\arrowvert_{x_3=0}\right)+F,&x\in\Sigma_n.
\end{array}\right. 
\end{equation}
Notice that the transparent boundary condition involving the Dirichlet to Neumann operator only makes sense if $u_n|_{x_3=0}$ is defined on the whole plane $\Sigma$ (and not merely on $\Sigma_n$), due to the non-locality of the operator $\DtoN$. This accounts for the condition $u_n|_{\Omb\setminus \Om_n}=0$.

Taking $u_n$ as a test function  in \eqref{SCblapp}, we get a first energy estimate on $u_n$
\begin{eqnarray}
\nonumber&&\left\|\nabla u_n\right\|_{L^2(\Omega^b)}^2\\
\label{est:un}&=&\underbrace{-\left\langle \DtoN\left(u_n\arrowvert_{x_3=0}\right),u_n\arrowvert_{x_3=0}\right\rangle}_{\leq 0}-\left\langle F,u_n\arrowvert_{x_3=0}\right\rangle+\left\langle f,u_n\right\rangle\\
\nonumber&\leq& Cn\left(\left\|u_{n,h}\arrowvert_{x_3=0}\right\|_{H^{1/2}\left(\Sigma_n\right)}+\left\|\int_{\omega(x_h)}^0u_{n,h}(x_h,z')dz'\right\|_{H^{1/2}\left(\Sigma_n\right)}\right)+Cn\left\|u_n\right\|_{H^1\left(\Omega_n\right)}\\
\nonumber&\leq& Cn\left\|u_n\right\|_{H^1\left(\Omega_n\right)},
\end{eqnarray}
where the constant $C$ depends only on $\|u_0\|_{H^2_{uloc}}$ and $\|\om\|_{W^{1,\infty}}$. This implies, thanks to the Poincar\'e inequality,
\begin{equation}\label{estE_n}
E_n:=\int_{\Omega}\nabla u_n\cdot\nabla u_n\leq C_0n^2.
\end{equation}
The existence of $u_n$ in $H^1(\Omega^b)$ follows. Uniqueness is a consequence of equality \eqref{est:un} with $F=0$ and $f=0$.

In order to prove the existence of $u$, we will derive $H^1_{uloc}$ estimates on $u_n$, uniform with respect to $n$. Then, passing to the limit in \eqref{SCblapp} and in the estimates, we deduce the exi\-stence of a solution of \eqref{SC0} in $H^1_{uloc}(\Omb)$. In order to obtain $H^1_{uloc}$ estimates on $u_n$, we follow the strategy of G\'erard-Varet and Masmoudi in \cite{DGVNMnoslip}, which is inspired from the work of Ladyzhenskaya and Solonnikov \cite{LS}. We work with the energies
\begin{equation}\label{defE_k}
E_k:=\int_{\Omega_k}\nabla u_n\cdot\nabla u_n.
\end{equation}
The goal is to prove an inequality of the type
\be\label{est:E_k-simplifiee}
E_k\leq C\left(k^2 + (E_{k+1}-E_k)\right),\quad\forall k\in \{m,\ldots\ n\},
\ee
where $m\in \N$ is a large, but fixed integer (independent of $n$) and $C$ is a constant depending only on $\|\om\|_{W^{1,\infty}}$ and $\|u_{0,h}\|_{H^2_{uloc}}, \|u_{0,3}\|_{H^1_{uloc}}, \|U_h\|_{H^{1/2}_{uloc}}$.
Then, by backwards induction on $k$, we deduce that
$$
E_k\leq C k^2 \quad\forall k\in \{m,\ldots\ n\}
$$
so that $E_m$, in particular, is bounded, uniformly in $n$. Since the  derivation of  the energy estimates is invariant by translation in the horizontal variable, we infer that for all $n\in \N$,
$$
\sup_{c\in \mathcal C_m} \int_{(c\times(-1,0))\cap \Omb} |\na u_n|^2\leq C
$$
where 
\be\label{def:Cm}
\mathcal C_m:=\left\{c,\ \mbox{square of edge of length}\ m\ \mbox{contained in}\ \Sigma_n\ \mbox{with vertices in}\ \mathbb Z^2\right\}.\ee
Hence the uniform $H^1_{uloc}$ bound on $u_n$ is proved. As a consequence, by a diagonal argument, we can extract a subsequence $\left(u_{\psi(n)}\right)_{n\in \N}$ such that $u_{\psi(n)}\rightharpoonup u$ weakly in $H^1(\Om_k)$ and $u_{\psi(n)}|_{x_3=0}\rightharpoonup u|_{x_3=0}$ weakly in $H^{1/2}(\Sigma_k)$ for all $k\in \N$. Of course, $u$ is a solution of the Stokes-Coriolis system in $\Omb$, and $u\in H^1_{uloc}(\Omb)$. Looking closely at the representation formula in Proposition \ref{prop:def-DN}, we infer that
$$
\langle \DtoN u_{\psi(n)}|_{x_3=0}, \varphi \rangle_{\mathcal D',\mathcal D}\stackrel{n\rightarrow\infty}{\longrightarrow}\langle \DtoN u|_{x_3=0}, \varphi \rangle_{\mathcal D',\mathcal D} 
$$
for all admissible test functions $\varphi$. For instance, 
\begin{eqnarray*}
&&\int_{\R^2}\varphi M^{rem}_{HF}\ast (1-\chi )\left({u_{\psi(n)}}|_{x_3=0}-u|_{x_3=0}\right)\\
&=&\int_{\R^2} dx \int_{|t|\leq k} dt \:\varphi(x) M^{rem}_{HF}(x-t)(1-\chi )\left({u_{\psi(n)}}|_{x_3=0}-u|_{x_3=0}\right)(t)\\
&+& \int_{\R^2} dx \int_{|t|\geq k} dt \:\varphi(x) M^{rem}_{HF}(x-t)(1-\chi )\left({u_{\psi(n)}}|_{x_3=0}-u|_{x_3=0}\right)(t).
\end{eqnarray*}
For all $k$, the first integral vanishes as $n\to \infty$ as a consequence of the weak convergence in $L^2(\Sigma_k)$. 
As for the second integral, let $R>0$ such that $\Supp \varphi \subset B_R$, and let $k\geq R+1$. Then
\begin{eqnarray*}
&& \int_{\R^2} dx \int_{|t|\geq k} dt \varphi(x) M^{rem}_{HF}(x-t)((1-\chi )\left({u_{\psi(n)}}|_{x_3=0}-u|_{x_3=0}\right)(t)\\
&&C\int_{\R^2} dx \int_{|t|\geq k} dt |\varphi(x) |\frac{1}{|x-t|^3}\left(\bigl|{u_{\psi(n)}}|_{x_3=0}(t)\bigr|+\bigl|u|_{x_3=0}(t)\bigr|\right)\\
&\leq & C\int_{\R^2} dx |\varphi(x) |\left(\int_{|t|\geq k}\frac{1}{|x-t|^3}dt\right)^{1/2}\left(\int_{|x-t|\geq 1} \frac{dt}{|x-t|^3}(\left|u|_{x_3=0}\right|^2 + \left|u_{\psi(n)}|_{x_3=0}\right|^2)\right)^{1/2}\\
&\leq & C\left(\bigl\|u|_{x_3=0}\bigr\|_{L^2_{uloc}} + \sup_n\bigl\|{u_{n}}|_{x_3=0}\bigr\|_{L^2_{uloc}}\right)\int_{\R^2} dx |\varphi(x) |\left(\int_{|t|\geq k}\frac{1}{|x-t|^3}dt\right)^{1/2}\\
&\leq & C\left(\bigl\|u|_{x_3=0}\bigr\|_{L^2_{uloc}} + \sup_n\bigl\|{u_{n}}|_{x_3=0}\bigr\|_{L^2_{uloc}}\right)\|\varphi\|_{L^1}(k-R)^{-1/2}.
\end{eqnarray*}
Hence the second integral vanishes as $k\to \infty$ uniformly in $n$. We infer that 
$$
\lim_{n\to \infty}\int_{\R^2}\varphi M^{rem}_{HF}\ast ((1-\chi )(u_{\psi(n)}|_{x_3=0}-u|_{x_3=0})=0.
$$
Therefore $u$ is a solution of \eqref{SC0}.

The final induction inequality we will be much more complicated than \eqref{est:E_k-simplifiee}, and the proof will also be  more involved than the one of \cite{DGVNMnoslip}. However, the general scheme will be very close to the one described above.

\vskip2mm

$\bullet$ Concerning uniqueness of solutions of \eqref{SC0}, we use the same type of energy estimates as above. Once again, we give in the present paragraph a very rough idea of the computations, and we refer to section \ref{sec:uniqueness} for all details. When $f=0$ and $F=0$, the energy estimates \eqref{est:E_k-simplifiee} become
$$
E_k\leq C (E_{k+1}-E_k),
$$
and therefore
$$
E_k\leq r E_{k+1}
$$
with $r:=C/(1+C)\in (0,1)$. Hence, by induction,
$$
E_1\leq  r^{k-1} E_k\leq Cr^{k-1} k^2
$$
for all $k\geq 1$, since $u$ is assumed to be bounded in $H^1_{uloc}(\Omb)$. Letting $k\to \infty$, we deduce that $E_1=0$. Since all estimates are invariant by translation in $x_h$, we obtain that $u=0$.

\section{Estimates in the rough channel}
\label{sec:existence}

This section is devoted to the proof of energy estimates of the type \eqref{est:E_k-simplifiee} for solutions of the system \eqref{SCblapp}, which eventually lead to the existence of a solution of \eqref{SC0}.

The goal  is to prove that for some $m\geq 1$ sufficiently large (but independent of $n$), $E_m$ is bounded uniformly in $n$, which automatically implies the boundedness of $u_n$ in $H^1_{uloc}\left(\Omb\right)$. We reach this objective in two steps:
\begin{itemize}
\item We prove a Saint-Venant estimate: we claim that there exists a constant $C_1>0$ uniform in $n$ such that for all $m\in\mathbb N\setminus\{0\}$, for all $k\in\mathbb N$, $k\geq m$,
\begin{equation}\label{estStVenant}
E_k\leq C_1\left[k^2+E_{k+m+1}-E_k+\frac{k^4}{m^5}\sup_{j\geq m+k}\frac{E_{j+m}-E_j}{j}\right].
\end{equation}
The crucial fact is that $C_1$ depends only on $\|\om\|_{W^{1,\infty}}$ and $\|u_{0,h}\|_{H^2_{uloc}}, \|u_{0,3}\|_{H^1_{uloc}}, \|U_h\|_{H^{1/2}_{uloc}}$, so that it is independent of $n$, $k$ and $m$. 
\item  This estimate allows to deduce the bound in $H^1_{uloc}(\Omega)$ via a non trivial induction argument. 
\end{itemize}

Let us first explain the induction, assuming that \eqref{estStVenant} holds. The proof of \eqref{estStVenant} is postponed to the subsection \ref{secstvenant}.

\subsection{Induction}
\label{ssec:induction}
We aim at deducing from \eqref{estStVenant} that there exists $m\in\mathbb N\setminus\{0\}$, $C>0$ such that for all $n\in\mathbb N$,
\begin{equation}\label{borneunifu_n}
\int_{\Omega_m}\nabla u_n\cdot\nabla u_n\leq C.
\end{equation}
The proof of this uniform bound is divided into two points: 
\begin{itemize}
\item Firstly, we deduce from \eqref{estStVenant}, by downward induction on $k$, that there exist positive constants $C_2,\ C_3, m_0$, depending only on $C_0$ and $C_1$ appearing respectively in \eqref{estE_n} and \eqref{estStVenant}, such that for all $(k,m)$ such that $k\geq C_3m$ and $m\geq m_0$,
\begin{equation}\label{estStVenantrec}
E_k\leq C_2\left[k^2+m^3+\frac{k^4}{m^5}\sup_{j\geq m+k}\frac{E_{j+m}-E_j}{j}\right].
\end{equation}
Let us insist on the fact that $C_2$ and  $C_3$  are independent of $n,k,m$. They will be adjusted in the course of the induction argument (see \eqref{condconst}).
\item Secondly, we notice that \eqref{estStVenantrec} yields the bound we are looking for, choosing $k=\lfloor C_3m\rfloor +1$ and $m$ large enough.
\end{itemize}

$\bullet$ We thus start with the proof of \eqref{estStVenantrec}, assuming that \eqref{estStVenant} holds. 

First, notice that thanks to \eqref{estE_n}, \eqref{estStVenantrec} is true for $k\geq n$ as soon as $C_2\geq C_0$, remembering that $u_n=0$ on $\Omb\setminus \Om_n$. We then assume that \eqref{estStVenantrec} holds for $n,n-1,\ldots\ k+1$, where $k$ is an integer such that $k\geq C_3m$ (further conditions on $C_2,C_3$ will be derived at the end of the induction argument, see \eqref{def:const}).

 We prove \eqref{estStVenantrec} at the rank $k$ by contradiction. Hence, assume that \eqref{estStVenantrec} does not hold at the rank $k$, so that
\begin{equation}\label{estStVenantrecnonk}
E_k>C_2\left[k^2+m^3+\frac{k^4}{m^5}\sup_{j\geq m+k}\frac{E_{j+m}-E_j}{j}\right].
\end{equation}
Then, the induction assumption implies
\begin{eqnarray}
\nonumber&&E_{k+m+1}-E_k\\
\nonumber&\leq &C_2\left[(k+m+1)^2-k^2+\frac{(k+m+1)^4-k^4}{m^5}\sup_{j\geq k+m}\frac{E_{j+m}-E_j}{j}\right]\\
\label{induction}&\leq& C_2\left[2k(m+1)+(m+1)^2+ 80\frac{k^3}{m^4}\sup_{j\geq k+m}\frac{E_{j+m}-E_j}{j}\right].
\end{eqnarray}
Above, we have used the following inequality, which holds for all $k\geq m\geq 1$
\begin{eqnarray*}
(k+m+1)^4-k^4&=&4k^3(m+1)+6k^2(m+1)^2+4k(m+1)^3+(m+1)^4\\
&\leq & 8mk^3 + 6 k^2\times 4m^2 + 4k\times 8 m^3  + 16 m^4\\
&\leq & 80 m k^3.
\end{eqnarray*}

Using \eqref{estStVenantrecnonk}, \eqref{estStVenant} and \eqref{induction}, we get
\begin{eqnarray}
\nonumber&&C_2\left[k^2+m^3+\frac{k^4}{m^5}\sup_{j\geq k+m}\frac{E_{j+m}-E_j}{j}\right]\\
\label{estStVenantrecabs}&<&E_k\\
\nonumber&\leq &C_1\left[k^2+2C_2k(m+1)+C_2(m+1)^2+ \left(80C_2\frac{k^3}{m^4} + \frac{k^4}{m^5}\right)\sup_{j\geq k+m}\frac{E_{j+m}-E_j}{j}\right].
\end{eqnarray}
The constants $C_0,\ C_1>0$ are fixed and depend only on $\|\om\|_{W^{1,\infty}}$ and $\|u_{0,h}\|_{H^2_{uloc}}$, $\|u_{0,3}\|_{H^1_{uloc}}$, $\|U_h\|_{H^{1/2}_{uloc}}$ (cf. \eqref{estE_n} for the definition of $C_0$). We choose $m_0>1,\ C_2>C_0$ and $C_3\geq 1$ depending only on $C_0$ and $C_1$ so that
\be\label{def:const}
\left\{
\begin{array}{l}
k\geq C_3 m\\
\text{and }m\geq m_0
\end{array}
\right.\quad\mbox{implies}\quad \left\{  \begin{array}{l}
C_2 (k^2 + m^3)> C_1\left[k^2+2C_2k(m+1)+C_2(m+1)^2\right]\\
\text{and } C_2 \frac{k^4}{m^5}\geq C_1\left(80C_2\frac{k^3}{m^4} + \frac{k^4}{m^5}\right).
\end{array}\right.
\ee
One can easily check that it suffices to choose $C_2,C_3$ and $m_0$ so that
\be\label{condconst}
\begin{aligned}
C_2>\max(2 C_1,C_0),\\
(C_2-C_1)C_3>80 C_1 C_2,\\
 \forall m\geq m_0,\quad (C_2C_1 + C_1 )(m+1)^2<  m^3.
\end{aligned}
\ee
Plugging \eqref{def:const} into \eqref{estStVenantrecabs}, we reach a contradiction. Therefore \eqref{estStVenantrec} is true at the rank $k$. By induction, \eqref{estStVenantrec} is proved for all $m\geq m_0$ and for all $k\geq C_3 m$.

$\bullet$ It follows from \eqref{estStVenantrec}, choosing $k=\lfloor C_3 m\rfloor +1$, that there exists a constant $C>0$, depending only on $C_0,\ C_1,\ C_2,\ C_3$, and therefore only on $\|\om\|_{W^{1,\infty}}$ and on Sobolev-Kato norms on $u_0$ and $U_h$, such that for all $m\geq m_0$,
\begin{equation}\label{estStVenantrec'}
E_{\lfloor m/2\rfloor}\leq E_{\lfloor C_3 m\rfloor +1}\leq C\left[m^3+\frac{1}{m}\sup_{j\geq \lfloor C_3 m\rfloor +m+1}\frac{E_{j+m}-E_j}{j}\right].
\end{equation}
Let us now consider the set $\mathcal C_m$  defined by \eqref{def:Cm} for an even integer $m$.
As $\mathcal C_m$ is finite, there exists a square $c$ in $\mathcal C_m$, which maximizes 
\begin{equation*}
\left\{\|u_n\|_{H^1\left(\Omega_c\right)},c\in\mathcal C_m\right\}
\end{equation*}
where $\Omega_c=\left\{x\in\Omb,x_h\in c\right\}$. We then shift $u_n$ in such a manner that $c$ is centered at $0$. We call $\tilde{u}_n$ the shifted function. It is still compactly supported, yet not in $\Omega_n$ but in $\Omega_{2n}$,
\begin{equation*}
\int_{\Omega_{2n}}\left|\nabla\tilde{u}_n\right|^2=\int_{\Omega_n}\left|\nabla u_n\right|^2\quad\mbox{and}\quad\int_{\Omega_{m/2}}\left|\nabla\tilde{u}_n\right|^2=\int_{\Omega_c}\left|\nabla u_n\right|^2.
\end{equation*}
Analogously to $E_k$, we define $\widetilde{E}_k$. Since the arguments leading to the derivation of energy estimates are invariant by horizontal translation, and all constants depend only on Sobolev norms on $u_0, U_h$ and $\om$, we infer that \eqref{estStVenantrec'} still holds when $E_k$ is replaced by $\widetilde{E}_{k}$. On the other hand, recall that $\widetilde E_{m/2}$ maximizes $\left\|\tilde{u}_n\right\|_{H^1\left(\Omega_c\right)}^2$ on the set of squares of edge length $m$. Moreover, in the set $\Sigma_{j+m}\setminus\Sigma_j$ for $j\geq 1$, there are at most $4(j+m)/m$ squares of edge length $m$. As a consequence, we have, for all $j\in \N^*$,
\begin{equation*}
\widetilde E_{j+m}-\widetilde E_j\leq 4\frac{j+m}{m}\widetilde E_{m/2},
\end{equation*}
so that \eqref{estStVenantrec'} written for $\tilde{u}_n$ becomes 
\begin{align*}
\widetilde E_{m/2}&\leq C\left[m^3+\frac{1}{m^2}\left(\sup_{j\geq (C_3+1)m}1+\frac{m}{j}\right)\widetilde E_{m/2}\right]\\
&\leq C\left[m^3+\frac{1}{m^2}\widetilde E_{m/2}\right].
\end{align*}

This estimate being uniform in $m\in\mathbb N$ provided $m\geq m_0$, we can take $m$ large enough and get
\begin{equation*}
\widetilde E_{m/2}\leq C\frac{m^3}{1-C\frac{1}{m^2}},
\end{equation*}
so that eventually there exists $m\in \N$ such that
$$
\sup_{c\in \mathcal{C}_m}\|u_n\|_{H^1((c\times(-1,0)\cap \Omb))}^2 \leq C\frac{m^3}{1-C\frac{1}{m^2}}.
$$
This means exactly that $u_n$ is uniformly bounded in $H^1_{uloc}(\Omb)$. Existence follows, as explained in paragraph \ref{ssec:strategy}.

\subsection{Saint-Venant estimate}
\label{secstvenant}

This part is devoted to the proof of \eqref{estStVenant}. We carry out a Saint-Venant estimate on the system \eqref{SCblapp}, focusing on having constants uniform in $n$ as explained in the section \ref{ssec:strategy}. The preparatory work of  sections \ref{secesthalfspace} and \ref{subsecDtoN} allows us to focus on very few issues. The main problem is the non-locality of the Dirichlet to Neumann operator, which at first sight does not seem to be compatible with getting estimates independent of the size of the support of $u_n$. 

Let $n\in\mathbb N\setminus\{0\}$ be fixed. Let also $\varphi\in \mathcal C^\infty_0(\Omega^b)$ such that 
\begin{equation}\label{testfctCinfty}
\nabla\cdot\varphi=0,\qquad\varphi=0\mbox{ on }\Omega^b\setminus\Omega_n,\qquad\varphi|_{x_3=\omega(x_h)}=0.
\end{equation}
Remark \ref{remfcttest} states that such a function $\varphi$ is an appropriate test function for \eqref{SCblapp}. In the spirit of Definition \ref{def:sol-SC0}, we are led to the following weak formulation:
\begin{multline}\label{weakSCblapp}
\int_{\Omega^b}\nabla u_n\cdot\nabla\varphi+\int_{\Omega^b}u_{n,h}^\perp\cdot\varphi_h\\
=-\left\langle\DtoN\left(u_n|_{x_3=0^-}\right),\varphi|_{x_3=0^-}\right\rangle_{\mathcal D',\mathcal D}-\left\langle F,\varphi|_{x_3=0^-}\right\rangle_{\mathcal D',\mathcal D}+\left\langle f,\varphi\right\rangle_{\mathcal D',\mathcal D}
\end{multline}
Thanks to the representation formula for $\DtoN$ in Proposition \ref{prop:def-DN}, and to the estimates \eqref{estf} for $f$ and \eqref{estF} for $F$, the weak formulation \eqref{weakSCblapp} still makes sense for $\varphi\in H^1(\Omega^b)$ satisfying \eqref{testfctCinfty}.

\emph{In the sequel we drop the subscripts $n$. Note that all constants appearing in the inequalities below are uniform in $n$. However, one should be aware that $E_k$ defined by \eqref{defE_k} depends on $n$.} Furthermore, we denote $u|_{x_3=0^-}$ by $v_0$.

In order to estimate $E_k$, we introduce a smooth cutoff function $\chi_k=\chi_k(y_h)$ supported in $\Sigma_{k+1}$ and identically equal to $1$ on $\Sigma_k$. We carry out energy estimates on the system \eqref{SCblapp}. Remember that a test function has to meet the conditions \eqref{testfctCinfty}. We therefore choose
\begin{eqnarray*}
\varphi&=&
\left(
\begin{array}{c}
\varphi_h\\
\nabla\cdot\Phi_h
\end{array}
\right):=\left(\begin{array}{c}
\chi_k u_h\\
-\nabla_h\cdot\left(\chi_k\int_{\omega(x_h)}^zu_h(x_h,z')dz'\right)
\end{array}
\right)\in H^1(\Omega^b),\\
&=&\chi_k u - \begin{pmatrix}
0\\ \na_h \chi_k(x_h)\cdot \int_{\omega(x_h)}^zu_h(x_h,z')dz'
\end{pmatrix}
\end{eqnarray*}
which can be readily checked to satisfy \eqref{testfctCinfty}. Notice that this choice of test function is different from the one of \cite{DGVNMnoslip}, which is merely $\chi_k u$. Aside from being a suitable test function for \eqref{SCblapp}, the function $\varphi$ has the advantage of being divergence free, so that there will be no need to estimate commutator terms stemming from the pressure.

Plugging $\varphi$ in the weak formulation \eqref{weakSCblapp}, we get
\begin{eqnarray}
\int_{\Omega}\chi_k\left|\nabla u\right|^2&=&-\int_{\Omega}\nabla u\cdot\left(\nabla\chi_k\right)u+\int_{\Omega}\nabla u_3\cdot\nabla\left(\nabla_h\chi_k(x_h)\cdot\int_{\omega(x_h)}^zu_h(x_h,z')dz'\right)\nonumber\\
&&\label{estenchiku}-\left\langle\DtoN\left(v_0\right),\varphi|_{x_3=0^-}\right\rangle-\left\langle F,\varphi|_{x_3=0^-}\right\rangle+\left\langle f,\varphi\right\rangle.
\end{eqnarray} 
Before coming to the estimates, we state an easy bound on $\Phi_h$ and $\varphi$
\be\label{easybdvarphi_3}
\left\|\Phi_h\right\|_{H^1\left(\Omb\right)}
+ \| \varphi\|_{H^{1}\left(\Omb\right)}+ \|\Phi_h|_{x_3=0}\|_{H^{1/2}(\R^2)} + \|\varphi|_{x_3=0}\|_{H^{1/2}(\R^2)}\leq CE_{k+1}^\frac{1}{2}.
\ee
As we have recourse to Lemma \ref{lem:est-DN} to estimate some terms in \eqref{estenchiku}, we use \eqref{easybdvarphi_3} repeatedly in the sequel, sometimes with slight changes.

We have to estimate each of the terms appearing in \eqref{estenchiku}. The most difficult term is the one involving the Dirichlet to Neumann operator, because of the non-local feature of the latter: although $v_0$ is supported in $\Sigma_n$, $\DtoN(v_0)$ is not in general. However, each term in \eqref{estenchiku}, except $-\left\langle\DtoN \left(v_0\right),\varphi|_{x_3=0^-}\right\rangle$, is local, and hence very easy to bound. Let us sketch the estimates of the local terms. For the first term, we simply use the Cauchy-Schwarz and the Poincar\'e inequalities:
\begin{equation*}
\left|\int_{\Omega}\nabla u\cdot\left(\nabla\chi_k\right)u\right|\leq C\left(\int_{\Omega_{k,k+1}}\left|\nabla u\right|^2\right)^\frac{1}{2}\left(\int_{\Omega_{k,k+1}}\left|u\right|^2\right)^\frac{1}{2}\leq C\left(E_{k+1}-E_k\right).
\end{equation*}
In the same fashion, using \eqref{easybdvarphi_3}, we find that the second term is bounded by
\begin{eqnarray*}
&&\left|\int_{\Omega}\nabla u_3\cdot\nabla\left(\nabla_h\chi_k(x_h)\cdot\int_{\omega(x_h)}^zu_h(x_h,z')dz'\right)dx_hdz\right|\\
&\leq&\int_{\Omega}\left|\nabla u_3\right|\left|\nabla\nabla_h\chi_k(x_h)\right|\int_{\omega(x_h)}^z\left|u_h(x_h,z')\right|dz'dx_hdz\\
&&+\int_{\Omega}\left|\nabla_h u_3\right|\left|\nabla_h\chi_k(x_h)\right|\int_{\omega(x_h)}^z\left|\nabla_h u_h(x_h,z')\right|dz'dx_hdz\\
&&+\int_{\Omega}\left|\partial_3u_3\nabla_h\chi_k(x_h)\cdot u_h(x_h,z)\right|dx_hdz\\
&\leq& C\left(E_{k+1}-E_k\right).
\end{eqnarray*}
We finally bound the two last terms in \eqref{estenchiku} using \eqref{easybdvarphi_3}, and \eqref{estF} or \eqref{estf}:
\begin{align*}
\left|\left\langle F,\varphi|_{x_3=0^-}\right\rangle\right|&\leq C(k+1)\left[\left\|\chi_ku_h|_{x_3=0^-}\right\|_{H^{1/2}\left(\mathbb R^2\right)}+\left\|\nabla_h\cdot\left(\chi_k\int_{\omega(x_h)}^0u_h(x_h,z')dz'\right)\right\|_{H^{1/2}\left(\mathbb R^2\right)}\right]\\
&\leq C(k+1)\left[E_{k+1}^\frac{1}{2}+\left(E_{k+1}-E_k\right)^\frac{1}{2}\right]\leq C (k+1)E_{k+1}^{1/2},\\
\left|\left\langle f,\varphi\right\rangle\right|&\leq (k+1)E_{k+1}^\frac{1}{2}.
\end{align*}

The last term to handle is $-\left\langle\DtoN_h\left(v_0\right),\varphi|_{x_3=0^-}\right\rangle$. The issue of the non-locality of the Dirichlet to Neumann operator is already present for the Stokes system. Again, we attempt to adapt the ideas of \cite{DGVNMnoslip}. So as to handle the large scales of $\DtoN(v_0)$, we are led to introduce the auxiliary parameter $m\in \N^*$, which appears in \eqref{estStVenant}. We decompose $v_0$ into
\begin{align*}
v_0&=\left(\begin{array}{c}
\chi_kv_{0,h}\\
-\nabla_h\cdot\left(\chi_k\int_{\omega(x_h)}^0u_h(x_h,z')dz'\right)
\end{array}\right)+
\left(\begin{array}{c}
\left(\chi_{k+m}-\chi_k\right)v_{0,h}\\
-\nabla_h\cdot\left(\left(\chi_{k+m}-\chi_k\right)\int_{\omega(x_h)}^0u_h(x_h,z')dz'\right)
\end{array}\right)\\
&+\left(\begin{array}{c}
\left(1-\chi_{k+m}\right)v_{0,h}\\
-\nabla_h\cdot\left(\left(1-\chi_{k+m}\right)\int_{\omega(x_h)}^0u_h(x_h,z')dz'\right)
\end{array}\right).
\end{align*}
The truncations on the vertical component of $v_0$ are put inside the horizontal divergence, in order to apply the Dirichlet to Neumann operator to functions in $\mathbb K$.

The term corresponding to the truncation of $v_0$ by $\chi_k$, namely
\begin{eqnarray*}
&&-\left\langle\DtoN\left(\begin{array}{c}
\chi_kv_{0,h}\\
-\nabla_h\cdot\left(\chi_k\int_{\omega(x_h)}^0u_h(x_h,z')dz'\right)
\end{array}\right),\left(\begin{array}{c}
\varphi_h|_{x_3=0^-}\\
\nabla_h\cdot\Phi_h|_{x_3=0^-}
\end{array}\right)\right\rangle\\
&=&-\left\langle\DtoN\left(\begin{array}{c}
\chi_kv_{0,h}\\
-\nabla_h\cdot\left(\chi_k\int_{\omega(x_h)}^0u_h(x_h,z')dz'\right)
\end{array}\right),\left(\begin{array}{c}
\chi_kv_{0,h}\\
-\nabla_h\cdot\left(\chi_k\int_{\omega(x_h)}^0u_h(x_h,z')dz'\right)
\end{array}\right)\right\rangle
\end{eqnarray*}
is negative by positivity of the operator $\DtoN$ (see Lemma \ref{lem:lien-SC/DN}). For the term corresponding to the truncation by $\chi_{k+m}-\chi_k$ we resort to Lemma \ref{lem:est-DN} and \eqref{easybdvarphi_3}. This yields
\begin{eqnarray*}
&&\left|\left\langle\DtoN\left(\begin{array}{c}
\left(\chi_{k+m}-\chi_k\right)v_{0,h}\\
-\nabla_h\cdot\left(\left(\chi_{k+m}-\chi_k\right)\int_{\omega(x_h)}^0u_h(x_h,z')dz'\right)
\end{array}\right),\left(\begin{array}{c}
\varphi_h|_{x_3=0^-}\\
\nabla_h\cdot\Phi_h|_{x_3=0^-}
\end{array}\right)\right\rangle\right|\\
&\leq& C\left(E_{k+m+1}-E_k\right)^\frac{1}{2}E_{k+1}^\frac{1}{2}.
\end{eqnarray*}
However, the estimate of Lemma \ref{lem:est-DN} is not refined enough to address the large scales independently of $n$. For the term
\begin{equation*}
\left\langle\DtoN\left(\begin{array}{c}
\left(1-\chi_{k+m}\right)v_{0,h}\\
-\nabla_h\cdot\left(\left(1-\chi_{k+m}\right)\int_{\omega(x_h)}^0u_h(x_h,z')dz'\right)
\end{array}\right),\left(\begin{array}{c}
\varphi_h|_{x_3=0^-}\\
\nabla_h\cdot\Phi_h|_{x_3=0^-}
\end{array}\right)\right\rangle,
\end{equation*}
we must have a closer look at the representation formula given in Proposition \ref{prop:def-DN}. Let 
\begin{equation*}
\tilde{v}_0:=\left(\begin{array}{c}
\left(1-\chi_{k+m}\right)v_{0,h}\\
-\nabla_h\cdot\left(\left(1-\chi_{k+m}\right)\int_{\omega(x_h)}^0u_h(x_h,z')dz'\right)
\end{array}\right)=
\left(\begin{array}{c}
\left(1-\chi_{k+m}\right)v_{0,h}\\
-\nabla_h\cdot\tilde{V}_h
\end{array}\right).
\end{equation*}
We take $\chi:=\chi_{k+1}$ in the formula of Proposition \ref{prop:def-DN}. If $m\geq 2$, $\Supp \chi_{k+1}\cap \Supp (1-\chi_{k+m})=\emptyset$, so that the formula of Proposition \ref{prop:def-DN} becomes\footnote{Here, we use in a crucial (but hidden) way the fact that the zero order terms at low frequencies are constant. Indeed, such terms are local, so that 
$$
\int_{\R^2} \varphi|_{x_3=0^-}\cdot \bar M \tilde v_0=0.
$$
}
\begin{align*}
\left\langle \DtoN\tilde{v}_0,\varphi\right\rangle&=\int_{\R^2} \varphi|_{x_3=0^-}\cdot K_S\ast \tilde{v}_{0} +\int_{\mathbb R^2}\varphi|_{x_3=0^-}\cdot M^{rem}_{HF}*\tilde{v}_0\\
&+\int_{\mathbb R^2}\varphi_{h|x_3=0^-}\cdot \left\{\mathcal I[M_1]\left(\rho*\tilde{v}_{0,h}\right)+ K^{rem}_1*\tilde{v}_{0,h}\right\} \\
&+\int_{\mathbb R^2}\varphi_{h|x_3=0^-}\cdot \left\{\mathcal I[M_2]\left(\rho*\tilde{V}_h\right)+K^{rem}_2*\tilde{V}_h\right\} \\
&+\int_{\mathbb R^2}\Phi_{h|x_3=0^-}\cdot \left\{\mathcal I[M_3]\left(\rho*\tilde{v}_{0,h}\right)+K^{rem}_3*\tilde{v}_{0,h}\right\} \\
&+\int_{\mathbb R^2}\Phi_{h|x_3=0^-}\cdot \left\{\mathcal I[M_4]\left(\rho*\tilde{V}_h\right)+K^{rem}_4*\tilde{V}_h\right\}.
\end{align*}
Thus, we have two types of terms to estimate:
\begin{itemize}
\item On the one hand are the convolution terms with the kernels $K_S,M^{rem}_{HF}$, and $K^{rem}_i$ for $1\leq i\leq 4$, which all decay like $\frac{1}{|x_h|^3}$. 
\item On the other hand are the terms involving $\mathcal I[M_i]$ for $1\leq i\leq 4$.
\end{itemize}

For the first ones, we rely on the following nontrivial estimate:
\begin{lemma}\label{lemnoyautypeStokes}
For all $k\geq m$,
\begin{equation}\label{estnoyautypeStokes}
\left\|\tilde v_0 \ast \frac{1}{|\cdot|^3}\right\|_{L^2(\Sigma_{k+1})}\leq C\frac{k^\frac{3}{2}}{m^2}\left(\sup_{j\geq k+m}\frac{E_{j+m}-E_j}{j}\right)^\frac{1}{2}.
\end{equation}
This estimate still holds with $\tilde{V}_h$ in place of $\tilde{v}_0$.
\end{lemma}

For the second ones, we have recourse to:
\begin{lemma}\label{lemnoyauordre1}
For all $k\geq m$, for all $1\leq i,\ j\leq 2$,
\begin{equation}\label{estnoyauordre1}
\left\|\mathcal I\left[\frac{\xi_i\xi_j}{|\xi|}\right]\left(\rho*\tilde v_{0,h}\right)\right\|_{L^2(\Sigma_{k+1})}\leq C\frac{k^2}{m^\frac{5}{2}}\left(\sup_{j\geq k+m}\frac{E_{j+m}-E_j}{j}\right)^\frac{1}{2}.
\end{equation}
This estimate still holds with $\tilde{V}_h$ in place of $v_{0,h}$.
\end{lemma}
We postpone the proofs of these two key lemmas to  section \ref{secprooflemmas}. Applying repeatedly  Lemma \ref{lemnoyautypeStokes} and  Lemma \ref{lemnoyauordre1} together with the estimates \eqref{easybdvarphi_3}, we are finally led to the estimate
\begin{multline*}
E_k\leq C\left((k+1)E_{k+1}^\frac{1}{2}+\left(E_{k+1}-E_k\right)+E_{k+1}^\frac{1}{2}\left(E_{k+m+1}-E_k\right)^\frac{1}{2}\right.\\\left.+\frac{k^2}{m^\frac{5}{2}}E_{k+1}^\frac{1}{2}\left(\sup_{j\geq k+m}\frac{E_{j+m}-E_j}{j}\right)^\frac{1}{2}\right),
\end{multline*}
for all $k\geq m\geq 1$. Now, since $E_k$ is increasing in $k$, we have
$$
E_{k+1}\leq E_k + (E_{k+m+1} -E_k).
$$
Using Young's inequality, we infer that for all $\nu>0$, there exists a constant $C_\nu$ such that for all $k\geq 1$,
$$
E_k \leq \nu E_k + C_\nu \left(k^2+ E_{k+m+1}-E_k + \frac{k^4}{m^5}\sup_{j\geq k+m}\frac{E_{j+m}-E_j}{j}\right).
$$
Choosing $\nu<1$,  inequality \eqref{estStVenant} follows.

\subsection{Proof of the key lemmas}
\label{secprooflemmas}

It remains to establish the estimates \eqref{estnoyautypeStokes} and \eqref{estnoyauordre1}. The proofs are quite technical, but similar ideas and tools are used in the two proofs.

\begin{proof}[Proof of Lemma \ref{lemnoyautypeStokes}]
We use an idea of G\'erard-Varet and Masmoudi (see \cite{DGVNMnoslip}) to treat the large scales: we decompose the set $\Sigma\setminus\Sigma_{k+m}$ as
$$
\Sigma\setminus\Sigma_{k+m}=\bigcup_{j=1}^\infty \Sigma_{k+m(j+1)}\setminus \Sigma_{k+mj}.
$$
On every set $\Sigma_{k+m(j+1)}\setminus \Sigma_{k+mj}$, we bound the $L^2$ norm of $\tilde v_0$ by $E_{k+m(j+1)} - E_{k+mj}$. Let us stress here a technical difference with the work of G\'erard-Varet and Masmoudi: since $\Sigma$ has dimension two, the area of the set $\Sigma_{k+m(j+1)}\setminus \Sigma_{k+mj}$ is of order $(k+mj)m$. In particular, we expect $E_{k+m(j+1)} - E_{k+mj}\sim (k+mj)m\|u\|_{H^1_{uloc}}^2$ to grow with $j$. Thus we work with the quantity
$$
\sup_{j\geq k+m}\frac{E_{j+m}-E_j}{j},
$$
which we expect to be bounded uniformly in $n,k$, rather than with $\sup_{j\geq k+m}(E_{j+m}-E_j).$

Now, applying the Cauchy-Schwarz inequality  yields for $\eta>0$
$$\int_{\Sigma_{k+1}}dy \left(\int_{\mathbb R^2}\frac{1}{|y-t|^3}\tilde{v}_0(t)dt\right)^2\\
\leq C\int_{\Sigma_{k+1}}dy\int_{\Sigma\setminus\Sigma_{k+m}}\frac{|t|}{|y-t|^{3+2\eta}}dt\int_{\Sigma\setminus\Sigma_{k+m}}\frac{|\tilde{v}_0(t)|^2}{|t||y-t|^{3-2\eta}}dt.
$$
The role of the division by the $|t|$ factor in the second integral is precisely to force the apparition of the quantities $({E_{j+m}-E_j})/{j}$.
More precisely, for $y\in \Sigma_{k+1}$ and $m\geq 1$,
\begin{align*}
\int_{\Sigma\setminus\Sigma_{k+m}}\frac{|\tilde{v}_0(t)|^2}{|t||y-t|^{3-2\eta}}dt&=\sum_{j=1}^\infty\int_{\Sigma_{k+m(j+1)}\setminus \Sigma_{k+mj}}\frac{|\tilde{v}_0(t)|^2}{|t||y-t|^{3-2\eta}}dt\\
&\leq C \sum_{j=1}^\infty (E_{k+m(j+1)} - E_{k+mj}) \frac{1}{(k+mj) |mj+k-|y|_\infty|^{3-2\eta}}\\
&\leq C\left(\sup_{j\geq k+m}\frac{E_{j+m}-E_j}{j}\right)\sum_{j=1}^\infty\frac{1}{|mj+k-|y|_\infty|^{3-2\eta}}\\
&\leq C_\eta\frac{1}{m}\frac{1}{|m+k-|y|_\infty|^{2-2\eta}}\left(\sup_{j\geq k+m}\frac{E_{j+m}-E_j}{j}\right),
\end{align*}
where $|x|_\infty:=\max(|x_1|, |x_2|)$ for $x\in \R^2$.
A simple rescaling yields
\begin{eqnarray*}
&&\int_{\Sigma_{k+1}}\int_{\Sigma\setminus\Sigma_{k+m}}\frac{|t|}{|y-t|^{3+2\eta}|m+k-|y|_\infty|^{2-2\eta}}\:dt\:dy\\
&=&\int_{\Sigma_{1+\frac{1}{k}}}\int_{\Sigma\setminus\Sigma_{1+\frac{m}{k}}}\frac{|t|}{|y-t|^{3+2\eta}\left|1+\frac{m}{k}-|y|_\infty\right|^{2-2\eta}}\:dt\:dy.
\end{eqnarray*}
Let us assume that $k\geq m\geq 2$ and take $\eta\in\left]\frac{1}{2},1\right[$. We decompose $\Sigma \setminus \Sigma_{1+\frac{m}{k}}$ as $(\Sigma\setminus\Sigma_2) \cup (\Sigma_2 \setminus \Sigma_{1+\frac{m}{k}})$.
On the one hand, since $|t-y|\geq C|t-y|_\infty\geq C( |t|_\infty-|y|_\infty) \geq C(|t|_\infty-3/2)$,
$$
\int_{\Sigma_{1+\frac{1}{k}}}\int_{\Sigma\setminus\Sigma_2}\frac{|t|}{|y-t|^{3+2\eta}\left|1+\frac{m}{k}-|y|_\infty\right|^{2-2\eta}}dtdy\leq C_\eta\int_{\Sigma_{1+\frac{1}{k}}}\frac{dy}{\left|1+\frac{m}{k}-|y|_\infty\right|^{2-2\eta}}.
$$
Decomposing $\Sigma_{1+\frac{1}{k}}$ into elementary  regions of the type $\Sigma_{r+dr}\setminus \Sigma_r$, on which $|y|_\infty\simeq r$, we infer that the right-hand side of the above inequality is bounded by
\begin{align*}
& C\int_0^{1+\frac{1}{k}}\frac{r}{\left|1+\frac{m}{k}-r\right|^{2-2\eta}}dr\leq C\int_0^{1+\frac{1}{k}}\frac{dr}{\left|r+\frac{m-1}{k}\right|^{2-2\eta}}\\
&\leq C_\eta\left(\left(1+\frac{m}{k}\right)^{2\eta-1}-\left(\frac{m-1}{k}\right)^{2\eta-1}\right)\leq C_\eta.
\end{align*}
On the other hand, $y\in\Sigma_{1+\frac{1}{k}}$ implies $\left|1+\frac{m}{k}-|y|_\infty\right|\geq \frac{m-1}{k}$, so
\begin{eqnarray*}
&&\int_{\Sigma_{1+\frac{1}{k}}}\int_{\Sigma_2\setminus\Sigma_{1+\frac{m}{k}}}\frac{|t|}{|y-t|^{3+2\eta}\left|1+\frac{m}{k}-|y|_\infty\right|^{2-2\eta}}\:dt\:dy\\
&\leq&C\left(\frac{k}{m-1}\right)^{2-2\eta}\int_{\Sigma_{1+\frac{1}{k}}}dy\int_{\Sigma_2\setminus\Sigma_{1+\frac{m}{k}}}\frac{dt}{|t-y|^{3+2\eta}}\\
&\leq&C\left(\frac{k}{m-1}\right)^{2-2\eta}\int_{\begin{subarray}{c}X\in\mathbb R^2\\
\frac{m-1}{k}\leq |X|\leq C
\end{subarray}}\frac{dX}{|X|^{3+2\eta}}\leq C_\eta\left(\frac{k}{m}\right)^3.
\end{eqnarray*}
Gathering these bounds leads to \eqref{estnoyautypeStokes}.
\end{proof}

\begin{proof}[Proof of Lemma \ref{lemnoyauordre1}]
As in the preceding proof, the overall strategy is to decompose 
\begin{equation*}
(1-\chi_{k+m})v_{0,h}=\sum_{j=1}^\infty(\chi_{k+m(j+1)}-\chi_{k+mj})v_{0,h}.
\end{equation*}
In the course of the proof, we introduce some auxiliary parameters, whose meaning we explain. We cannot use Lemma \ref{lemauxorder1} as such, because we will need a much finer estimate. 
We therefore rely on the splitting \eqref{splitABC} with $K:=\frac{m}{2}$. An important property is the fact that $\rho:=\mathcal F^{-1}\phi$ belongs to the Schwartz space $\mathcal S\left(\mathbb R^2\right)$ of rapidly decreasing functions.

As in the proof of Lemma \ref{lemauxorder1}, for $K=m/2$ and $x\in \Sigma_{k+1}$, we have
$$
|\Aa(x)|\leq C m \|\na^2 \rho \ast ((1-\chi_{k+m} v_{0,h}))\|_{L^\infty(\Sigma_{k+1+\frac{m}{2}})},
$$
and for all $\alpha>0$, for all $y\in \Sigma_{k+1+\frac{m}{2}}$,
\begin{align*}
\left|\nabla^2\rho*(1-\chi_{k+m})v_{0,h}(y)\right|&\leq \int_{\Sigma\setminus\Sigma_{k+m}}\left|\nabla^2\rho(y-t)\right||v_{0,h}(t)|dt\\
&\leq \left(\int_{\Sigma\setminus\Sigma_{k+m}}\left|\nabla^2\rho(y-t)\right|^2|t|^\alpha dt\right)^{1/2}\left(\int_{\Sigma\setminus\Sigma_{k+m}}\frac{|v_{0,h}(t)|^2}{|t|^\alpha}dt\right)^{1/2}.
\end{align*}
Yet, on the one hand, for $\alpha>2$,
\begin{align*}
\int_{\Sigma\setminus\Sigma_{k+m}}\frac{|v_{0,h}(t)|^2}{|t|^\alpha}dt&=\sum_{j=1}^\infty\int_{\Sigma_{k+m(j+1)}\setminus\Sigma_{k+mj}}\frac{|v_{0,h}(t)|^2}{|t|^\alpha}dt\\
&\leq \left(\sup_{j\geq k+m}\frac{E_{j+m}-E_j}{j}\right)\sum_{j=1}^\infty\frac{1}{(k+mj)^{\alpha-1}}\\
&\leq C\frac{1}{m}\frac{1}{(k+m)^{\alpha-2}}\left(\sup_{j\geq k+m}\frac{E_{j+m}-E_j}{j}\right).
\end{align*}
On the other hand, $y\in\Sigma_{k+1+\frac{m}{2}}$ and $t\in\Sigma\setminus\Sigma_{k+m}$ implies $|y-t|\geq\frac{m}{2}-1$, 
\begin{eqnarray*}
&&\int_{\Sigma\setminus\Sigma_{k+m}}\left|\nabla^2\rho(y-t)\right|^2|t|^\alpha dt\\
&\leq & C \int_{\Sigma\setminus\Sigma_{k+m}}\left|\nabla^2\rho(y-t)\right|^2(|y-t|^\alpha + |y|^\alpha) dt\\
&\leq&C\left(\left(k+1+\frac{m}{2}\right)^\alpha\int_{|s|\geq\frac{m}{2}-1}\left|\nabla^2\rho(s)\right|^2ds+\int_{|s|\geq\frac{m}{2}-1}\left|\nabla^2\rho(s)\right|^2|s|^\alpha ds\right).
\end{eqnarray*}
Now, since $\rho \in \mathcal S(\R^2)$, for all $\beta>0, \alpha>0$ there exists a constant $C_{\alpha,\beta}$ such that
$$
\int_{|s|\geq\frac{m}{2}-1}(1+|s|^\alpha)\left|\nabla^2\rho(s)\right|^2ds\leq C_\beta m^{-2\beta}.
$$
The role of auxiliary parameter $\beta$ is to ``eat'' the powers of $k$ in order to get a Saint-Venant estimate for which the induction procedure of section \ref{ssec:induction} works. Gathering the latter bounds, we obtain for $k\geq m$
\be
\|\Aa\|_{L^\infty(\Sigma_{k+1})}\leq C_\beta k m^{-\beta}\left(\sup_{j\geq k+m}\frac{E_{j+m}-E_j}{j}\right)^{1/2}.\label{estfstintorder1ls}
\ee
The second term in \eqref{splitABC} is even simpler to estimate. One ends up with
\begin{equation}\label{estsndintorder1ls}
\|\Bb\|_{L^\infty(\Sigma_{k+1})}\leq  C_\beta k m^{-\beta}\left(\sup_{j\geq k+m}\frac{E_{j+m}-E_j}{j}\right)^{1/2}.
\end{equation}
Therefore $\Aa$ and $\Bb$ satisfy the desired estimate, since
$$
\|\Aa\|_{L^2(\Sigma_{k+1})} \leq C k \|\Aa\|_{L^\infty(\Sigma_{k+1})},\quad \|\Bb\|_{L^2(\Sigma_{k+1})} \leq C k \|\Bb\|_{L^\infty(\Sigma_{k+1})}.
$$

The last integral in \eqref{splitABC} is more intricate, because it is a convolution integral. Moreover, $\rho*(1-\chi_{k+m})v_{0,h}(y)$ is no longer supported  in $\Sigma\setminus\Sigma_{k+m}$. The idea is to ``exchange'' the variables $y$ and $t$, i.e. to replace the kernel
 $|x-y|^{-3}$ by $|x-t|^{-3}$. Indeed, we have, for all $x,\ y,\ t\in \R^2$,
 \be\label{est:|x-y|^3}
 \left|\frac{1}{|x-y|^3}- \frac{1}{|x-t|^3}\right|\leq \frac{C|y-t|}{|x-y||x-t|^3} +  \frac{C|y-t|}{|x-y|^3|x-t|}.
 \ee
We decompose the integral term accordingly. We obtain, using the fast decay of $\rho$,
\begin{eqnarray*}
&& \int_{|x-y|\geq m/2} dy \frac{1}{|x-y|^3} |\rho\ast ((1-\chi_{k+m})v_{0,h})(y)|\\
&\leq & C \int_{|x-y|\geq m/2} dy \int_{\Sigma \setminus \Sigma_{k+m}} dt \frac{1}{|x-t|^3} |\rho(y-t)| |v_{0,h}(t)|\\
&&+ C  \int_{|x-y|\geq m/2} dy \int_{\Sigma \setminus \Sigma_{k+m}} dt \frac{|y-t|}{|x-y|^3|x-t|} |\rho(y-t)| |v_{0,h}(t)|\\
&&+C\int_{|x-y|\geq m/2} dy \int_{\Sigma \setminus \Sigma_{k+m}} dt \frac{|y-t|}{|x-y||x-t|^3}  |\rho(y-t)| |v_{0,h}(t)|\\
&\leq & C \int_{\Sigma \setminus \Sigma_{k+m}} dt \frac{1}{|x-t|^3}  |v_{0,h}(t)|\\
&&+ C \int_{|x-y|\geq m/2} dy \int_{\Sigma \setminus \Sigma_{k+m}} dt \frac{|y-t|}{|x-y|^3|x-t|} |\rho(y-t)| |v_{0,h}(t)|.
\end{eqnarray*}
The first term in the right hand side above can be addressed thanks to Lemma \ref{lemnoyautypeStokes}. We focus on the second term. As above, we use the Cauchy-Schwarz inequality
\begin{eqnarray*}
&&\int_{\Sigma\setminus\Sigma_{k+m}}\frac{|y-t|\left|\rho(y-t)\right|}{|x-t|}|v_{0,h}(t)|dt\\
&\leq&\sum_{j=1}^\infty\int_{\Sigma_{k+m(j+1)}\setminus\Sigma_{k+mj}}\frac{|y-t|\left|\rho(y-t)\right|}{|x-t|}|v_{0,h}(t)|dt\\
&\leq&\left(\sup_{j\geq k+m}\frac{E_{m+j}-E_j}{j}\right)^\frac{1}{2}\sum_{j=1}^\infty\frac{1}{k+mj-|x|_\infty}\left(\int_{\Sigma_{k+m(j+1)}\setminus\Sigma_{k+mj}}|y-t|^2|\rho(y-t)|^2|t|dt\right)^\frac{1}{2}.
\end{eqnarray*}
The idea is  to use the fast decay of $\rho$ so as to bound the integral over $\Sigma_{k+m(j+1)}\setminus \Sigma_{k+mj}$. However, $\sum_{j=1}^\infty\frac{1}{k+mj-|x|}=\infty$, so that we also need to recover some decay with respect to $j$ in this integral. For $t\in \Sigma_{k+m(j+1)}\setminus \Sigma_{k+mj}$, 
$$
1\leq \frac{|t|-|x|_\infty}{k+mj-|x|_\infty} \leq  \frac{|t|}{k+mj-|x|_\infty} ,
$$
so that for all $\eta>0$,
\begin{eqnarray*}
&&\int_{\Sigma_{k+m(j+1)}\setminus\Sigma_{k+mj}}|y-t|^2|\rho(y-t)|^2|t|dt\\
&\leq & \frac{1}{(k+mj-|x|_\infty)^{2\eta}} \int_{\Sigma_{k+m(j+1)}\setminus\Sigma_{k+mj}}|y-t|^2|\rho(y-t)|^2|t|^{1+2\eta}dt\\
&\leq &  \frac{C}{(k+mj-|x|_\infty)^{2\eta}} \int_{\Sigma_{k+m(j+1)}\setminus\Sigma_{k+mj}}|y-t|^2(|y-t|^{1+2\eta} + |y|^{1+2\eta})|\rho(y-t)|^2dt\\
&\leq & \frac{C_\eta}{(k+mj-|x|_\infty)^{2\eta}}(1+|y-x|^{1+2\eta} + |x|^{1+2\eta})).
\end{eqnarray*}

Summing in $j$, we have as before
$$
\sum_{j=1}^\infty \frac{1}{(k+mj-|x|_\infty)^{1+\eta}}\leq \frac{C_\eta}{m(k+m-|x|_\infty)^\eta}\leq \frac{C_\eta}{m^{1+\eta}}
$$
so that for $0<\eta<\frac{1}{2}$, one finally obtains, for $x\in \Sigma_{k+1}$,
\begin{eqnarray*}
&&\int_{|x-y|\geq\frac{m}{2}}dy\int_{\Sigma\setminus\Sigma_{k+m}}\frac{|y-t|\left|\rho(y-t)\right|}{|x-y|^3|x-t|}|v_{0,h}(t)|dt\\
&\leq&Cm^{-1-\eta}\left(\sup_{j\geq k+m}\frac{E_{m+j}-E_j}{j}\right)^\frac{1}{2}\int_{|x-y|\geq\frac{m}{2}}\left[|x-y|^{-\frac{5}{2}+\eta}+|x|^{\frac{1}{2} + \eta}|x-y|^{-3}\right]dy\\
&\leq&Cm^{-\frac{3}{2}}\left[1+\left(\frac{k}{m}\right)^{\frac{1}{2}+\eta}\right]\left(\sup_{j\geq k+m}\frac{E_{k+j}-E_j}{j}\right)^\frac{1}{2}.
\end{eqnarray*}
Gathering all the terms, and using one again the fact that
$$
\|F\|_{L^2(\Sigma_{k+1})}\leq C k \|F\|_{L^\infty(\Sigma_{k+1})}\quad \forall F \in L^\infty(\Sigma_{k+1}),
$$
we infer that for all $k\geq m$, for all $\eta>0$,
$$
\| \Cc\|_{L^2(\Sigma_{k+1})} \leq C_\eta \frac{k^{\frac{3}{2}+ \eta}}{m^{2+\eta}}\left(\sup_{j\geq k+m}\frac{E_{k+j}-E_j}{j}\right)^\frac{1}{2}.
$$
Choose $\eta=1/2$; Lemma \ref{lemnoyauordre1} is thus proved.
\end{proof}

\section {Uniqueness}

\label{sec:uniqueness}

This section is devoted to the proof of uniqueness of solutions of \eqref{SC0}. Therefore we consider the system \eqref{SC0} with $f=0$ and $F=0$, and we intend to prove that the solution $u$ is identically zero.

Following the notations of the previous section, we set
\begin{equation*}
E_k:=\int_{\Omega_k}\nabla u\cdot\nabla u.
\end{equation*}
We can carry out the same estimates as those of paragraph \ref{secstvenant} and get a constant $C_1>0$ such that for all $m\in\mathbb N$, for all $k\geq m$,
\begin{equation}\label{Stvenantw}
E_k\leq C_1\left(E_{k+m+1}-E_k+\frac{k^4}{m^5}\sup_{j\geq k+m}\frac{E_{j+m}-E_j}{j}\right).
\end{equation}
Let $m$ a positive even integer and $\varepsilon>0$ be fixed. Analogously to paragraph \ref{ssec:induction}, the set $\mathcal C_m$ is defined by
\begin{equation*}
\mathcal C_m:=\left\{c,\ \mbox{square of edge of length}\ m\ \mbox{with vertices in}\ \mathbb Z^2\right\}.
\end{equation*}
Note that the situation is not quite the same as in paragraph \ref{ssec:induction} since this set is infinite. The values of $E_c:=\int_{\Omega_c}\left|\nabla u\right|^2$, when $c\in\mathcal C_m$ are bounded by $Cm^2\left\|u\right\|^2_{H^1_{uloc}(\Omb)}$, so the following supremum exists
\begin{equation*}
\mathcal E_m:=\sup_{c\in\mathcal C_m}E_c<\infty,
\end{equation*}
but may not be attained. Therefore for $\varepsilon>0$, we choose a square $c\in\mathcal C_m$ such that $\mathcal E_m-\varepsilon\leq E_c\leq\mathcal E_m$. As in paragraph \ref{ssec:induction}, up to a shift we can always assume that $c$ is centered in $0$. 

From \eqref{Stvenantw}, we retrieve, for all $m,k\in \N$ with $k\geq m$,
\begin{equation*}
E_k\leq\frac{C_1}{C_1+1}E_{k+m+1}+\frac{C_1}{C_1+1}\frac{k^4}{m^5}\sup_{j\geq k+m}\frac{E_{j+m}-E_j}{j}.
\end{equation*}
Again, the conclusion $E_k=0$ would be very easy to get if there were no second term in the right hand side taking into account the large scales due to the non local operator $\DtoN$. 

An induction argument then implies that for all $r\in\mathbb N$,
\begin{equation}\label{Stvenantrecw}
E_k\leq\left(\frac{C_1}{C_1+1}\right)^rE_{k+r(m+1)}+\sum_{r'=0}^{r-1}\left(\frac{C_1}{C_1+1}\right)^{r'+1}\frac{(k+r'(m+1))^4}{m^5}\sup_{j\geq k+m}\frac{E_{j+m}-E_j}{j}.
\end{equation}
Now, for $\kappa:=\logg\left(\frac{C_1}{C_1+1}\right)<0$ and for $k\in \N$ large enough, the function $x\mapsto \exp(\kappa(x+1))(k+x(m+1))^4$ is decreasing on $(-1,\infty)$, so that
\begin{align*}
\sum_{r'=0}^{r-1}\left(\frac{C_1}{C_1+1}\right)^{r'+1}\frac{(k+r'(m+1))^4}{m^5}&\leq
\sum_{r'=0}^\infty\left(\frac{C_1}{C_1+1}\right)^{r'+1}\frac{(k+r'(m+1))^4}{m^5}\\
&\leq\frac{1}{m^5}\int_{-1}^\infty\exp\left(\kappa(x+1)\right)\left(k+x(m+1)\right)^4dx\\
&\leq C\frac{k^5}{m^6}\int_{-\frac{m+1}{k}}^\infty\exp\left(\frac{\kappa k}{m+1}u\right)\left(1+u\right)^4du\\
&\leq C\frac{k^5}{m^6}
\end{align*}
since $k/(m+1)\geq 1/2$ as soon as $k\geq m\geq 1$.
Therefore, we conclude from \eqref{Stvenantrecw} for $k=m$ that for all $r\in \N$,
\begin{align*}
\mathcal E_m-\varepsilon \leq E_m=E_c&\leq \left(\frac{C_1}{C_1+1}\right)^rE_{m+r(m+1)}+\frac{C}{m}\sup_{j\geq 2m}\frac{E_{j+m}-E_j}{j}\\
&\leq \left(\frac{C_1}{C_1+1}\right)^r(r+1)^2(m+1)^2\|u\|_{H^1_{uloc}}^2+4\frac{C}{m}\sup_{j\geq 2m}\frac{j+m}{jm}\mathcal E_m\\
&\leq \left(\frac{C_1}{C_1+1}\right)^r(r+1)^2(m+1)^2\|u\|_{H^1_{uloc}}^2+\frac{C}{m^2}\mathcal E_m.
\end{align*}
Since the constants are uniform in $m$, we have for $m$ sufficiently large and for all $\varepsilon>0$,
\begin{equation*}
\mathcal E_m\leq C\left[\left(\frac{C_1}{C_1+1}\right)^r(r+1)^2(m+1)^2+\varepsilon\right],
\end{equation*}
which letting $r\rightarrow\infty$ and $\varepsilon\to 0$ gives $\mathcal E_m=0$. The latter holds for all $m$ large enough, and thus we have  $u=0$.


\section*{Acknowledgements}

The authors wish to thank David G\'erard-Varet for his helpful insights on the derivation of energy estimates.

\appendix

\section{Proof of Lemmas \ref{lem:det-M} and \ref{lem:dev-lambda-A}}
\label{appendixexp}

This section is devoted to the proofs of Lemma \ref{lem:det-M}, which gives a formula for the determinant of $M$, and Lemma \ref{lem:dev-lambda-A}, containing the low and high frequency expansions of the main functions we work with, namely $\lambda_k$ and $A_k$. As $A_1,A_2,A_3$ can be expressed in terms of the eigenvalues $\lambda_k$ solution to \eqref{eq:lambda}, it is essential to begin by stating some properties of the latter. Usual properties on the roots of polynomials entail that the eigenvalues satisfy
\begin{equation}\label{rellambda}
\begin{array}{l}
\mathcal{R}(\lambda_k)>0 \text{ for }k=1,2,3,\quad \lambda_1\in\mathbb ]0,\infty[,\quad \lambda_2=\overline{\lambda_3},\\
-\left(\lambda_1\lambda_2\lambda_3\right)^2=-|\xi|^6,\quad \lambda_1\lambda_2\lambda_3=|\xi|^3,\\
\left(|\xi|^2-\lambda_1^2\right)\left(|\xi|^2-\lambda_2^2\right)\left(|\xi|^2-\lambda_3^2\right)=|\xi|^2,\\
\ds\frac{\left(|\xi|^2-\lambda_k^2\right)^2}{\lambda_k}=\frac{\lambda_k}{|\xi|^2-\lambda_k^2}
\end{array}
\end{equation}
and can be computed exactly
\begin{subequations}\label{lambda_k^2expl}
\begin{align}
\lambda_1^2(\xi)&=|\xi|^2+\left(\frac{-|\xi|^2+\left(|\xi|^4+\frac{4}{27}\right)^\frac{1}{2}}{2}\right)^\frac{1}{3}-\left(\frac{|\xi|^2+\left(|\xi|^4+\frac{4}{27}\right)^\frac{1}{2}}{2}\right)^\frac{1}{3},\\
\lambda_2^2(\xi)&=|\xi|^2+j\left(\frac{-|\xi|^2+\left(|\xi|^4+\frac{4}{27}\right)^\frac{1}{2}}{2}\right)^\frac{1}{3}-j^2\left(\frac{|\xi|^2+\left(|\xi|^4+\frac{4}{27}\right)^\frac{1}{2}}{2}\right)^\frac{1}{3},\\
\lambda_3^2(\xi)&=|\xi|^2+j^2\left(\frac{-|\xi|^2+\left(|\xi|^4+\frac{4}{27}\right)^\frac{1}{2}}{2}\right)^\frac{1}{3}-j\left(\frac{|\xi|^2+\left(|\xi|^4+\frac{4}{27}\right)^\frac{1}{2}}{2}\right)^\frac{1}{3}.
\end{align}
\end{subequations}

\subsection{Expansion of the eigenvalues $\lambda_k$}

The expansions below follow directly from the exact formulas \eqref{lambda_k^2expl}. In high frequencies, that is for $|\xi|\gg 1$, we have
\begin{subequations}\label{dvptlambdakHF}
\begin{eqnarray}
\lambda_1^2=|\xi|^2\left(1-|\xi|^{-\frac{4}{3}}+O\left(|\xi|^{-\frac{8}{3}}\right)\right),&\lambda_1=|\xi|-\frac{1}{2}|\xi|^{-\frac{1}{3}}+O\left(|\xi|^{-\frac{5}{3}}\right),\\
\lambda_2^2=|\xi|^2\left(1-j^2|\xi|^{-\frac{4}{3}}+O\left(|\xi|^{-\frac{8}{3}}\right)\right),&\lambda_2=|\xi|-\frac{j^2}{2}|\xi|^{-\frac{1}{3}}+O\left(|\xi|^{-\frac{5}{3}}\right),\\
\lambda_3^2=|\xi|^2\left(1-j|\xi|^{-\frac{4}{3}}+O\left(|\xi|^{-\frac{8}{3}}\right)\right),&\lambda_3=|\xi|-\frac{j}{2}|\xi|^{-\frac{1}{3}}+O\left(|\xi|^{-\frac{5}{3}}\right).
\end{eqnarray}
\end{subequations}
In low frequencies, that is for $|\xi|\ll 1$, we have
\begin{align*}
\left(|\xi|^4+\frac{4}{27}\right)^\frac{1}{2}&=\frac{2}{\sqrt{27}}\left[1+\frac{27}{8}|\xi|^4+O\left(|\xi|^8\right)\right],\\
\left(\frac{-|\xi|^2+\left(|\xi|^4+\frac{4}{27}\right)^\frac{1}{2}}{2}\right)^\frac{1}{3}&=\frac{1}{\sqrt{3}}-\frac{1}{2}|\xi|^2-\frac{\sqrt{3}}{8}|\xi|^4+O\left(|\xi|^6\right),\\
\left(\frac{|\xi|^2+\left(|\xi|^4+\frac{4}{27}\right)^\frac{1}{2}}{2}\right)^\frac{1}{3}&=\frac{1}{\sqrt{3}}+\frac{1}{2}|\xi|^2-\frac{\sqrt{3}}{8}|\xi|^4+O\left(|\xi|^6\right),
\end{align*}
from which we deduce
\begin{subequations}\label{dvptlambdakBF}
\begin{eqnarray}
\lambda_2^2=i+\frac{3}{2}|\xi|^2-\frac{3}{8}i|\xi|^4+ O(|\xi|^6),&\lambda_2=e^{i\frac{\pi}{4}}\left(1-\frac{3}{4}i |\xi|^2+\frac{3}{32}|\xi|^4 + O(|\xi|^6)\right),\\
\lambda_3^2=-i+\frac{3}{2}|\xi|^2+\frac{3}{8}i|\xi|^4+ O(|\xi|^6),&\lambda_3=e^{-i\frac{\pi}{4}}\left(1+\frac{3}{4}i |\xi|^2 + \frac{3}{32}|\xi|^4 + O(|\xi|^6)\right).
\end{eqnarray}
\end{subequations}
Since $\lambda_1\lambda_2\lambda_3=|\xi|^3$, we infer that
$$
\lambda_1=|\xi|^3 + O(|\xi|^7).
$$

\subsection{Expansion of $A_1$, $A_2$ and $A_3$}

Let us recall that $A_k=A_k(\xi)$, $k=1,\ldots\ 3$, solve the linear system
\begin{equation*}
\underbrace{\left(\begin{array}{ccc}
1&1&1\\
\lambda_1&\lambda_2&\lambda_3\\
\frac{\left(|\xi|^2-\lambda_1^2\right)^2}{\lambda_1}&\frac{\left(|\xi|^2-\lambda_2^2\right)^2}{\lambda_2}&\frac{\left(|\xi|^2-\lambda_3^2\right)^2}{\lambda_3}
\end{array}\right)}_{=:M(\xi)}
\left(\begin{array}{c}
A_1\\A_2\\A_3
\end{array}\right)
=\left(\begin{array}{c}
\widehat{v_{0,3}}\\
i\xi\cdot\widehat{v_{0,h}}\\
-i\xi^\perp\cdot\widehat{v_{0,h}}
\end{array}
\right).
\end{equation*}
The exact computation of $A_k$ is not necessary. For the record, note however that $A_k$ can be written in the form of a quotient
\begin{equation}\label{A_kquot}
A_k=\frac{P\left(\xi_1,\xi_2,\lambda_1,\lambda_2,\lambda_3\right)}{Q\left(|\xi|,\lambda_1,\lambda_2,\lambda_3\right)}
\end{equation}
where $P$ is a polynomial with complex coefficients and 
\begin{equation}\label{detM}
Q:=\det(M)=\left(\lambda_1-\lambda_2\right)\left(\lambda_2-\lambda_3\right)\left(\lambda_3-\lambda_1\right)\left(|\xi|+\lambda_1+\lambda_2+\lambda_3\right).
\end{equation}
This formula for $\det(M)$ is shown using the relations \eqref{rellambda}
\begin{align*}
\det(M)&=\frac{\lambda_2^2\left(|\xi|^2-\lambda_3^2\right)^2-\lambda_3^2\left(|\xi|^2-\lambda_2^2\right)^2}{\lambda_2\lambda_3}-\frac{\lambda_1^2\left(|\xi|^2-\lambda_3^2\right)^2-\lambda_3^2\left(|\xi|^2-\lambda_1^2\right)^2}{\lambda_1\lambda_3}\\
&\qquad\qquad\qquad\qquad\qquad+\frac{\lambda_1^2\left(|\xi|^2-\lambda_2^2\right)^2-\lambda_2^2\left(|\xi|^2-\lambda_1^2\right)^2}{\lambda_1\lambda_2}\\
&=|\xi|\left(\lambda_1\left(\lambda_2^2-\lambda_3^2\right)-\lambda_2\left(\lambda_1^2-\lambda_3^2\right)+\lambda_3\left(\lambda_1^2-\lambda_2^2\right)\right)\\
&\qquad\qquad+\lambda_2\lambda_3\left(\lambda_3^2-\lambda_2^2\right)-\lambda_1\lambda_3\left(\lambda_3^2-\lambda_1^2\right)+\lambda_1\lambda_2\left(\lambda_2^2-\lambda_1^2\right)\\
&=\left(\lambda_1-\lambda_2\right)\left(\lambda_2-\lambda_3\right)\left(\lambda_3-\lambda_1\right)\left(|\xi|+\lambda_1+\lambda_2+\lambda_3\right).
\end{align*}
This proves \eqref{detM}, and thus lemma \ref{lem:det-M}.

We now concentrate on the expansions of $M(\xi)$ for $|\xi|\gg 1$ and $|\xi|\ll 1$.
\subsubsection{High frequency expansion}

At high frequencies, it is convenient to work with the quantities $B_1,B_2,B_3$ introduced in \eqref{def:B}. Indeed, inserting the expansions \eqref{dvptlambdakHF} into the system \eqref{def:A} yields
$$
\begin{aligned}
B_1=\widehat{v_{0,3}},\\
|\xi| B_1 - \frac{1}{2}|\xi|^{-1/3} B_2 + O(|\xi|^{-5/3} |A|)= i\xi\cdot\widehat{v_{0,h}},\\
|\xi|^{1/3} B_3 + O(|\xi|^{-1} |A|)= -i\xi^\perp\cdot\widehat{v_{0,h}}.
\end{aligned}
$$
Of course $A$ and $B$ are of the same order, so that the above system becomes
$$
\begin{aligned}
B_1=\hat v_{0,3},\\
B_2=2|\xi|^{1/3}(|\xi| \widehat{v_{0,3}}-i\xi\cdot\widehat{v_{0,h}}) + O(|\xi|^{-4/3}|B|),\\
B_3=-i|\xi|^{-1/3}\xi^\perp\cdot\widehat{v_{0,h}} +O(|\xi|^{-4/3}|B|).
\end{aligned}
$$
We infer immediately that $|B|=O(|\xi|^{4/3}|\widehat{v_0}|)$, and therefore the result of Lemma \ref{lem:dev-lambda-A} follows.

\subsubsection{Low frequency expansion}
 At low frequencies, we invert $M$ thanks to the adjugate matrix formula
\begin{equation*}
M^{-1}(\xi)=\frac{1}{\det(M(\xi))}\left[\cof(M(\xi))\right]^T.
\end{equation*}

We have
$$\frac{\left(|\xi|^2-\lambda_2^2\right)^2}{\lambda_2}=\frac{e^{i\pi}(1+ O(|\xi|^2))}{e^{i\pi/4}(1+ O(|\xi|^2))}= - e^{-i\pi/4} + O(|\xi|^2)
=\overline{\frac{\left(|\xi|^2-\lambda_3^2\right)^2}{\lambda_3}}.
$$
Hence,
\begin{equation*}
M(\xi)=\left(\begin{array}{ccc}
1&1&1\\
O\left(|\xi|^3\right)&e^{i\frac{\pi}{4}}+O\left(|\xi|^2\right)&e^{-i\frac{\pi}{4}}+O\left(|\xi|^2\right)\\
|\xi|+O\left(|\xi|^5\right)&-e^{-i\frac{\pi}{4}}+O\left(|\xi|^2\right)&-e^{i\frac{\pi}{4}}+O\left(|\xi|^2\right)
\end{array}
\right)
\end{equation*}
and
\begin{equation*}
\cof(M)=\left(
\begin{array}{ccc}
-2i&|\xi|e^{-i\frac{\pi}{4}}&-|\xi|e^{i\frac{\pi}{4}}\\
\sqrt{2}i&-e^{i\frac{\pi}{4}}-|\xi|&e^{-i\frac{\pi}{4}}+|\xi|\\
-\sqrt{2}i&-e^{-i\frac{\pi}{4}}&e^{i\frac{\pi}{4}}
\end{array}
\right)+O\left(|\xi|^2\right).
\end{equation*}
We deduce that 
\begin{align*}
&M^{-1}(\xi)=-\frac{1}{2i\left(1+\frac{\sqrt{2}}{2}|\xi|+O\left(|\xi|^2\right)\right)}\left[\cof(M(\xi))\right]^T\\
&=\left(
\begin{array}{ccc}
1-\frac{\sqrt{2}}{2}|\xi|&-\frac{\sqrt{2}}{2}\left[1-\frac{\sqrt{2}}{2}|\xi|\right]&+\frac{\sqrt{2}}{2}\left[1-\frac{\sqrt{2}}{2}|\xi|\right]\\
\frac{e^{i\frac{\pi}{4}}}{2}|\xi|&-\frac{1}{2i}\left[-e^{i\frac{\pi}{4}}-\left(1-\frac{\sqrt{2}}{2}e^{i\frac{\pi}{4}}\right)|\xi|\right]&-\frac{e^{i\frac{\pi}{4}}}{2}\left[1-\frac{\sqrt{2}}{2}|\xi|\right]\\
\frac{e^{-i\frac{\pi}{4}}}{2}|\xi|&-\frac{1}{2i}\left[e^{-i\frac{\pi}{4}}+\left(1-\frac{\sqrt{2}}{2}e^{-i\frac{\pi}{4}}\right)|\xi|\right]&-\frac{e^{-i\frac{\pi}{4}}}{2}\left[1-\frac{\sqrt{2}}{2}|\xi|\right]
\end{array}
\right)+O\left(|\xi|^2\right).
\end{align*}
Finally,
\begin{subequations}\label{dvptBFAk}
\begin{align}
A_1&=\left(1-\frac{\sqrt{2}}{2}|\xi|\right)\widehat{v_{0,3}}-\frac{\sqrt{2}}{2}i\left(\xi+\xi^\perp\right)\cdot\widehat{v_{0,h}}+O\left(|\xi|^2\left|\widehat{v_0}\right|\right),\\
A_2&=\frac{e^{i\frac{\pi}{4}}}{2}|\xi|\widehat{v_{0,3}}+\frac{1}{2}e^{i\frac{\pi}{4}}\xi\cdot\widehat{v_{0,h}}-\frac{1}{2}e^{-i\frac{\pi}{4}}\xi^\perp\cdot\widehat{v_{0,h}}+O\left(|\xi|^2\left|\widehat{v_0}\right|\right),\\
A_3&=\frac{e^{-i\frac{\pi}{4}}}{2}|\xi|\widehat{v_{0,3}}-\frac{1}{2}e^{-i\frac{\pi}{4}}\xi\cdot\widehat{v_{0,h}}+\frac{1}{2}e^{i\frac{\pi}{4}}\xi^\perp\cdot\widehat{v_{0,h}}+O\left(|\xi|^2\left|\widehat{v_0}\right|\right).
\end{align}
\end{subequations}

\subsection{Low frequency expansion for $L_1$, $L_2$ and $L_3$}

For the sake of completeness, we sketch the low frequency expansion of $L_1$ in detail. 
We recall that
$$
L_k (\xi)\widehat{v_0}(\xi)=\begin{pmatrix}
\frac{i}{|\xi|^2} (-\lambda_k \xi + \frac{(|\xi|^2-\lambda_k^2)^2}{\lambda_k}\xi^\bot)\\1
\end{pmatrix}A_k(\xi)
$$
Hence, for $|\xi|\ll 1$,
$$
L_1(\xi)=\begin{pmatrix}
\frac{i}{|\xi|}\xi^\bot + O(|\xi|^2)\\
1
\end{pmatrix}
\begin{pmatrix}
-\frac{i\sqrt{2}}{2}(\xi_1-\xi_2) & -\frac{i\sqrt{2}}{2}(\xi_1+\xi_2) & 1- \frac{\sqrt{2}}{2}|\xi|
\end{pmatrix} + O(|\xi|^2)
$$
which yields \eqref{exprL1}. The calculations for $L_2$ and $L_3$ are completely analogous.

\subsection{The Dirichlet to Neumann operator}

Let us recall the expression of the operator $\DtoN$ in Fourier space:
\begin{align}
\widehat{\DtoN(v^0)}&=\sum_{k=1}^3\left(\begin{array}{c}
\frac{i}{|\xi|^2}\left[\left(|\xi|^2-\lambda_k^2\right)^2\xi^\perp-\lambda_k^2\xi\right]\\
\lambda_k+\frac{|\xi|^2-\lambda_k^2}{\lambda_k}
\end{array}\right)A_k\label{DNexpr1}\\
&=\left(\begin{array}{c}
-i\widehat{v^0_3}(\xi)\xi\\
i\xi\cdot\widehat{v^0_h}(\xi)
\end{array}\right)
+\sum_{k=1}^3\left(\begin{array}{c}
\frac{i}{|\xi|^2}\left[\left(|\xi|^2-\lambda_k^2\right)^2\xi^\perp+\left(|\xi|^2-\lambda_k^2\right)\xi\right]\\
\frac{|\xi|^2-\lambda_k^2}{\lambda_k}
\end{array}\right)A_k.\label{DNexpr2}
\end{align}

\subsubsection{High frequency expansion}

Using the exact formula \eqref{DNexpr2} for $\widehat{\DtoN v_0}$ together with the expansions \eqref{dvptlambdakHF} and \eqref{dev:A-infty}, we get for the high frequencies
\begin{align}\label{dvptDNHF}
&\widehat{\DtoN v_0}=\left(\begin{array}{c}
-i\widehat{v^0_3}(\xi)\xi\\
i\xi\cdot\widehat{v^0_h}(\xi)
\end{array}\right)
+ \begin{pmatrix}
\frac{i}{|\xi|^2}\left((|\xi|^{4/3} B_3 + O(|\xi|^{4/3}|\widehat{v_0}|))\xi^\bot + (|\xi|^{2/3} B_2 + O(|\xi|^{2/3}|\widehat{v_0}|))\xi \right)\\
|\xi|^{-1/3} B_2 + O(|\xi|^{-1/3}|\widehat{v_0}|)
\end{pmatrix}\\
&=\left(\begin{array}{c}
|\xi|\widehat{v^0_h}+\frac{\xi\cdot \widehat{v^0_h}}{|\xi|}\xi+i\widehat{v^0_3}\xi\\
2|\xi|\widehat{v^0_3}-i\xi\cdot\widehat{v^0_h}
\end{array}
\right)
+
O\left(|\xi|^\frac{1}{3}|\widehat{v_0}|\right)\nonumber.
\end{align}

\subsubsection{Low frequency expansion}
\label{sssec:LFDN}
For $|\xi|\ll 1$, using \eqref{DNexpr1}, \eqref{dvptlambdakBF} and \eqref{dvptBFAk} leads to
\begin{subequations}\label{DNBF}
\begin{align}
&\widehat{\DtoN_h v_0}\nonumber\\
&=\frac{i}{2|\xi|^2}\sum_\pm \left(- \xi^\bot \mp i \xi + O(|\xi|^3)\right) \left(e^{\pm i \pi/4} |\xi| \widehat{v_{0,3}} \pm e^{\pm i \pi/4} \xi \cdot \widehat{v_{0,h}} \mp e^{\mp i \pi/4} \xi^\bot \cdot \widehat{v_{0,h}}  + O(|\xi|^2 | \widehat{v_{0}}|)\right)\\
&=\frac{\sqrt{2}i}{2}\frac{\xi-\xi^\bot}{|\xi|}\widehat{v_{0,3}} + \frac{\sqrt{2}}{2}(\widehat{v_{0,h}} + \widehat{v_{0,h}}^\bot ) + O(|\xi| |\widehat{v_{0}}|).\label{DNhBF}
\end{align}
For the vertical component of the operator $\DtoN$, we have in low frequencies
\begin{align}
\widehat{\DtoN_3v_0}&=i\xi\cdot\widehat{v_{0,h}}+\left(\frac{1}{|\xi|}+O\left(|\xi|\right)\right)A_1(\xi)-\left(e^{i\frac{\pi}{4}}+O\left(|\xi|^2\right)\right)A_2(\xi)-\left(e^{-i\frac{\pi}{4}}+O\left(|\xi|^2\right)\right)A_3(\xi)\nonumber\\
&=\frac{\widehat{v_{0,3}}}{|\xi|}-\frac{\sqrt{2}}{2}\widehat{v_{0,3}}-\frac{\sqrt{2}i}{2}\frac{\xi\cdot\widehat{v_{0,h}}+\xi^\perp\cdot\widehat{v_{0,h}}}{|\xi|}+O\left(|\xi|\left|\widehat{v_0}\right|\right).\label{DN3BF}
\end{align}
\end{subequations}

\section{Lemmas for the remainder terms}
\label{appendixrestes}

The goal of this section is to prove that the various remainder terms encountered throughout the paper decay like $|x|^{-3}$. To that end, we introduce the algebra
\be\label{def:E}\begin{aligned}
E:=\Bigl\{f\in \mathcal C([0,\infty),\R), \exists \mathcal A\subset \R \text{ finite},\ \exists r_0>0, \ f(r)=\sum_{\alpha\in \mathcal A} r^\alpha  f_\alpha(r) \ \forall r\in [0,r_0),\Bigr.\\
\Bigl.\text{where } \forall \alpha \in \mathcal A,\ f_\alpha:\R\to \R \text{ is analytic in }B(0,r_0)\Bigr\}.
\end{aligned}
\ee
We then have the following result:
\begin{lemma}
Let $\varphi \in \mathcal S'(\R^2)$.
\begin{itemize}
\item Assume that $\Supp \hat \varphi \subset B(0,1)$, and that $\hat \varphi(\xi)= f(|\xi|)$ for $\xi$ in a neighbourhood of zero, with $f\in E$ and $f(r)=O(r^\alpha)$ for some $\alpha>1$. Then $\varphi \in L^\infty_{loc}(\R^2\setminus\{0\})$ and there exists a constant $C$ such that
$$
|\varphi(x)|\leq \frac{C}{|x|^3}\quad \forall x\in \R^2.
$$

\item Assume that $\Supp \hat \varphi \subset \R^2\setminus B(0,1)$, and that $\hat \varphi(\xi)= f(|\xi|^{-1})$ for $|\xi|>1$, with $f\in E$ and $f(r)=O(r^\alpha)$ for some $\alpha>-1$. Then $\varphi \in L^\infty_{loc}(\R^2\setminus\{0\})$ and there exists a constant $C$ such that
$$
|\varphi(x)|\leq \frac{C}{|x|^3}\quad \forall x\in \R^2.
$$

\end{itemize}

\label{lem:restes}
\end{lemma}

We prove the Lemma in several steps: we first give some properties of the algebra $E$. We then compute the derivatives of order 3 of functions of the type $f(|\xi|)$ and $f(|\xi|^{-1})$. Eventually, we explain the link between the bounds in Fourier space and in the physical space.

\subsection*{Properties of the algebra E}
\begin{lemma}
\begin{itemize}
\item $E$ is stable by differentiation.
\item Let $f\in E$ with $f(r)=\sum_{\alpha \in \mathcal A} r^\alpha f_\alpha(r)$, and let $\alpha_0\in \R$. Assume that 
$$
f(r)=O(r^{\alpha_0})
$$
for $r$ in a neighbourhood of zero. Then
$$
\inf\{\alpha \in \mathcal A, f_\alpha(0)\neq 0\}\geq \alpha_0.
$$

\item Let $f\in E$, and let $\alpha_0\in \R$ such that 
$$
f(r)=O(r^{\alpha_0})
$$
for $r$ in a neighbourhood of zero. Then
$$
f'(r)=O(r^{\alpha_0-1})
$$
for $0<r\ll 1$.

\end{itemize}
\end{lemma}
\begin{proof}
The first point simply follows from the chain rule and the fact that if $f_\alpha$ is analytic in $B(0,r_0)$, then so is $f_\alpha'$. Concerning the second point, notice that we can always choose the set $\mathcal A$  and the functions $f_\alpha$ so that
$$
f(r)=r^{\alpha_1} f_{\alpha_1}(r) + \cdots + r^{\alpha_s} f_{\alpha_s}(r),
$$
where $\alpha_1<\cdots < \alpha_s$ and $f_{\alpha_i}$ is analytic in $B(0,r_0)$ with $f_{\alpha_i}(0)\neq 0$. Therefore
$$
f(r)\sim r^{\alpha_1} f_{\alpha_1}(0)\text{ as }r\to 0, 
$$
so that $r^{\alpha_1} =O(r^{\alpha_0})$. It follows that $\alpha_1\geq \alpha_0$. Using the same expansion, we also obtain
$$
f'(r)=\sum_{i=1}^s\alpha_i r^{\alpha_i-1} f_{\alpha_i}(r) + r^{\alpha_i} f_{\alpha_i}'(r)=O(r^{\alpha_1-1}).
$$
Since $r^{\alpha_1} =O(r^{\alpha_0})$, we infer eventually that $
f'(r)=O(r^{\alpha_0-1})
$.
\end{proof}

\subsection*{Differentiation formulas}

Now, since we wish to apply the preceding Lemma to functions of the type $f(|\xi|)$, or $f(|\xi|^{-1})$, where $f\in E$, we need to have differentiation formulas for such functions. Tedious but easy computations yield, for $\varphi\in \mathcal C^3(\R)$,
\begin{eqnarray*}
\pa_{\xi_i}^3f(|\xi|)&=&\left(3 \frac{\xi_i^3}{|\xi|^5} - 3 \frac{\xi_i}{|\xi|^3}\right) f'(|\xi|)\\
&+& \left(3 \frac{\xi_i}{|\xi|^2} - \frac{\xi_i^3}{|\xi|^4}\right) f''(|\xi|)\\
&+& \frac{\xi_i^3}{|\xi|^3}f^{(3)}(|\xi|)
\end{eqnarray*}
and
\begin{eqnarray*}
\pa_{\xi_i}^3f (|\xi|^{-1})&=&\left(9 \frac{\xi_i}{|\xi|^5} - 11 \frac{\xi_i^3}{|\xi|^7}\right) f'(|\xi|^{-1})\\
&+& \left(3 \frac{\xi_i}{|\xi|^6} - 7 \frac{\xi_i^3}{|\xi|^8}\right) f''(|\xi|^{-1})\\
&+& \frac{\xi_i^3}{|\xi|^9}f^{(3)}(|\xi|^{-1})
\end{eqnarray*}
In particular, if $\varphi:\R^2\to \R$ is such that $\varphi(\xi)= f(|\xi|)$ for $\xi$ in a neighbourhood of zero, where $f\in E$ is such that $f(r)=O(r^\alpha)$ for $r$ close to zero, we infer that
$$
|\pa_{\xi_1}^3 \varphi(\xi)| + |\pa_{\xi_2}^3 \varphi(\xi)| = O(|\xi|^{\alpha-3})
$$
for $|\xi|\ll 1$. In a similar fashion, if  $\varphi(\xi)= f(|\xi|^{-1})$ for $\xi$ in a neighbourhood of zero, where $f\in E$ is such that $f(r)=O(r^\alpha)$ for $r$ close to zero, we infer that
$$
|\pa_{\xi_1}^3 \varphi(\xi)| + |\pa_{\xi_2}^3 \varphi(\xi)| = O\left\{|\xi|^{-4}(|\xi|^{-1})^{-\alpha-1} +|\xi|^{-5}(|\xi|^{-1})^{-\alpha-2}+|\xi|^{-6}(|\xi|^{-1})^{-\alpha-3}\right\}=O(|\xi|^{\alpha-3}).
$$

\subsection*{Moments of order 3 in the physical space}

\begin{lemma}Let $\varphi \in \mathcal S'(\R^2)$ such that
$\pa_{\xi_1}^3\varphi,\pa_{\xi_2}^3\varphi\in L^1\left(\mathbb R^2\right)$.\\
Then
\begin{equation*}
\left|\mathcal F^{-1}\left(\varphi\right)(x_h)\right|\leq \frac{C}{|x_h|^3}\quad \text{in }\mathcal D'(\R^2\setminus\{0\}).
\end{equation*}\label{lemappendixB1}
\end{lemma}
\begin{proof}
The proof follows from the formula
$$
x_h^\alpha \mathcal F^{-1}\left(\varphi\right)=i \mathcal F^{-1}(\nabla_\xi^\alpha\varphi)
$$
for all $\alpha\in\mathbb N^2$ such that $|\alpha|=3$. When $\varphi\in \mathcal S(\R^2)$, the formula is a consequence of standard properties of the Fourier transform. It is then extended to $\varphi\in \mathcal S'(\R^2)$ by duality.
\end{proof}

\begin{remark}
Notice that constants or polynomials of order less that two satisfy the assumptions of the above Lemma. In this case, the inverse Fourier transform is a distribution whose support is $\{0\}$ (Dirac mass or derivative of a Dirac mass). This is of course compatible with the result of Lemma \ref{lemappendixB1}.
\end{remark}

The result of Lemma \ref{lem:restes} then follows easily. There only remains to explain how we can apply it to the functions in the present paper.
To that end, we first notice that for all $k\in \{1,2,3\}$, $\lambda_k$ is a function of $|\xi|$ only, say $\lambda_k=f_k(|\xi|)$. In a similar fashion, 
$$
L_k(\xi)= G_k^0(|\xi|) + \xi_1 G_k^1(|\xi|) + \xi_2 G_k^2(|\xi|).
$$
We then claim  the following result:

\begin{lemma}\label{lem:restes-bis}
\begin{itemize}
\item For all $k\in \{1,2,3\}$, $j\in \{0,1,2\}$, the functions $f_k, G_k^j$, as well as
\be\label{f_kinE}
r\mapsto f_k(r^{-1}),\  r\mapsto G_k^j(r^{-1})
\ee
all belong to $E$.

\item For $\xi$ in a neighbourhood of zero,
$$
\begin{aligned}
M_k^{rem}= P_k(\xi) + \sum_{1\leq i,j,\leq 2} \xi_i \xi_j a_k^{ij}(|\xi|) + \xi \cdot b_k(|\xi|),\\
N_k^{rem}= Q_k(\xi) + \sum_{1\leq i,j,\leq 2} \xi_i \xi_j c_k^{ij}(|\xi|) + \xi \cdot d_k(|\xi|),
\end{aligned}
$$
where $P_k, Q_k$ are polynomials, and $a_k^{ij}, c_k^{i,j}\in E$, $b_k,\ d_k\in E^2$ with $b_k(r),\ d_k(r)=O(r)$ for $r$ close to zero.

\item There exists a function $m\in E$ such that
$$
(M_{SC}-M_S)(\xi)=m(|\xi|^{-1})
$$
for $|\xi|\gg 1$.
\end{itemize}
\end{lemma}
The lemma can be easily proved using the formulas \eqref{lambda_k^2expl} together with the Maclaurin series for functions of the type $x\mapsto (1+x)^s$ for $s\in \R$.

\section{Fourier multipliers supported in low frequencies}
\label{appendixDroniouImbert}

This appendix is concerned with the proof of Lemma \ref{lem:noyau-ordre1}, which is a slight variant of a result by Droniou and Imbert \cite{DroniouImbert} on integral formulas for the fractional laplacian. Notice that this corresponds to the operator $\mathcal I[|\xi|]=\mathcal I\left[\frac{\xi_1^2 + \xi_2^2}{|\xi|}\right]$. We recall that $g\in \mathcal S\left(\mathbb R^2\right)$, $\zeta\in\mathcal C^\infty_0\left(\mathbb R^2\right)$ and $\rho:=\mathcal F^{-1}\zeta\in\mathcal S\left(\mathbb R^2\right)$. Then, for all $x\in\mathbb R^2$,
$$
\mathcal F^{-1}\left(\frac{\xi_i\xi_j}{|\xi|}\zeta(\xi)\hat{g}(\xi)\right)(x)=\mathcal F^{-1}\left(\frac{1}{|\xi|}\right)*\mathcal F^{-1}\left(\xi_i\xi_j\zeta(\xi)\hat{g}(\xi)\right)(x).
$$
As explained in \cite{DroniouImbert}, the function $|\xi|^{-1}$ is locally integrable in $\R^2$ and therefore belongs to $\mathcal S'(\R^2)$. Its inverse Fourier transform is a radially symmetric distribution with homogeneity $-2+1=-1$. Hence there exists a constant $C_I$ such that
$$
\mathcal F^{-1}\left(\frac{1}{|\xi|}\right)= \frac{C_I}{|x|}.
$$
We infer that

\begin{align*}
\mathcal F^{-1}\left(\frac{\xi_i\xi_j}{|\xi|}\zeta(\xi)\hat{g}(\xi)\right)(x)&=\frac{C_I}{|\cdot|}*\pa_{ij}(\rho*g)\\
&=C_I\int_{\mathbb R^2}\frac{1}{|x-y|}\pa_{ij}(\rho*g)(y)dy\\
&=C_I\int_{\mathbb R^2}\frac{1}{|y|}\pa_{ij}(\rho*g)(x+y)dy.
\end{align*}
The idea is to put the derivatives $\pa_{ij}$ on the kernel $\frac{1}{|y|}$ through integrations by parts.
As such it is not possible to realize this idea. Indeed, $y\mapsto\partial_{i}\left(\frac{1}{|y|}\right)\partial_{j}(\rho*g)(x+y)$ is not integrable in the vicinity of $0$. In order to compensate for this lack of integrability, we consider an even function $\theta\in\mathcal  C^\infty_0\left(\mathbb R^2\right)$ such that $0\leq\theta\leq 1$ and $\theta=1$ on $B(0,K)$, and we introduce the auxiliary function
\begin{equation*}
U_x(y):=\rho*g(x+y)-\rho*g(x)-\theta(y)\left(y\cdot\nabla\right)\rho*g(x)
\end{equation*}
which satisfies
\begin{equation}\label{proprU_x(y)}
|U_x(y)|\leq C|y|^2,\qquad |\nabla_yU_x(y)|\leq C|y|,
\end{equation}
for $y$ close to $0$. Then, for all $y\in\mathbb R^2$,
\begin{align*}
\partial_{y_i}\partial_{y_j}U_x&=\partial_{y_i}\partial_{y_j}\rho*g(x+y)-\left(\partial_{y_i}\partial_{y_j}\theta\right)(y\cdot\nabla)\rho*g(x)-\left(\partial_{y_j}\theta\right)\partial_{x_i}\rho*g(x)-\left(\partial_{y_i}\theta\right)\partial_{x_j}\rho*g(x)
\end{align*}
where 
\begin{equation*}
y\mapsto -\left(\partial_{y_i}\partial_{y_j}\theta\right)(y\cdot\nabla)\rho*g(x)-\left(\partial_{y_j}\theta\right)\partial_{x_i}\rho*g(x)-\left(\partial_{y_i}\theta\right)\partial_{x_j}\rho*g(x)
\end{equation*}
is an odd function. Therefore, for all $\varepsilon>0$, 
$$
\int_{\varepsilon<|y|<\varepsilon^{-1}}\frac{1}{|y|}\pa_{ij}(\rho*g)(x+y)dy=\int_{\varepsilon\leq|y|\leq\frac{1}{\varepsilon}}\frac{1}{|y|}\partial_{y_i}\partial_{y_j}U_x(y)dy.
$$
A first integration by parts yields
\begin{eqnarray*}
&&\int_{\varepsilon\leq|y|\leq\frac{1}{\varepsilon}}\frac{1}{|y|}\partial_{y_i}\partial_{y_j}\rho*g(x+y)dy\\
&=&\int_{\varepsilon\leq|y|\leq\frac{1}{\varepsilon}}\frac{1}{|y|}\partial_{y_i}\partial_{y_j}U_x(y)dy\\
&=&\int_{|y|=\varepsilon}\frac{1}{|y|}\partial_{y_j}U_x(y)n_i(y)dy+\int_{|y|=\frac{1}{\varepsilon}}\frac{1}{|y|}\partial_{y_j}U_x(y)n_i(y)dy+\int_{\varepsilon\leq|y|\leq\frac{1}{\varepsilon}}\frac{y_i}{|y|^3}\partial_{y_j}U_x(y)dy.
\end{eqnarray*} 
The first boundary integral vanishes as $\varepsilon\to 0$ because of \eqref{proprU_x(y)}, and the second thanks to the fast decay of $\rho*g\in\mathcal S\left(\mathbb R^2\right)$. Another integration by parts leads to
\begin{eqnarray*}
&&\int_{\varepsilon\leq|y|\leq\frac{1}{\varepsilon}}\frac{y_i}{|y|^3}\partial_{y_j}U_x(y)dy\\
&=&\int_{|y|=\varepsilon}\frac{y_i}{|y|^3}U_x(y) n_j(y)dy+\int_{|y|=\frac{1}{\varepsilon}}\frac{y_i}{|y|^3}U_x(y) n_j(y)dy+\int_{\varepsilon\leq|y|\leq\frac{1}{\varepsilon}}\left(\partial_{y_i}\partial_{y_j}\frac{1}{|y|}\right)U_x(y)dy\\
&&\stackrel{\varepsilon\rightarrow 0}{\longrightarrow}\int_{\mathbb R^2}\left(\partial_{y_i}\partial_{y_j}\frac{1}{|y|}\right)U_x(y)dy,
\end{eqnarray*}
where 
\begin{equation*}
\partial_{y_i}\partial_{y_j}\frac{1}{|y|}=-\frac{\delta_{ij}}{|y|^3}+3\frac{y_iy_j}{|y|^5},\qquad\left|\partial_{y_i}\partial_{y_j}\frac{1}{|y|}\right|\leq\frac{C}{|y|^3},
\end{equation*}
and the boundary terms vanish because of \eqref{proprU_x(y)} and the fast decay of $U_x$. Therefore, for all $x\in\mathbb R^2$,
\begin{eqnarray*}
&&\mathcal F^{-1}\left(\frac{\xi_i\xi_j}{|\xi|}\zeta(\xi)\hat{g}(\xi)\right)(x)=C_I\int_{\mathbb R^2}\left(\partial_{y_i}\partial_{y_j}\frac{1}{|y|}\right)U_x(y)dy\nonumber\\
&=&C_I\int_{\mathbb R^2}\left(\partial_{y_i}\partial_{y_j}\frac{1}{|y|}\right)\left[\rho*g(x+y)-\rho*g(x)-\theta(y)\left(y\cdot\nabla\right)\rho*g(x)\right]dy\nonumber\\
&=&C_I\int_{B(0,K)}\left(\partial_{y_i}\partial_{y_j}\frac{1}{|y|}\right)\left[\rho*g(x+y)-\rho*g(x)-y\cdot\nabla \rho*g(x)\right]dy\nonumber\\
&+&C_I\int_{\mathbb R^2\setminus B(0,K)}\left(\partial_{y_i}\partial_{y_j}\frac{1}{|y|}\right)\left[\rho*g(x+y)-\rho*g(x)\right]dy\nonumber\\
&-&C_I\int_{\mathbb R^2\setminus B(0,K)}\left(\partial_{y_i}\partial_{y_j}\frac{1}{|y|}\right)\theta(y)\left(y\cdot\nabla\right)\rho*g(x)dy.
\end{eqnarray*}
The last integral is zero as $y\mapsto\theta(y)\left(\partial_{y_i}\partial_{y_j}\frac{1}{|y|}\right)y$ is odd. We then perform a last change of variables by setting $y'=x+y$, and we obtain 
\begin{eqnarray*}
&&\mathcal F^{-1}\left(\frac{\xi_i\xi_j}{|\xi|}\zeta(\xi)\hat{g}(\xi)\right)(x)\\&=&-\int_{|x-y'|\leq K}\gamma_{ij}(x-y')\left\{\rho*g(y')-\rho*g(x)-(y'-x)\na \rho*g(x)\right\}dy'\\
&&-\int_{|x-y'|\geq K}\gamma_{ij}(x-y')\left\{\rho*g(y')-\rho*g(x)\right\}dy'.
\end{eqnarray*}
This terminates the proof of Lemma \ref{lem:noyau-ordre1}.


\bibliographystyle{amsplain}
\bibliography{SC-biblio}

\providecommand{\bysame}{\leavevmode\hbox to3em{\hrulefill}\thinspace}
\providecommand{\MR}{\relax\ifhmode\unskip\space\fi MR }
\providecommand{\MRhref}[2]{%
  \href{http://www.ams.org/mathscinet-getitem?mr=#1}{#2}
}
\providecommand{\href}[2]{#2}
\begin{thebibliography}{10}

\bibitem{ABZ}
T.~Alazard, N.~Burq, and C.~Zuily, \emph{Cauchy theory for the gravity water
  waves system with non localized initial data}, Preprint arXiv:1305.0457,
  2013.

\bibitem{BaeJin12}
Hyeong-Ohk Bae and Bum~Ja Jin, \emph{Existence of strong mild solution of the
  {N}avier-{S}tokes equations in the half space with nondecaying initial data},
  J. Korean Math. Soc. \textbf{49} (2012), no.~1, 113--138.

\bibitem{Basson06}
Arnaud Basson, \emph{Homogeneous statistical solutions and local energy
  inequality for 3{D} {N}avier-{S}tokes equations}, Comm. Math. Phys.
  \textbf{266} (2006), no.~1, 17--35.

\bibitem{CK70}
J.~R. Cannon and George~H. Knightly, \emph{A note on the {C}auchy problem for
  the {N}avier-{S}tokes equations}, SIAM J. Appl. Math. \textbf{18} (1970),
  641--644.

\bibitem{CDGG}
J.-Y. Chemin, B.~Desjardins, I.~Gallagher, and E.~Grenier, \emph{Mathematical
  geophysics}, Oxford Lecture Series in Mathematics and its Applications,
  vol.~32, The Clarendon Press Oxford University Press, Oxford, 2006, An
  introduction to rotating fluids and the Navier-Stokes equations.

\bibitem{CDGGEkman}
Jean-Yves Chemin, Beno{\^{\i}}t Desjardins, Isabelle Gallagher, and Emmanuel
  Grenier, \emph{Ekman boundary layers in rotating fluids}, ESAIM Control
  Optim. Calc. Var. \textbf{8} (2002), 441--466 (electronic), A tribute to J.
  L. Lions.

\bibitem{DGVALDnavier}
Anne-Laure Dalibard and David G{\'e}rard-Varet, \emph{Effective boundary
  condition at a rough surface starting from a slip condition}, J. Differential
  Equations \textbf{251} (2011), no.~12, 3450--3487.

\bibitem{DroniouImbert}
J{\'e}r{\^o}me Droniou and Cyril Imbert, \emph{Fractal first-order partial
  differential equations}, Arch. Ration. Mech. Anal. \textbf{182} (2006),
  no.~2, 299--331.

\bibitem{Ekman05}
V.~Ekman, \emph{On the influence of the earth's rotation on ocean currents},
  Ark. Mat. Astr. Fys. \textbf{2} (1905), no.~11.

\bibitem{GalaQGS06}
Sadek Gala, \emph{Quasi-geostrophic equations with initial data in {B}anach
  spaces of local measures}, Electron. J. Differential Equations (2005).

\bibitem{GaldiI}
Giovanni~P. Galdi, \emph{An introduction to the mathematical theory of the
  {N}avier-{S}tokes equations. {V}ol. {I}}, Springer Tracts in Natural
  Philosophy, vol.~38, Springer-Verlag, New York, 1994, Linearized steady
  problems.

\bibitem{GaPl}
Isabelle Gallagher and Fabrice Planchon, \emph{On global infinite energy
  solutions to the {N}avier-{S}tokes equations in two dimensions}, Arch.
  Ration. Mech. Anal. \textbf{161} (2002), no.~4, 307--337.

\bibitem{DGVNMnoslip}
David G{\'e}rard-Varet and Nader Masmoudi, \emph{Relevance of the slip
  condition for fluid flows near an irregular boundary}, Comm. Math. Phys.
  \textbf{295} (2010), no.~1, 99--137.

\bibitem{GMS01}
Y.~Giga, S.~Matsui, and O.~Sawada, \emph{Global existence of two-dimensional
  {N}avier-{S}tokes flow with nondecaying initial velocity}, J. Math. Fluid
  Mech. \textbf{3} (2001), no.~3, 302--315.

\bibitem{Giga++06}
Yoshikazu Giga, Katsuya Inui, Alex Mahalov, and Shin'ya Matsui,
  \emph{Navier-{S}tokes equations in a rotating frame in {$\Bbb R^3$} with
  initial data nondecreasing at infinity}, Hokkaido Math. J. \textbf{35}
  (2006), no.~2, 321--364.

\bibitem{Giga++07}
Yoshikazu Giga, Katsuya Inui, Alex Mahalov, Shin'ya Matsui, and J{\"u}rgen
  Saal, \emph{Rotating {N}avier-{S}tokes equations in {$\Bbb R_+^3$} with
  initial data nondecreasing at infinity: the {E}kman boundary layer problem},
  Arch. Ration. Mech. Anal. \textbf{186} (2007), no.~2, 177--224.

\bibitem{Giga++08}
Yoshikazu Giga, Katsuya Inui, Alex Mahalov, and J{\"u}rgen Saal, \emph{Uniform
  global solvability of the rotating {N}avier-{S}tokes equations for
  nondecaying initial data}, Indiana Univ. Math. J. \textbf{57} (2008), no.~6,
  2775--2791.

\bibitem{GIM99}
Yoshikazu Giga, Katsuya Inui, and Shin'ya Matsui, \emph{On the {C}auchy problem
  for the {N}avier-{S}tokes equations with nondecaying initial data}, Advances
  in fluid dynamics, Quad. Mat., vol.~4, Dept. Math., Seconda Univ. Napoli,
  Caserta, 1999, pp.~27--68.

\bibitem{GiMi}
Yoshikazu Giga and Tetsuro Miyakawa, \emph{Navier-{S}tokes flow in {$\bold
  R^3$} with measures as initial vorticity and {M}orrey spaces}, Comm. Partial
  Differential Equations \textbf{14} (1989), no.~5, 577--618.

\bibitem{Greenspan}
H.P. Greenspan, \emph{The theory of rotating fluids}, Cambridge University
  Press, 1969.

\bibitem{kato}
Tosio Kato, \emph{The {C}auchy problem for quasi-linear symmetric hyperbolic
  systems}, Arch. Rational Mech. Anal. \textbf{58} (1975), no.~3, 181--205.

\bibitem{Kato92}
\bysame, \emph{Strong solutions of the {N}avier-{S}tokes equation in {M}orrey
  spaces}, Bol. Soc. Brasil. Mat. (N.S.) \textbf{22} (1992), no.~2, 127--155.

\bibitem{KoYo11}
Pawe{\l} Konieczny and Tsuyoshi Yoneda, \emph{On dispersive effect of the
  {C}oriolis force for the stationary {N}avier-{S}tokes equations}, J.
  Differential Equations \textbf{250} (2011), no.~10, 3859--3873.

\bibitem{LS}
O.~A. Lady{\v{z}}enskaja and V.~A. Solonnikov, \emph{Determination of solutions
  of boundary value problems for stationary {S}tokes and {N}avier-{S}tokes
  equations having an unbounded {D}irichlet integral}, Zap. Nauchn. Sem.
  Leningrad. Otdel. Mat. Inst. Steklov. (LOMI) \textbf{96} (1980), 117--160,
  308, Boundary value problems of mathematical physics and related questions in
  the theory of functions, 12.

\bibitem{Lemarecent}
P.~G. Lemari{\'e}-Rieusset, \emph{Recent developments in the {N}avier-{S}tokes
  problem}, Chapman \& Hall/CRC Research Notes in Mathematics, vol. 431,
  Chapman \& Hall/CRC, Boca Raton, FL, 2002.

\bibitem{LemRieu1}
Pierre~Gilles Lemari{\'e}-Rieusset, \emph{Solutions faibles d'\'energie infinie
  pour les \'equations de {N}avier-{S}tokes dans {$\bold R^3$}}, C. R. Acad.
  Sci. Paris S\'er. I Math. \textbf{328} (1999), no.~12, 1133--1138.

\bibitem{MaeTera06}
Yasunori Maekawa and Yutaka Terasawa, \emph{The {N}avier-{S}tokes equations
  with initial data in uniformly local {$L^p$} spaces}, Differential Integral
  Equations \textbf{19} (2006), no.~4, 369--400.

\bibitem{Pedlovsky}
Joseph Pedlosky, \emph{Geophysical fluid dynamics}, Springer, 1987.

\bibitem{Sol03nondecay}
V.~A. Solonnikov, \emph{On nonstationary {S}tokes problem and {N}avier-{S}tokes
  problem in a half-space with initial data nondecreasing at infinity}, J.
  Math. Sci. (N. Y.) \textbf{114} (2003), no.~5, 1726--1740, Function theory
  and applications.

\bibitem{Taylor92}
Michael~E. Taylor, \emph{Analysis on {M}orrey spaces and applications to
  {N}avier-{S}tokes and other evolution equations}, Comm. Partial Differential
  Equations \textbf{17} (1992), no.~9-10, 1407--1456.

\bibitem{Yoneda_survey}
T.~Yoneda, \emph{Global solvability of the navier-stokes equations in a
  rotating frame with spatially almost periodic data}, 2009.

\end{thebibliography}

\end{document}